\documentclass[12pt,reqno,twoside]{amsart}
\usepackage{amsthm,amsmath,amssymb,mathrsfs,amsbsy,mathabx}
\usepackage{graphicx,epsfig,wrapfig,sidecap,subfig}
\usepackage[all]{xy}
\usepackage{proof}
\usepackage{xcolor}
\usepackage{hyperref}
\usepackage{bm}
\usepackage[capitalize]{cleveref}
\usepackage{listings}

\usepackage[a4paper, hmargin=2.5cm, vmargin={3.2cm, 3.2cm}]{geometry}

\xyoption{dvips}
\xyoption{color}

\usepackage{color}

\title{Multiscale Substitution Tilings}
\date{}
\author{Yotam Smilansky}
\address{Yotam Smilansky\newline\indent Department of Mathematics, Rutgers University, NJ, USA. \newline\indent {\tt yotam.smilansky@rutgers.edu}}
\author{Yaar Solomon}
\address{Yaar Solomon\newline\indent Department of Mathematics, Ben-Gurion University of the Negev, Israel.\newline\indent {\tt yaars@bgu.ac.il}}

\newcommand{\N}{{\mathbb{N}}}
\newcommand{\Z}{{\mathbb{Z}}}
\newcommand{\Q}{{\mathbb {Q}}}
\newcommand{\R}{{\mathbb{R}}}

\newcommand{\X}{{\mathbb{X}}}

\newcommand{\EE}{\mathcal{E}_0}
\newcommand{\EEE}{|\EE|}

\newcommand{\PP}{\mathcal{P}}

\newcommand{\OO}{\mathcal{O}}
\newcommand{\QQ}{\mathcal{Q}}
\newcommand{\RR}{\mathcal{R}}
\newcommand{\TT}{\mathcal{T}}
\renewcommand{\OO}{\mathcal{O}}

\renewcommand{\SS}{\mathcal{S}}
\newcommand{\CC}{\mathscr{C}}

\newcommand{\PPP}{\mathscr{P}}
\newcommand{\SSS}{\mathscr{S}}

\newcommand{\dist}{\mathbf{dist}}
\newcommand{\supp}{\mathrm{supp}}

\newcommand{\absolute}[1] {\left|{#1}\right|}

\newcommand{\norm}[1]{\left\|{#1}\right\|}

\newcommand{\vol}{\mathrm{vol}}

\newcommand{\tr}{\rm tr}

\newcommand {\ignore}[1]  {}

\theoremstyle{plain}
\newtheorem{thm}{Theorem}[section]
\newtheorem{lem}[thm]{Lemma}
\newtheorem{prop}[thm]{Proposition}
\newtheorem{cor}[thm]{Corollary}

\newtheorem*{Structural}{Structural results}
\newtheorem*{Geometric}{Geometric results}
\newtheorem*{Statistical}{Statistical results}
\newtheorem*{Dynamical}{Dynamical results}

\theoremstyle{definition}
\newtheorem{definition}[thm]{Definition}

\newtheorem{example}[thm]{Example}
\newtheorem{remark}[thm]{Remark}

\numberwithin{equation}{section}
\swapnumbers

\newif\ifdraft\drafttrue
\usepackage{color}

\begin{document}
\begin{abstract}
      We introduce a new general framework for constructing tilings of Euclidean space, which we call multiscale substitution tilings. These tilings are generated by substitution schemes on a finite set of prototiles, in which multiple distinct scaling constants are allowed. This is in contrast to the standard case of the well-studied substitution tilings which includes examples such as the Penrose and the pinwheel tilings. Under an additional irrationality assumption on the scaling constants, our construction defines a new class of tilings and tiling spaces, which are intrinsically different from those that arise in the standard setup. We study various structural, geometric, statistical and dynamical aspects of these new objects and establish a wide variety of properties. Among our main results are explicit density formulas and the unique ergodicity of the associated tiling dynamical systems. 
\end{abstract}

\maketitle
\subsection*{Mathematics Subject Classification (2010)} 52C23, 52C22 (Primary) 37B10, 37C30, 05B45, 05C21, 37A05 (Secondary)
\section{Introduction}\label{sec:introduction}

In the construction of substitution tilings, which are a classical object of study within the field of aperiodic order and mathematical models of quasicrystals, a standing assumption is that the substitution rule that generates the tiling is associated with a single unique scaling constant. More precisely, given a set of initial tiles, also known as prototiles and usually assumed to be finite, the substitution rule describes a tessellation of each prototile by rescaled copies of prototiles, where the applied scaling constant is unique, and hence we refer to it as a fixed scale substitution rule. This implies that a uniform inflation by the reciprocal constant defines a patch of tiles, each of which is a copy of a prototile under some isometry of the space. By repeating this process countably many times, a tiling of the space can be defined, with the property that all of the tiles that appear in the tiling are copies of the original prototiles. The beautiful examples of the Penrose and pinwheel tilings, see \cite{Penrose}  and  \cite{Radin}, among other well-studied examples, can be constructed in this way, with isometry groups of translations and rigid motions, respectively. For a comprehensive discussion and additional examples see  \cite{BaakeGrimm} and references therein.

When considering the construction of substitution tilings, a natural question that may arise concerns the scaling constant: what kind of tilings emerge if the standard assumption of a single scaling constant is relaxed?

\begin{figure}[ht!]
	\includegraphics[scale=0.6]{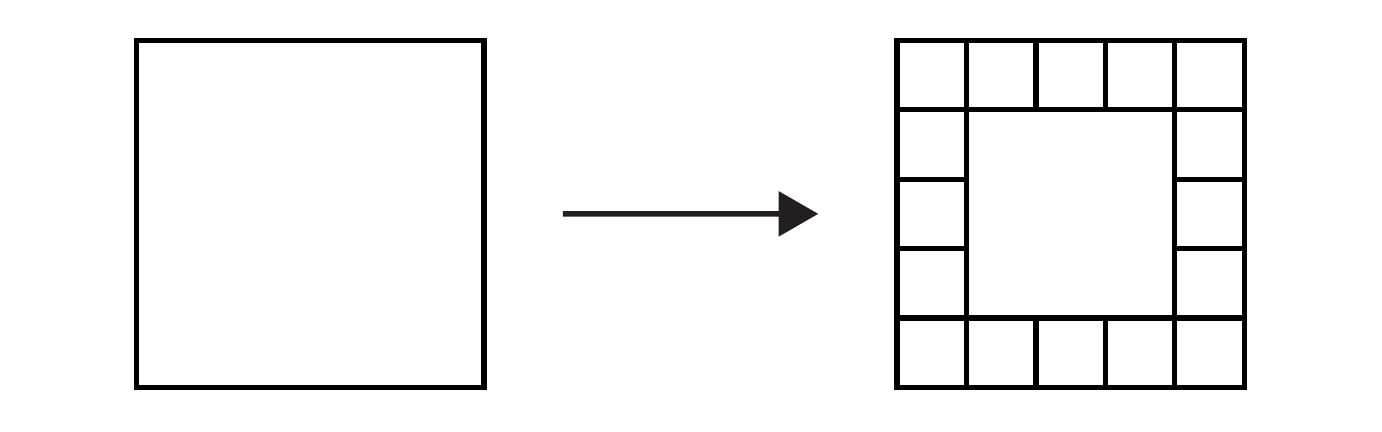}\caption{A multiscale substitution scheme on a unit square $S$, with scaling constants $1/5$ and $3/5$.}	\label{fig: square scheme}
\end{figure}

Clearly, if one is to use multiple distinct scaling constants, as in the example illustrated in Figure \ref{fig: square scheme}, an attempt to define tilings by consecutive substitutions and inflations in the way described above may result in tilings with either arbitrarily small or arbitrarily large tiles. These are of interest in their own right, as evident in \cite{Kenyon3} where substitution rules similar to that of Figure \ref{fig: square scheme} are introduced and studied by Kenyon in the context of self-similar sets. Yet such tilings do not naturally induce Delone sets, which are often used to model physical structures. 

In order to generate tilings with tiles of bounded scales, a different procedure must be used, and several concrete constructions of such tilings were previously considered. Sadun's beautiful generalized pinwheel tilings first appeared in \cite{Sadun - generalized Pinwhell}, where their geometrical and statistical properties were thoroughly studied. This was later extended to include dynamical and topological properties in \cite{Frank-Sadun (ILC fusion)}, in which generalized pinwheel tilings are considered in the context of Frank and Sadun's fusion tilings. A similar analysis was offered for special one dimensional tilings \cite[\S A.5]{Frank-Sadun (ILC fusion)}, and we will return to both of these examples in the following sections. 

We propose here a general procedure to overcome the above mentioned difficulty. In the present introductory section we shall describe and introduce the process in a somewhat intuitive way, a precise and detailed presentation will be given in the next sections. Start with a single tile, and inflate it continuously. When its volume passes a certain threshold, which is set {\it a priori} to be the unit volume, substitute the tile according to the prescribed substitution rule. Continue to inflate, and substitute any tile when its volume reaches the threshold. This defines a continuous family of patches of tiles within some bounded interval of scales, and from which tilings of the entire space can be defined as limiting objects with respect to a suitable topology. For an example of such a patch see Figure \ref{fig: square patch}, and note that it can also be defined by considering an inflation of the original tile, which is then substituted until all tiles are of volume smaller than the threshold. The collection of all tilings generated by a given substitution scheme defines the associated multiscale tiling space, which is a compact space of tilings and is closed under translations.

\begin{figure}[ht!]
	\includegraphics[scale=2.8]{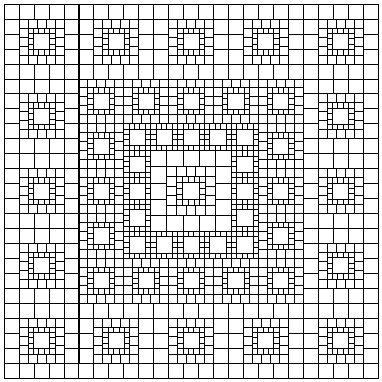}\caption{A patch of a multiscale substitution tiling.}	\label{fig: square patch}
\end{figure}

The aforementioned semi-flow, which we call the substitution semi-flow, is shown to be well defined on the space of tilings. It has periodic orbits which give rise to stationary points in this space, and these stationary tilings can be represented as limits of sequences of nested patches. As we will see, in the fixed scale case, this procedure generates tilings that can be defined by the standard substitution and inflation procedure. 

Our focus in this paper is on tilings generated by incommensurable multiscale substitution schemes on a finite set of prototiles, which are classified by a certain irrationality assumption, as will be precisely defined in \S\ref{sec: incommensurability}. For example, the substitution scheme illustrated in Figure \ref{fig: square scheme} is incommensurable. In this example incommensurability follows from the fact that the two participating scaling constants $1/5$ and $3/5$ are such that $(1/5)^m\neq(3/5)^n$ for any $0\neq m,n\in\Z$, which is true because $3$ and $5$ are co-prime. The substitution rule of the irrational case of the generalized pinwheel tiling is an early example of an incommensurable substitution scheme, as is the aforementioned one dimensional fusion tiling example.
In general, incommensurability amounts to the existence of at least two periodic orbits of the substitution semi-flow for which the logarithms of their lengths are linearly independent over the rationals. As we will see, there are more natural ways to define incommensurability, and incommensurability can be considered a typical property of multiscale substitution schemes. We remark that as in the fixed scale case, for non-incommensurable schemes, which we refer to as commensurable, the generated multiscale tilings can be defined also in the standard way, though perhaps a larger set of prototiles is required. In some aspects, fixed scale, commensurable and incommensurable schemes are analogous to integer, rational and irrational real numbers. In particular, fixed scale schemes are generalized by commensurable schemes, while incommensurability is a complement, disjoint property. 

We present a thorough study of this new class of incommensurable tilings and their properties. Such tilings consist of tiles of finitely many types that appear in infinitely many scales, and are therefore of infinite local complexity. The standard tools of the substitution matrix and the theory of Perron-Frobenius are no longer applicable, and the new tools of the associated directed weighted graph and the recent results of \cite{Yotam graphs} on the distribution of paths on incommensurable graphs are introduced. We show that incommensurability replaces the role of primitivity, and present explicit formulas for the asymptotic density of tiles, as well as for frequencies of tiles of certain types and scales and the volume they occupy. In addition we discuss a form of scale complexity and show that under a mild assumption on the geometry of prototiles, the Delone sets associated with incommensurable tilings are not uniformly spread, in the sense that the associated point sets are never bounded displacement equivalent to lattices. All the relevant terms will be precisely defined in the coming sections. 

Although the construction is well defined and many of our results hold also for the case where tiles are substituted by rescaled copies of prototiles under some isometry, when considering the associated tiling dynamical system, which is defined with respect to the action by translations,  we focus on the case where only translations of rescaled copies are allowed in the substitution rule. In such a case, the tiling dynamical system is shown to be minimal. Minimality, combined with the properties of the substitution semi-flow on the tiling space, allows for the existence of supertiles for any tiling in the tiling space, which form an extremely useful hierarchical structure on tilings. The appropriate variants of uniform patch frequencies, in which patches are counted together with their dilations in some non-trivial interval of scales, are shown to hold. Finally, we show that as in the fixed scale case, the existence of uniform patch frequencies can be used to establish unique ergodicity of the tiling dynamical system.

A detailed introduction of multiscale substitution schemes, the associated graphs and the precise notion of incommensurability, appears in \S\ref{sec:Settings and definitions} and \S\ref{sec: incommensurability}. Two large illustrations of fragments of multiscale substitution tilings are included as an appendix. The main results and the structure of the paper are summarized as follows, where the section number indicates the place in the paper where the result is properly stated and proved.

\begin{Structural} An incommensurable multiscale substitution scheme generates a multiscale tiling space. The substitution semi-flow acts on the multiscale tiling space, and periodic orbits give rise to stationary tilings \S\ref{sec:Construction of multiscale substitution tilings}. Each tiling is equipped with a hierarchical structure of supertiles, though not necessarily in a unique way \S\ref{sec:Dynamics}. 
\end{Structural}

\begin{Geometric}For every tiling in the multiscale tiling space, all tiles are similar to rescaled copies of prototiles \S\ref{sec:Construction of multiscale substitution tilings}. They appear in a dense set of scales within certain intervals of possible scales, and the same holds for any legal patch. Scale complexity is defined for stationary tilings, and the existence of ``Sturmian'' tilings is established \S\ref{sec: scales and complexity}. Appropriate variants of almost repetitivity and almost local indistinguishability are shown to hold \S\ref{sec:Dynamics}. Tilings are not uniformly spread \S\ref{sec: BD BL}.
\end{Geometric}

\begin{Statistical} Explicit asymptotic formulas for the number of tiles of a given type and scales within a given interval, that appear in large supertiles, are given \S\ref{sec: Frequencies}. A variant of uniform patch frequencies is established, where patches are counted together with dilations. For any non-trivial interval of dilations, legal patches have positive patch frequencies \S\ref{sec: patches}. When considered without dilations, all patches have uniform frequency zero.
	
\end{Statistical}

\begin{Dynamical} The multiscale tiling space is equipped with an action by translations to form a tiling dynamical system, which is minimal \S\ref{sec:Dynamics}, and in fact uniquely ergodic \S\ref{sec: unique ergodicity}. 
	
\end{Dynamical}

The pioneering tiling constructions introduced in \cite{Kenyon3, Sadun - generalized Pinwhell} and \cite{Frank-Sadun (ILC fusion)} were shown to satisfy some of the properties listed above using innovative methods. These include in particular the construction of stationary pinwheel tilings, a distinction between rational and irrational pinwheel tilings and tile statistics in the irrational case in \cite{Sadun - generalized Pinwhell}, and the unique ergodicity of the one dimensional tiling system in \cite[\S A.5]{Frank-Sadun (ILC fusion)}, which are recovered here. We note that when considered in the context of multiscale substitution tilings, these constructions have in common that their substitution scheme includes only a single prototile, and tiles appear in only two distinct scales in the substitution rule, which allows for direct computations. The new approach developed in this paper puts these constructions under the umbrella of incommensurable multiscale substitution schemes. The use of the directed weighted graph model and the properties of its flow allow for the study of more complicated constructions with any finite number of prototiles and participating scales. We would also like to mention the tilings studied recently in \cite{BV} and \cite{BBV}, where graph directed iterated function systems are used for generating tilings, but unlike the  work presented here the focus is on commensurable constructions. 

The definition and study of multiscale substitution tilings is part of the recent resurgence of interest in tilings that are not assumed to be of finite local complexity, but still possess rich structure, hierarchies and symmetries.  See among others \cite{Frank1, Frank2, Frettloh Richards, Frank-Robinson, Frank-Sadun1, Frank-Sadun (ILC fusion), Sadun - book} and \cite{Lee-Solomyak}, and the earlier \cite{Danzer, Kenyon1, Kenyon2, Kenyon3} and \cite{Sadun - generalized Pinwhell}. It is our hope that this new class of tilings and the examples it produces will become an object of study in the community of aperiodic order, as there are still many interesting questions to consider, both of geometric and of dynamical flavor.

The construction of incommensurable tilings and the associated multiscale tiling spaces may prove interesting in relation to various other research directions. Incommensurable $\alpha$-Kakutani substitution schemes, studied in \cite{Yotam Kakutani} in the context of uniform distribution of Kakutani sequences of partitions \cite{Kakutani}, generate one dimensional tilings. Every tile is a segment, and by identifying the endpoints of segments with point masses, the unique translation invariant measure on the tiling dynamical space defines a new type of point process on the real line, see \cite{Baake Birkner Moody,Baake-Lenz} for more on spectral theory, diffraction and point processes in the context of aperiodic order. Another example is the relation between the substitution semi-flow and the theory of hyperbolic dynamics and dynamical zeta functions, see \cite{Parry Pollicott book}. The study of uniform distribution and discrepancy of sequences of partitions is also closely related \cite{Yotam Kakutani}, as is the study of self-similar non-lattice fractal strings and sprays, see \cite{Lapidus} and references therein. 

\subsection*{Acknowledgments}
We would like to thank Barak Weiss for his ideas, suggestions and support from the very early stages of this project. We thank Boris Solomyak for many fruitful discussions throughout the years, as well as Pat Hooper, Avner Kiro, Zemer Kosloff, Lorenzo Sadun and Rodrigo Trevi\~{n}o. We thank Lars Blomberg and Uri Grupel for producing helpful numerical results. Finally we thank the anonymous referee for the many valuable comments and suggestions, and in particular for generously contributing the arguments that now appear as Proof 2 of Theorem \ref{thm:not_BD_to_Z^d}.

A large portion of this work was done while the first author was an Orzen Postdoctoral Fellow at the Hebrew University of Jerusalem, and he is grateful for the support of the David and Rosa Orzen Endowment Fund, and the ISF grant No. 1570/17. We are grateful to Dirk Frettl\"oh and to Neil Sloane for including our tilings in the Tiling Encyclopedia \cite{Frettloh-Harriss} and in the On-Line Encyclopedia of Integer Sequences (OEIS) \cite{OEIS}, respectively.

\section{Substitution schemes and the substitution semi-flow}\label{sec:Settings and definitions}

We begin with an introduction to multiscale substitution schemes and their associated graphs, and with detailed definitions of the basic objects required for the definition and study of multiscale substitution tilings, which will be defined in \S\ref{sec:Construction of multiscale substitution tilings}.

\subsection{Tiles and multiscale substitution schemes}
A \textit{tile} $T$ in $\R^d$ is a bounded Lebesgue measurable set of positive measure, which is denoted by $\vol T$ and referred to as the \textit{volume} of $T$, and with a boundary of measure zero. A \textit{tessellation} $\PP$ of a set $U\subset\R^d$ is a collection of tiles with pairwise disjoint interiors so that the union of their support is $U$. For the sake of clarity, a tessellation of a bounded set will be called a \textit{patch} and a tessellation of the entire space will be called a \textit{tiling}. Given a patch $\PP$, a tiling $\TT$ and a subset $B$ of their support, the sub-patch consisting of all tiles in $\PP$ that intersect $B$ is denoted by $[B]^\PP$, and the sub-patch $[B]^\TT$ is similarly defined. Note that we view tiles as embedded subsets of $\R^{d}$, though the location of the origin usually does not matter to us. We will specify the location of the origin when it is important.

\begin{definition}\label{def: substitution scheme}
A \textit{multiscale substitution scheme} $\sigma=(\tau_\sigma,\varrho_\sigma)$ in $\R^d$ consists of a finite list of labeled tiles $\tau_\sigma=(T_1,\ldots,T_n)$  in $\R^d$ called \textit{prototiles}, and a \textit{substitution rule} defining a tessellation $\varrho_\sigma (T_i)$ of each prototile $T_i$, so that every tile in $\varrho_\sigma (T_i)$ is a translation of a rescaled copy of a prototile in $\tau_\sigma$. We denote by $\omega_\sigma(T_i)$ the list of rescaled prototiles whose translations appear in the patch $\varrho_\sigma (T_i)$, presented as   

\begin{equation*}\label{eq:basic_scales}
\omega_\sigma(T_i)=\left(\alpha_{ij}^{\left(k\right)} T_j:\,j=1,\ldots,n,\,\,k=1,\ldots,k_{ij}\right),
\end{equation*} 
and referred to as the \textit{tiles of substitution}. Here $\alpha_{ij}^{\left(k\right)}$  are positive constants, and $k_{ij}$ is the number of \textit{tiles of type} $j$ in $\varrho_\sigma(T_i)$, that is, the number of rescaled copies of $T_j$ in $\varrho_\sigma (T_i)$. 
\end{definition}

We remark that for the sake of simplicity of presentation, we assume throughout that $T_i\neq \alpha T_j$ for any $i\neq j$ and all $\alpha>0$. In this case the geometry of the tile determines its label uniquely, while otherwise one must consider tiles as pairs consisting of sets in $\R^d$ and labels in $\{1,\ldots,n\}$, and the notation becomes cumbersome, see also Remark \ref{rem: labels topology}.

It is helpful to consider the analogy to jigsaw puzzles, where the prototiles $\tau_\sigma$ are puzzles to be
solved using the pieces in $\omega_\sigma$, and $\varrho_\sigma$ gives a solution $\varrho_\sigma(T_i)$ to each of the puzzles. We will refer to multiscale substitution schemes also as \textit{substitution schemes}, and occasionally simply as \textit{schemes}. Unless otherwise stated, all schemes are in $\R^d$ and $\tau_\sigma$ consists of $n$ prototiles $T_1,\ldots,T_n$.

\begin{figure}[ht!]
	\includegraphics[scale=0.6]{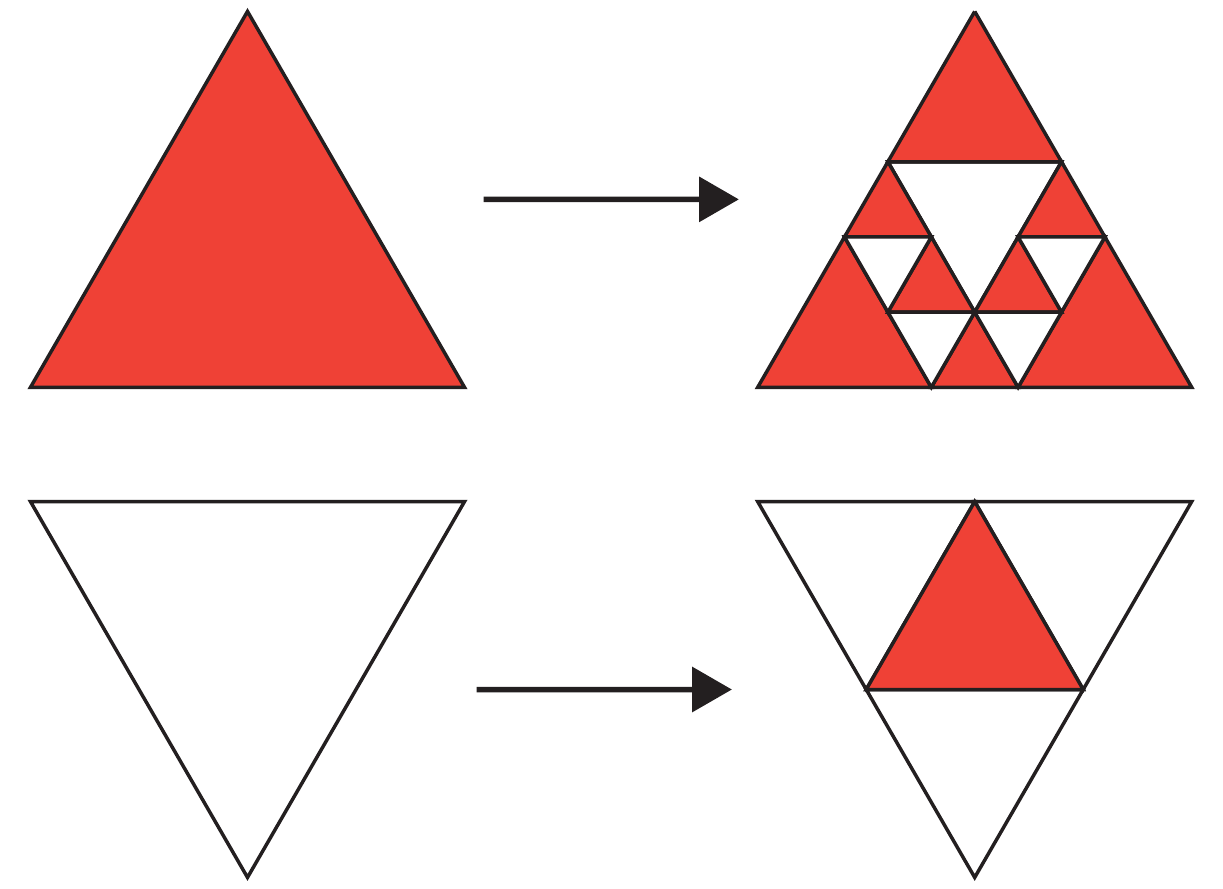}\caption{A multiscale substitution scheme on two triangles $U$ and $D$.}	\label{fig:triangle substitution rule}	
\end{figure}

\begin{example}\label{ex: triangle scheme}
	Figure \ref{fig:triangle substitution rule} describes a multiscale substitution scheme $\sigma$ in $\R^2$. The set of $n=2$ prototiles $\tau_\sigma=(U,D)$ consists of an equilateral triangle $U$ directed up and an equilateral $D$ triangle directed down, both of volume $1$. The tessellations $\varrho_\sigma(U),\varrho_\sigma(D)$ are illustrated on the right-hand side of the figure. For example, $\omega_\sigma(U)$ consists of three copies of $\frac{2}{5}U$, five copies of $\frac{1}{5}U$, four copies of  $\frac{1}{5}D$ and a single copy of $\frac{2}{5}D$. 
\end{example}

\begin{remark}\label{rem isometries instead of translations} Substitution schemes can also be defined so that tiles in $\varrho_\sigma(T_i)$ are only assumed to be isometric to the tiles of substitution, instead of the stronger restriction of being their translations. Examples include the substitution schemes that generate Sadun's generalized pinwheel tilings, see \cite{Sadun - generalized Pinwhell}, and we note that many of the constructions and results discussed below may be extended also to this more general construction. 
\end{remark}

\begin{definition}\label{def:equiv_schemes}
	A substitution scheme is \textit{normalized} if all prototiles are of unit volume. Two substitution schemes are said to be \textit{equivalent} if their prototile sets consist of the same tiles up to scale changes, and the tessellations of the prototiles prescribed by the substitution rules are the same up to an appropriate change of scales.
\end{definition}

 Geometric objects such as patches, tilings and sequences of partitions that are defined using equivalent substitution schemes are identical up to a rescaling by some positive constant.  Clearly, every equivalence class of schemes contains a unique normalized scheme. The geometric nature of our construction implies  the following result.

\begin{prop}\label{prop: volumes sum up to 1}
	Let $\sigma$ be a normalized substitution scheme. For every $i=1,\ldots,n$ the constants $\alpha_{ij}^k$ satisfy the following algebraic equations:
\begin{equation}\label{eq:volumes sum up to 1}
	\sum_{j=1}^{n}\sum_{k=1}^{k_{ij}}\left(\alpha_{ij}^{(k)}\right)^{d}=\sum_{j=1}^{n}\sum_{k=1}^{k_{ij}}\vol\left(\alpha_{ij}^{(k)}T_{j}\right)=\vol(T_{i})=1.\notag
\end{equation} 
\end{prop}

Given a substitution scheme $\sigma$, the constants $\alpha_{ij}^{\left(k\right)}$ associated with the equivalent normalized scheme are called the \textit{constants of substitution}.  From here on, \textbf{all substitution schemes are assumed to be normalized}. 

\subsection{Graphs associated with multiscale substitution schemes}

Denote by $G=\left(\mathcal{V},\mathcal{E},l\right)$ a directed weighted
multigraph with a set of vertices $\mathcal{V}$ and a set of weighted
edges $\mathcal{E}$, with positive weights which are regarded as
lengths. A \textit{path} in $G$ is a directed walk on the edges of
$G$ that originates and terminates at vertices of $G$. A \textit{metric path} in $G$ is a directed walk on edges of $G$ that does not necessarily
originate or terminate at vertices of $G$. An edge of weight $a$
is equipped with a linear parametrization by the interval $[0,a]$,
and the parametrization is used to define the \textit{path distance}
$l$ on edges, paths and metric paths in $G$. In our terminology an edge is assumed to contain its terminal vertex but not its initial one, that is, every vertex is seen as a point contained in each of its incoming edges. The path of length zero with initial vertex $i$ is assumed to consist only of the vertex $i$.

We now define the graph associated with a substitution scheme. 

\begin{definition}	\label{def: associated graph}Given a substitution scheme $\sigma$, the \textit{associated graph} $G_\sigma$ is the following directed weighted graph. The vertices $\mathcal{V}=\left\{ 1,\ldots,n\right\}$ are defined according to the prototiles $\tau_\sigma=(T_1,\ldots,T_n)$, where the vertex $i\in\mathcal{V}$ is associated with the prototile $T_{i}\in\tau_\sigma$. The  edges $\mathcal{E}$ are defined according to the tiles of substitution in $\omega_\sigma$, where the edge $\varepsilon\in\mathcal{E}$ associated with the tile $\alpha T_{j}\in\omega_\sigma(T_{i})$ has initial vertex $i$, terminal vertex $j$ and is of length 
	\begin{equation}
	l(\varepsilon):=\log\tfrac{1}{\alpha}.\notag
	\end{equation}	
\end{definition}
Note that $G_\sigma$ depends only on the elements of $\omega_\sigma(T_i)$ for $T_i\in\tau_\sigma$, and not on the specific configuration in which they appear in the patches $\varrho_\sigma(T_i)$. In other words, $G_\sigma$ can be thought of as the abelianization of $\sigma$, just as the substitution matrix is the abelianization of a fixed scale scheme, see e.g. \cite{BaakeGrimm}. We also remark that if $\sigma$ is a fixed scale scheme, the lengths of all the edges in $G_\sigma$ are the same and the adjacency matrix of $G_\sigma$, when thought of as a combinatorial graph, is precisely the substitution matrix of $\sigma$.

\begin{example} \label{ex: square graph}
	The graph associated with the square substitution scheme described in the introduction is illustrated in Figure \ref{fig:squaregraph}. It consists of a single vertex associated with the single prototile $S$, and self loops associated with the tiles of substitution. There is a single loop of length $\log\frac{5}{3}$ associated with the larger square $\frac{3}{5}S\in\omega_\sigma(S)$ and sixteen distinct loops of length $\log5$ associated with the sixteen smaller squares $\frac{1}{5}S\in\omega_\sigma(S)$.
\end{example}

\begin{figure}[ht!]
	\includegraphics[scale=1.7]{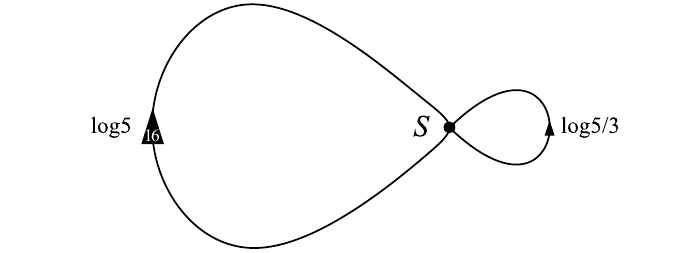}\caption{The graph associated with the  substitution scheme on the square $S$. The large arrow stands for $16$ distinct loops of length $\log5$.}	\label{fig:squaregraph}	
\end{figure}

\begin{example} \label{ex: triangle graph}
	Figure \ref{fig:triangle graph} illustrates the graph associated with the triangle substitution scheme described in Figure \ref{fig:triangle substitution rule}. It consists of $n=2$ vertices associated with the prototiles $(U,D)$. Edges associated with tiles in $\omega_\sigma(U)$ initiate from the vertex on the left, and edges associated with tiles that are rescaled copies of $D$ terminate at the vertex on the right.
\end{example}	

\begin{figure}[ht!]
	\includegraphics[scale=1.7]{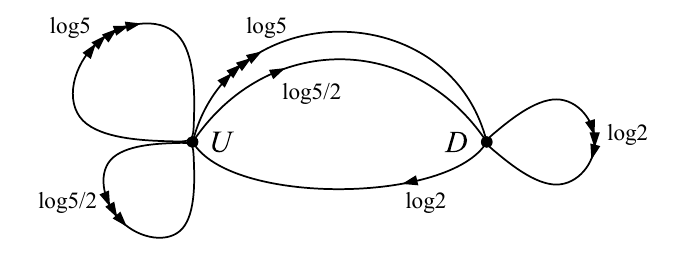}\caption{The graph associated with the substitution scheme on the triangles $U$ and $D$. Multiple arrows stand for multiple distinct edges.}	\label{fig:triangle graph}	
\end{figure}

\begin{remark}\label{rem:graphs of equivalent schemes}
Given a substitution scheme $\sigma$, the graph associated with any equivalent scheme can be derived from $G_\sigma$ by sliding its vertices along its paths so that the lengths of all closed paths is not changed. For a general discussion and additional examples of graphs associated with substitution schemes see  \cite[\S4]{Yotam Kakutani}. 
\end{remark}

\subsection{The substitution semi-flow and the generating patches} 

We now introduce important elements in the definition and study of  multiscale substitution tilings. 

\begin{definition}\label{def: substitution flow}
	Let $\sigma$ be a substitution scheme. The \textit{substitution semi-flow} $F_t(T_i)$, where $t\in\R \ge0$ is referred to as \textit{time}, defines a family of patches in the following way. At $t=0$, set $F_0(T_i)=T_i$, which is a patch consisting of a single tile. As $t$ increases, inflate the patch by a factor $e^{t}$, and substitute tiles of volume larger than $1$ according to the substitution rule  $\varrho_\sigma$.  Equivalently, $F_t(T_i)$ is the patch supported on $e^tT_i$, which is the result of the repeated substitution of $e^tT_i$ and all subsequent tiles with volume greater than $1$ in $e^tT_i$, until all tiles in the patch are of unit volume or less.
	
	Fix a position of $T_i$ so that the origin of $\R^d$ is an interior point, and denote 
	\begin{equation}
	\PPP_{i}:=\left\{F_{t}(T_{i}):\,t\in\R^+\right\},
	\end{equation}
	where $\R^+:=\{t\ge0:\,t\in\R\}$. Note that $\PPP_{i}$ exhausts $\R^d$. The patches $\PPP_\sigma:=\bigcup_{i=1}^{n}\PPP_{i}$ are called the \textit{generating patches} of the scheme $\sigma$. 
\end{definition}

\begin{example}
	Consider the substitution semi-flow associated with the substitution scheme on the two triangles $U,D$ as illustrated in Figure \ref{fig:triangle substitution rule}. The patches illustrated in Figure \ref{fig:new triangle generating patches} are the first elements of $\PPP_\sigma$ with the property that a tile of unit volume appears in the patch. The times $t\in\R^+$ in which they appear are also given, and note that the patch $F_t(U)$ is supported on a triangle of side length $e^t$. 
\end{example}

\begin{figure}[ht!]
	\includegraphics[scale=0.135]{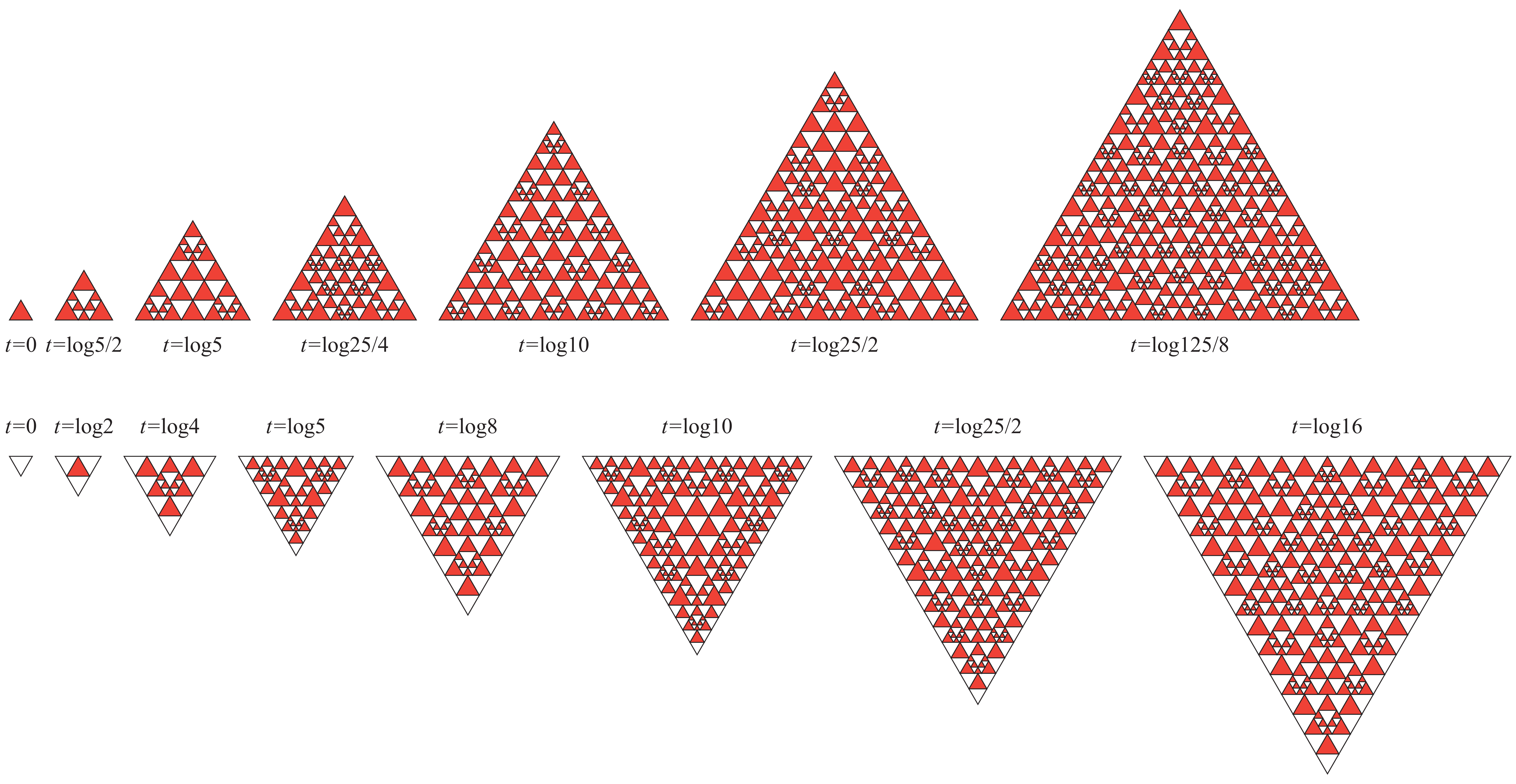}\caption{The patches $F_t(U)$ and $F_t(D)$ for the first values of $t$ for which the patches contain tiles of unit volume.\label{fig:new triangle generating patches}}	
\end{figure}

A very useful observation is that the substitution semi-flow can be modeled by the flow along the edges of the associated graph $G_\sigma$, and that tiles in a patch in $\PPP_\sigma$ correspond to metric paths in $G_\sigma$. This correspondence is summarized in the following proposition, which follows directly from our definitions.

\begin{prop}\label{prop: correspondence of paths and tiles}
	Let $\sigma$ be a substitution scheme with an associated graph $G_\sigma$. Let $T_i\in\tau_\sigma$, fix $t\in\R^+$ and consider a tile $T$ in the patch $F_t(T_i)$. 
	\begin{enumerate}
		\item $T$ corresponds to a unique metric path $\gamma_T$ in $G_\sigma$ that originates at vertex $i$ and is of length $t$. The edges in $\gamma_T$ are determined according to the sequence of ancestors of $T$ under the substitution semi-flow. 
		\item If $T$ is of type $j$ and scale $\alpha=e^{-\delta}$, then $\gamma_T$ terminates at a point on an edge $\varepsilon$ with terminal vertex $j$, and the termination point is of distance $\delta=\log\frac{1}{\alpha}$ from $j$.  
		\item If $\gamma_T$ is of length $t$ and terminates exactly at vertex $j$, then $T$ is a tile of type $j$ and unit volume.  
	\end{enumerate} 
\end{prop}

\begin{cor}
		For every $t\in\R^+$, the patch $F_t(T_i)\in \PPP_i$ corresponds to the family of all metric paths of length $t$ that originate at vertex $i$, and every path in $G_\sigma$ that originates at vertex $i$ and is of length $s>t$ is the continuation of exactly one of these paths. 
	
\end{cor}

\begin{example}\label{ex: Kakutani 1/3}
	A very simple but fundamental family of multiscale substitution schemes are the so-called $\alpha$-Kakutani schemes in $\R$, for $\alpha\in(0,1)$, which can be shown to generate the $\alpha$-Kakutani sequences of partitions of the unit interval, first introduced in \cite{Kakutani}. 
	The unit interval $I$ is the single prototile, and it is substituted by two intervals, one of length $\alpha$ and the other of length $1-\alpha$. 
	For $\alpha=\frac{1}{3}$, the associated graph in this example consists of a single vertex associated with the unit interval $I$, and two loops - the longer one, of length $\log3$, is associated with the interval $\frac{1}{3}I\in\varrho_\sigma(I)$, and the shorter, of length $\log\frac32$, is associated with $\frac{2}{3}I\in\varrho_\sigma(I)$. We remark that Frank and Sadun's one dimensional fusion tiling that was extensively studied in \cite[\S A.5]{Frank-Sadun (ILC fusion)} and mentioned in the introduction can be generated by the $\frac13$-Kakutani scheme.   

	Figure \ref{fig: Kakutani metric paths} concerns with the $\frac{1}{3}$-Kakutani substitution scheme, and illustrates the first three elements of $\PPP_\sigma$ with the property that a tile (interval) of unit volume appears in the patch, together with the metric paths associated with the tiles comprising the patches. The right-most interval in each patch corresponds to the top path beneath it. Note that the patch on the left is $F_0(I)$ and so the single tile in it corresponds to the single metric path of zero length. Also note that tiles of volume $1$ correspond to metric paths that terminate at the vertex.  
\end{example}

\begin{figure}[ht!]
\includegraphics[scale=0.55]{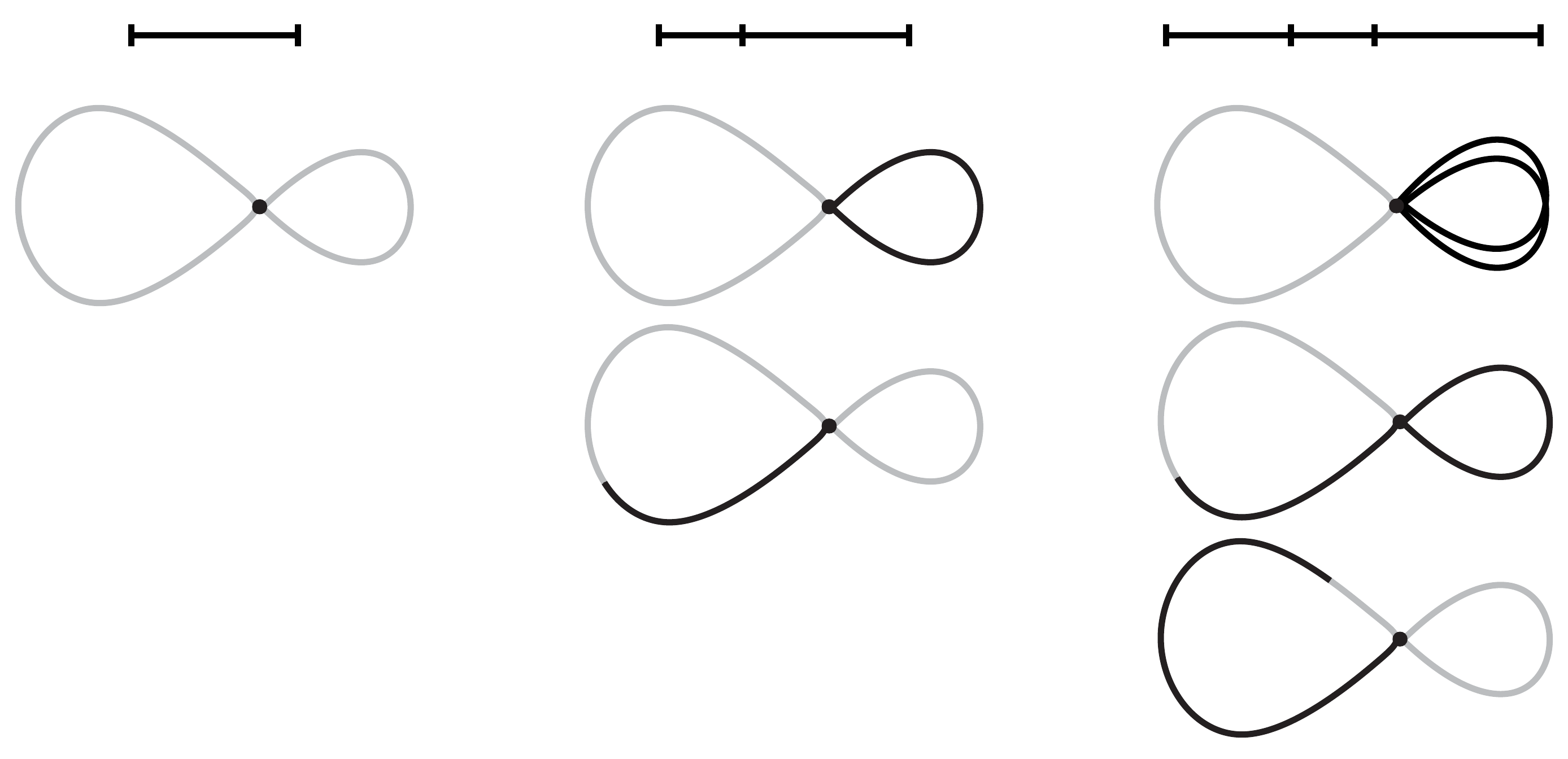}\caption{The patches $F_0(I)$, $F_{\log3/2}(I)$ and $F_{2\log3/2}(I)$ and the metric paths associated with the tiles that comprise them.\label{fig: Kakutani metric paths}}	
\end{figure}

\section{Incommensurable substitution schemes}\label{sec: incommensurability}
The following fundamental properties of substitution schemes are easier to present in terms of the substitution semi-flow and the associated graph. 

\begin{definition}
A substitution scheme $\sigma$ is \textit{irreducible} if  for any $i,j=1,\ldots,n$ there exists $t>0$ so that $F_t(T_i)$ contains a tile of type $j$. 
\end{definition}

In terms of the associated graph $G_\sigma$, this is equivalent to the statement that $G_\sigma$ is \textit{strongly connected}, that is, for any $i,j\in\mathcal{V}=\{1,\ldots,n\}$ there exists a path with initial vertex $i$ and terminal vertex $j$. We remark that this definition coincides with the definition of this notion in the fixed scale setup, see e.g. \cite[\S 2.4]{BaakeGrimm}. From here on \textbf{all schemes are assumed to be irreducible}.

\begin{definition}
	A substitution scheme $\sigma$ is \textit{incommensurable} if there exist $T_i,T_j\in\tau_\sigma$ and two tiles of type $i$ and $j$ in patches in $\PPP_i$ and $\PPP_j$, respectively, which are of unit volume at times $t_1$ and $t_2$, with $t_1\notin\Q t_2$. Otherwise, the scheme is called \textit{commensurable}. 
\end{definition}
	 
In terms of the associated graph $G_\sigma$, incommensurability is equivalent to having two closed paths in $G_\sigma$ of lengths $a$ and $b$, with $a\not\in\Q b$, in which case $G_\sigma$  is said to be an \emph{incommensurable graph}. This follows from part $(3)$ of Proposition \ref{prop: correspondence of paths and tiles}, because a tile $T$ of type $i$ and unit volume in $F_t(T_i)$ corresponds to a closed path in $G_\sigma$ with initial and terminal vertex $i$ and with length $t$. For equivalent definitions and more on incommensurability see \S \ref{subsec:incommensurabiliy}. 

In view of Remark \ref{rem:graphs of equivalent schemes}, incommensurability does not depend on the choice of representative of the substitution scheme equivalence class, in the sense of Definition \ref{def:equiv_schemes}. Note that examples of incommensurable schemes are quite easy to come up with. In fact, incommensurability can be thought of as typical property of substitution schemes, in the sense that for a naive choice of substitution rules, the resulting scheme is incommensurable. 
 
\begin{example}\label{ex: irreducibility and incommensurability of square scheme}
	The substitution schemes illustrated in Figures \ref{fig: square patch} and \ref{fig:triangle substitution rule} are both normalized, since all prototiles are assumed to be of unit volume. In addition they are both irreducible and incommensurable. This can be easily verified by the graphs illustrated in Figures \ref{fig:squaregraph} and \ref{fig:triangle graph}, as both graphs are strongly connected, and both contain pairs of loops of incommensurable lengths.
\end{example}

\begin{remark}
	Commensurable schemes include all \textit{fixed scale} substitution schemes, which are the schemes in which all constants of substitution are equal. The Rauzy fractal scheme introduced in \cite{Rauzy} can be viewed as an example of a commensurable multiscale substitution scheme on a single prototile, which is not of fixed scale. For a further discussion and illustrated examples see \cite{Yotam Kakutani}. 
\end{remark}

\subsection{Equivalent definitions of incommensurability}\label{subsec:incommensurabiliy}  As we focus in the coming sections on incommensurable substitution schemes, we now present some useful equivalent conditions to incommensurability of graphs and schemes.

\begin{lem}\label{lem: Incommensurable equivalent definitions}
	Let $G$ be a strongly connected directed weighted graph. The following are equivalent:
	\begin{enumerate}
		
		\item The graph $G$ is incommensurable.
		\item Every vertex in $G$ is contained in two closed paths of incommensurable lengths.
		\item The set of lengths of closed paths in $G$ is not a uniformly discrete\footnote{A set $S$ in $\R^d$ is \emph{uniformly discrete} if $\inf\{\dist(x,y):\,x,y\in S\}>0$, where $\dist$ is Euclidean distance.} subset of $\R$.
	\end{enumerate}
\end{lem}

\begin{proof}
	$\left(1\right)\Longleftrightarrow\left(2\right)$ It
	is enough to show that there exist two closed paths of incommensurable
	lengths that pass through all vertices of $G$. Let $\alpha$ and
	$\beta$ be two closed paths of lengths $a,b\in\R$ with $a\notin\Q b$, and assume that $\alpha$ passes through vertex $i$ and $\beta$ passes through vertex $j$. Since $G$ is strongly connected, there exists
	a closed path $\gamma$ that passes through all the vertices in $G$,
	and we denote the length of this path by $c$. 
	
	Using the closed paths $\alpha, \beta$ and $\gamma$ we can construct closed paths that pass through all vertices of $G$ and are of lengths $a+c, a+2c$ and $b+c$. By direct calculation
	\[
	\frac{a+c}{b+c}\in\mathbb{Q} \Longrightarrow \frac{a+2c}{b+c}\notin\mathbb{Q},
	\]
	for otherwise $\frac{a}{b}\in\Q$, which is a contradiction. We conclude that either
	the paths of lengths $a+c$ and $b+c$ or those of lengths $a+2c$ and $b+c$ constitute a pair of closed paths
	of incommensurable lengths that pass through all vertices of $G$.  The converse is trivial.
	
	$\left(1\right)\Longleftrightarrow\left(3\right)$ If there exist two closed paths in $G$ of
	lengths $a,b$ such that $a\not\in\mathbb{Q}b$, then for every $\varepsilon>0$ there exist $p,q\in\mathbb{N}$
	such that $\left|aq-pb\right|<\varepsilon$, and so the set of lengths
	of closed paths in $G$ is not uniformly discrete. 
	
	Conversely, assume that $G$ is commensurable, that is, $a\in\Q b$ for any two lengths $a,b$ of closed paths in $G$. Since $G$ is a finite graph, there is a finite set $L$ of lengths for which the length of any closed path in $G$ is a linear
	combination with integer coefficients of elements in $L$. It follows that there exists some $c>0$ so that every closed path is of length which is an integer multiple of $c$, and so the set of lengths of closed paths is uniformly discrete. 
\end{proof}

\begin{definition}
	Given a substitution scheme $\sigma$, denote
	\begin{equation}\label{eq:S_i-to-j}
	\SSS_{i\rightarrow j}:=\left\{ t\in\R^+\,:\,\text{a tile of type \ensuremath{j} and unit volume appears in }F_{t}\left(T_i\right)\right\},
	\end{equation}
	and set $\SSS_i:=\bigcup_{j=1}^n \SSS_{i\rightarrow j}$. 
\end{definition}

Irreducibility of $\sigma$ implies that the sets $\SSS_{i\rightarrow j}$ are all infinite. 
Note that by Proposition \ref{prop: correspondence of paths and tiles}, an equivalent definition is
	\begin{equation*}\label{eq:S_i-to-j_graph}
	\SSS_{i\rightarrow j}=\left\{ t\in\R^+\,:\,\text{a path of length } t, \text{origin \ensuremath{i} and termination \ensuremath{j} appears in }G_\sigma \right\}.
	\end{equation*}

\begin{example}
	The times $t$ that appear in Figure \ref{fig:new triangle generating patches} are the smallest elements in the sets $\SSS_1$ and $\SSS_2$, where $T_1$ is the triangle $U$, and $T_2$ is the triangle $D$. 
	
\end{example}

\begin{lem}	\label{lem:lengths of loops are getting mroe and more dense}Let $\sigma$
	be an irreducible substitution scheme, let $T_i\in\tau_\sigma$ and let $s_{1}<s_{2}<\ldots$
	be an increasing enumeration of the set $\SSS_{i\rightarrow i}$. Then $\sigma$ is incommensurable if and only if
	\begin{equation}
	\lim_{m\to\infty}s_{m+1}-s_{m}=0. \notag
	\end{equation}
\end{lem}

\begin{proof}
	Assume that $\sigma$ is incommensurable, and let $\varepsilon>0$. By Lemma \ref{lem: Incommensurable equivalent definitions}, the associated
	graph $G_\sigma$ has two closed paths $\alpha,\beta$ through vertex $i$, which are of lengths $a<b$ with $a\notin\Q b$. Therefore, there exists some $M\in\N$ so that the set 
	\[
	\left\{ ma(\text{mod }b)\,:\,m\in\left\{ 1,\ldots,M\right\} \right\}
	\]
	is $\varepsilon$-dense in $[0,b)$.
	
	For each $m=1,\ldots,M$ let $k_{m}\in\N$ be such that $0\le ma-k_{m}b<b$. Since $a<b$ we have $k_{m}<M$ for every $m$. Then the set 
	\[
	\left\{ Mb+\left(ma-k_{m}b\right)\,:\,m\in\left\{ 1,\ldots,M\right\} \right\} 
	\]
	is $\varepsilon$-dense in $[Mb,\left(M+1\right)b$), and hence the set 
	\[
	A:=\left\{ Nb+\left(ma-k_{m}b\right)\,:\,m\in\left\{ 1,\ldots,M\right\} ,M\le N\in\N\right\} 
	\]
	is $\varepsilon$-dense in the ray $[Mb,\infty)$. On the other hand, $A\subset\SSS_{i\rightarrow i}$, since $Nb+\left(ma-k_{m}b\right)$
	is the length of a closed path in $G_\sigma$ that consists of $m$ walks along
	$\alpha$ and $N-k_{m}>0$ walks along $\beta$. Since $\varepsilon$ is arbitrary, the assertion follows. 
	
	The converse follows directly from Lemma \ref{lem: Incommensurable equivalent definitions}.
\end{proof}

\begin{cor}\label{cor:SSS_(i-to-j)_becomes_dense}
	Fix $T_i,T_j\in\tau_\sigma$, and let $s_{1}<s_{2}<\ldots$
	be an increasing enumeration of the set $\SSS_{i\rightarrow j}$. Then $\sigma$ is incommensurable if and only if
	\begin{equation}
	\lim_{m\to\infty}s_{m+1}-s_{m}=0. \notag
	\end{equation}
\end{cor}

\begin{proof}
	Irreducibility implies that for every $i,j\in\{1,\ldots,n\}$ there is some $t>0$ so that a tile of type $j$ and unit volume appears in $F_t(T_i)$. So $s+t\in\SSS_{i\rightarrow j}$ wherever $s\in\SSS_{i\rightarrow i}$, and the assertion follows.	
\end{proof}

\section{Construction of multiscale substitution tilings and tiling spaces}\label{sec:Construction of multiscale substitution tilings} 
Let $\CC(\R^d)$ be the space of closed subsets of the metric space $(\R^d,\dist)$, where $\dist$ is the Euclidean distance. Define a metric $D$ on $\CC(\R^d)$ by
\begin{equation}\label{eq:CF-metric}
D\left(A_{1},A_{2}\right):=\inf\left(\left\{ r>0:\begin{matrix}A_{1}\cap B(0,1/r)\subset A_{2}^{+r}\\
A_{2}\cap B(0,1/r)\subset A_{1}^{+r}
\end{matrix}\right\} \cup\{1\}\right),
\end{equation}
where $A^{+r}$ stands for the $r$-neighborhood
of the set $A\subset \R^d$, and $B(x,R)$ is the open ball of radius $R>0$ centered at $x\in\R^d$, both with respect to the metric $\dist$. The topology induced
by the metric $D$ is called the \textit{Chabauty--Fell topology}, and in the context of tiling spaces it is often called the \textit{local rubber topology}. It follows that $D$ is a complete metric on $\CC(\R^d)$, and the space $\left(\CC(\R^d),D\right)$ is compact, see e.g. \cite{de la Harpe}, \cite{Lenz-Stollmann} for this and more concerning this topology.

\ignore{
	The compactness follows from a standard diagonalization argument. Given a sequence $\left(F_{n}\right)$ of closed sets, let $F_{n}^{(m)}:=F_{n}\cap\overline{B\left(0,m\right)}$. The Hausdorff metric is compact on the collection of closed sets in the closed balls $\overline{B\left(0,m\right)}$. Then $F_{n}^{(1)}$ has a converging subsequence in $\overline{B\left(0,1\right)}$, which has a subsequence that converges in $\overline{B\left(0,2\right)}$, and so on. The diagonal of the infinite array of sequence is a converging subsequence of $\left(F_{n}\right)$. 
}

Given a patch $\PP$ or a tiling $\TT$ in $\R^d$, it can be identified with the union of the boundaries of its tiles, denoted by $\partial\PP$ and $\partial\TT$, respectively. With this identification, the elements of $\PPP_\sigma$, as well as tilings of unbounded regions, are viewed as elements of $\CC(\R^d)$, and the metric $D$ can be applied. The following result is straightforward from the definitions.

\begin{prop}\label{prop: translations of close subsets are close}
	Let $\TT_1,\TT_2\in \CC(\R^d)$ be two closed subsets of $\R^d$, and assume that 
	\begin{equation*}
	D(\TT_1,\TT_2)<\varepsilon
	\end{equation*}
	for some $\varepsilon>0$. If $v\in\R^d$ is a vector of Euclidean norm $\norm{v}\le\frac{1}{2\varepsilon}$, then
	\begin{equation*}
	D(\TT_1-v,\TT_2-v)<2\varepsilon.
	\end{equation*}
\end{prop}

Note that for any $t\in\R^+$ and $T_i\in\tau_\sigma$, the patch $F_t(T_i)$ contains finitely many tiles, all of which are of volume at most $1$. We deduce the following result from our definition of the substitution semi-flow.

\begin{prop}\label{prop:F_t is left-continuous}  
	Fix $i\in\{1,\ldots,n\}$ and let $t_0\in\R^+$.
	The function $t\mapsto D\left(F_t(T_i),F_{t_0}(T_i)\right)$, defined on $\R^+$, is left-continuous at $t_0$, that is, 
	\begin{equation*}
	\lim_{t\rightarrow t_0^-} D\left(F_t(T_i),F_{t_0}(T_i)\right)=0.
	\end{equation*} 
\end{prop}

\subsection{The multiscale tiling space}

	The \textit{multiscale tiling space} generated by a substitution scheme $\sigma$ is the space of all tilings $\TT$ of $\R^d$ with the property that every sub-patch of $\TT$ is a limit of translated sub-patches of elements in $\PPP_\sigma$ 
	with respect to the metric $D$. The multiscale tiling space is denoted by $\X^F_\sigma$ and its elements are called \textit{multiscale substitution tilings}. We will often refer to $\X^F_\sigma$ simply as the \textit{tiling space}. Following the terminology of \cite{Frank-Sadun (ILC fusion)}, 
	we refer to patches of tilings in $\X^F_\sigma$ that are sub-patches of elements of $\PPP_\sigma$ as \textit{legal} patches, and to the rest of the patches as \textit{admitted in the limit}.
	
The following Proposition \ref{prop:Equivalent def for X^F_sigma} provides an equivalent definition for the tiling space $\X^F_\sigma$.
\begin{prop}\label{prop:Equivalent def for X^F_sigma}
	Let $\sigma$ be a substitution 	scheme. Then 
	\begin{equation}\label{eq:another_def_of_the_tiling_space}
	\X^F_\sigma=\left\{ \TT=\lim_{k\to\infty}\PP_{k}+v_{k}:\PP_{k}\in\PPP_\sigma,\,v_{k}\in\R^d,\text{ and }\TT\text{ tiles }\R^d\right\},
	\end{equation}
	where limits are taken with respect to the metric $D$.
\end{prop}

\begin{proof}
	Denote by $Y$ the right-hand side of \eqref{eq:another_def_of_the_tiling_space}. To see that $\X^F_\sigma\subset Y$,
	take $\TT\in \X^F_\sigma$ and set $\QQ_{k}=\left[B(0,k)\right]^{\TT}$, the patch that consists of all tiles in $\TT$ that intersect the ball of radius $k$ around the origin. On the one hand, clearly $D\left(\TT,\QQ_{k}\right)\le1/k$  by definition of the metric $D$. On the other hand, by definition of the space $\X^F_\sigma$,	for any $k\in\N$ the patch $\QQ_{k}$ is obtained as a limit of the form $\lim_{j\to\infty}\RR_{j}+u_{j}$, where $\RR_{j}$ are sub-patches of patches $\PP_{j}\in\PPP_\sigma$ and $u_{j}\in\R^d$. It follows that there exists $j_{k}\in\N$ such that $D\left(\QQ_{k},\PP_j+u_j\right)<1/k$ for any $j\ge j_k$. Combining the above we get $\TT=\lim_{k\to\infty}\PP_{j_{k}}+u_{j_{k}}$, and so $\TT\in Y$. 
	
	Conversely, every sub-patch of a limiting object of the form $\lim_{k\to\infty}\PP_k+v_k$
	for $\PP_k\in\PPP_\sigma,\,v_k\in\R^d$, is a limit of
	sub-patches of the patches $P_k+v_k$. Assuming additionally that
	$\lim_{k\to\infty}\PP_k+v_k$ tiles $\R^d$ implies that
	the limit is in $\X^F_\sigma$.
\end{proof}

\begin{cor}
	The space $\X^F_\sigma$ is a closed, non-empty subset of $\CC(\R^d)$. 
\end{cor}

\begin{proof}
	To see that $\X^F_\sigma\neq\varnothing$ one simply takes a limit of a
	converging sequence of patches $\PP_k\in\PPP_\sigma$ whose supports exhaust $\R^d$. By compactness of the space $(\CC(\R^d),D)$ such sequences exist. To see that it is closed,
	let $\TT_{k}\in \X^F_\sigma$ be a sequence so that $\TT=\lim_{k\to\infty}\TT_k$.
	For each $k$ let $\PP_k\in\PPP_\sigma,\,v_k\in\R^d$
	be so that $D\left(\TT_{k,}\PP_k+v_k\right)<1/k$. Then $\TT=\lim_{k\to\infty}\PP_{k}+v_{k}$
	and so $\TT\in \X^F_\sigma$.
\end{proof}

The following proposition is a simple exercise in convergence of compact sets with respect to the Hausdorff metric, and establishes a first simple result about tiles that appear in multiscale substitution tilings.

\begin{prop}\label{prop:limiting tiles are similar to prototiles} 
	Let $\sigma$ be a substitution scheme. Every tile of every $\TT\in \X^F_\sigma$ is similar to one of the prototiles in $\tau_\sigma$. 

\end{prop}

Notice that due to the nature of the metric $D$, even though in legal patches tiles of unit volume are not yet substituted, the tiling space $\X^F_\sigma$ contains tilings in which they are already subdivided, that is, instead of a translated copy of $T_i\in\tau_\sigma$ there appears a translated copy of the patch $\varrho_\sigma\left(T_i\right)$. This is only possible in tilings that are admitted in the limit, hence they are negligible for most of our purposes, and certainly there are legal patches that come arbitrarily close to them.

\begin{remark}\label{rem: labels topology}
	In view of Proposition \ref{prop:limiting tiles are similar to prototiles}, it is possible to assign labels (or equivalently colors) to every tile of every tiling in $\X^F_\sigma$. Under the assumption that distinct prototiles are geometrically different, that is, that $T_i\neq \alpha T_j$ for any $i\neq j$ and all $\alpha>0$, this can be done in a unique way. Relaxing this assumption, there may exist multiple possible labelings. In such a case we can define a \textit{space of labeled tilings} as a product space of $\X^F_\sigma$, and consider the natural metric on this product space, equipped with an additional discrete component that takes into account that labeled patches are close not only if they are close geometrically, but also only if their corresponding tiles have matching labels. Then we may restrict to the sub-space of all limits of translations of sub-patches of labeled elements of $\PPP_\sigma$ with respect to this refined metric, to define the \textit{labeled tiling space} generated by $\sigma$. We note that this is similar in nature to the \textit{space of Delone $\kappa$-sets} considered in \cite{Lee-Solomyak}. All results of this paper hold for such labeled tilings as well, with similar proofs, but the notations become rather cumbersome. Therefore, for simplicity of presentation, we assume that prototiles are geometrically different and proceed to work with the metric $D$ and the space $\X_\sigma^F$ as defined above.  
\end{remark}

\subsection{The substitution semi-flow on $\X_\sigma^F$ and stationary tilings} 
Consider a substitution scheme $\sigma$ and a tiling $\TT\in\X_\sigma^F$. By Proposition \ref{prop:limiting tiles are similar to prototiles}, every tile $T$ in $\TT$ is a rescaled copy of a prototile in $\tau_\sigma$. The substitution semi-flow $F_t$, which was defined for prototiles in Definition \ref{def: substitution flow}, can thus be naturally extended to tiles in $\TT$ and to patches in $\TT$, and so to the entire tiling $\TT$. 
It follows that for every $\TT\in\X_\sigma^F$ one has:
	\begin{enumerate}
		\item For any $t\in\R^+$, $F_t(\TT)\in\X_\sigma^F$.
		\item For any $t,s\in\R^+$, $F_t(F_s(\TT))=F_{t+s}(\TT)$.
	\end{enumerate}
Thus, the substitution semi-flow defines a semi-flow on the space $\X_\sigma^F$. Also note that for every $x\in\R^d$
	\begin{equation}\label{eq: flow on translation is translation of flow}
	F_s(\TT-x)=F_s(\TT)-e^sx,
	\end{equation}
which brings to mind the relationship between the geodesic and horospheric flows in homogeneous dynamics. Theorem \ref{thm:stationary tilings exist} below establishes the existence of non-trivial periodic orbits for the substitution semi-flow.

\begin{thm}\label{thm:stationary tilings exist}
	Let $\sigma$ be an irreducible substitution scheme. There exist an $s\in\R^+$ and a tiling $\SS\in\X^F_\sigma$ so that  $F_s(\SS)=\SS$.
\end{thm}

\begin{proof}
	Irreducibility of $\sigma$ implies that for every $i$ there are infinitely many $t\in\R^+$ for which the patch $F_t(T_i)$ contains a tile of type $i$ that does not share its boundary. It follows that there are infinitely many $s\in\SSS_{i\rightarrow i}$ for which the patch $F_s(T_i)$ contains a translate of $T_i$ that does not share its boundary. For any such $s\in\SSS_{i\rightarrow i}$ there exists a unique location in which $T_i$ can be initially positioned around the origin, so that the inflated patch $F_s(T_i)$ contains the patch  $T_i=F_0(T_i)$ as a sub-patch with support in that very same location. Note that under our assumption, this \emph{control point} is an interior point of $T_i$. In such a case also $F_{2s}(T_i)$ contains $F_s(T_i)$ as a sub-patch, and more generally the patch $F_{ks}(T_i)$ contains $F_{(k-1)s}(T_i)$ for every $k\in\N$, that is, the patches $F_{ks}(T_i)$ define a nested sequence of patches. Therefore, the union 
	\[
	\SS:=\bigcup_{k=0}^{\infty}F_{ks}(T_i)
	\]
	is a tiling in $\X_\sigma^F$, and it clearly satisfies $F_s(\SS)=\SS$.
\end{proof}

Let $\SS$ be as above, then for any $a\in\R^+$
\begin{equation}\label{eq:stationary_contains_all_legal_patches}
F_s(F_a(\SS))=F_a(F_s(\SS))=F_a(\SS)
\end{equation} 
and $F_a(\SS)=\bigcup_{k=0}^{\infty}F_{a+ks}(T_i)$, which leads us to the following definition. 

\begin{definition}
	A tiling $\SS\in\X_\sigma^F$ of the form 
	\begin{equation}\label{eq: stationary}
	\SS=\bigcup_{k=0}^{\infty}F_{a+ks}(T_i),
	\end{equation}
	where $s\in\SSS_{i\rightarrow i}$, $a\in\R^+$ and $F_{a+ks}(T_i)$ define a nested sequence of patches, is called a \textit{stationary tiling}. 
	As in the construction in the proof of Theorem \ref{thm:stationary tilings exist}, we assume throughout that the origin is an interior point of $T_i$. 
	We also assume that $s\in\SSS_{i\rightarrow i}$ is minimal in the sense that the associated closed path in $G_\sigma$ is a prime orbit, that is, not the concatenation of multiple copies of a single closed orbit. 
\end{definition}

We remark that in fact there are infinitely many values of $s\in\R^+$ for which there are tilings $\SS\in\X_\sigma^F$ that satisfy $F_s(\SS)=\SS$. Moreover, in view of \eqref{eq:stationary_contains_all_legal_patches}, for every $\PP\in\PPP_\sigma$ there exists a stationary tiling that contains a translate of $\PP$ as a sub-patch. In addition, every stationary tiling can be represented as 
	\begin{equation*}\label{eq: stationary2}
	\SS=\bigcup_{k=0}^{\infty}F_{ks}(T),
	\end{equation*}
   where $T$ is some rescaled copy of a prototile.

\begin{example}\label{ex: stationary squares}
	Once again we consider the square substitution scheme with $\tau_\sigma=(S)$, illustrated in Figure \ref{fig: square scheme} and discussed in previous examples. The large square in the middle of the patch $\varrho_\sigma(S)$ is associated with the single loop of length $\log\frac{5}{3}$ in the associated graph $G_\sigma$ illustrated in Figure \ref{fig:squaregraph}. It follows that for $s=\log\frac{5}{3}$ and the initial positioning of $S$ so that the origin is in the center of the square, we can define a sequence of patches $F_{ks}(S)=F_{k\log(5/3)}(S)$ with support $e^{ks}=(5/3)^k$ so that the patch $F_{ks}(S)$ contains $F_{(k-1)s}(S)$ for every $k\in\N$. 
	For example, the patches illustrated in Figure  \ref{fig:squares stationary} are the $k=0,1,\ldots,6$ elements of the nested sequence.
\end{example}

 \begin{figure}[ht!]
 \includegraphics[scale=0.16]{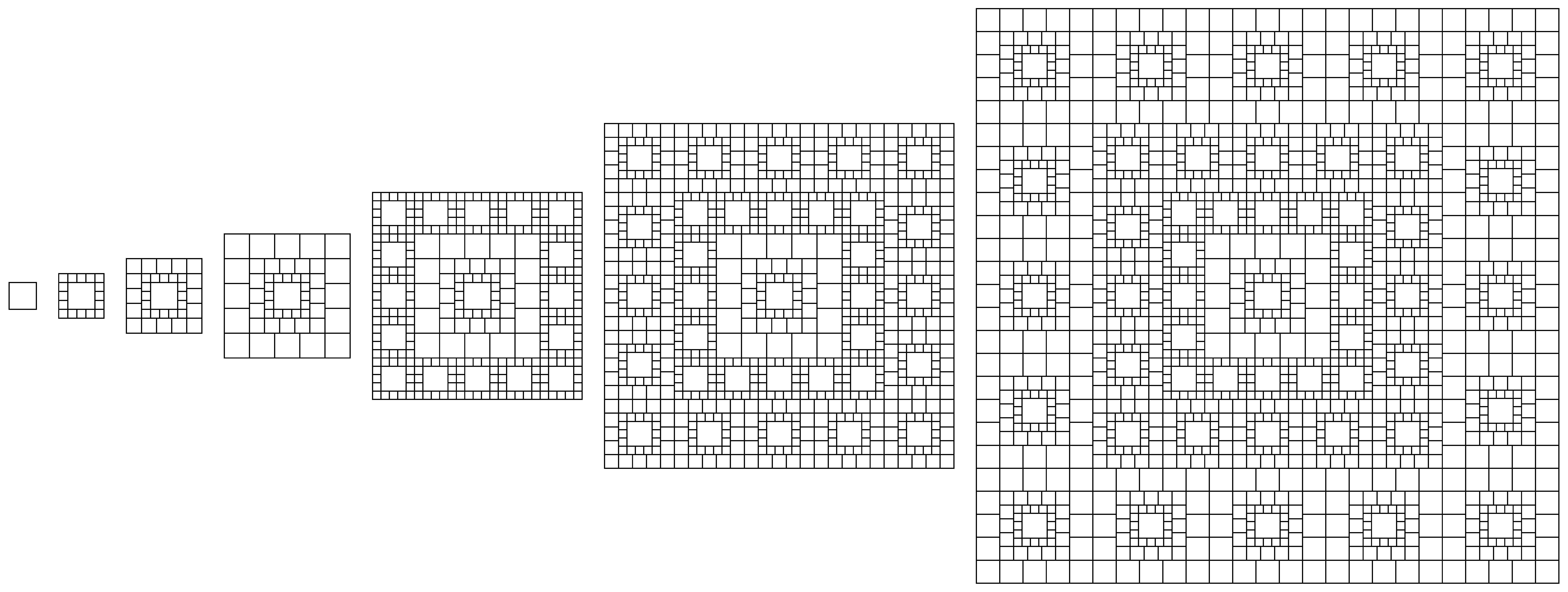}\caption{The patches $F_0(S), F_{1\cdot\log5/3}(S),\ldots,F_{6\cdot\log5/3}(S)$, the first seven elements of the nested sequence of patches that define a stationary tiling.\label{fig:squares stationary}}	
 \end{figure}  
 
 \begin{example}
 	The three patches illustrated in Figure \ref{fig:triangle stationary} are the $k=0,1$ and $2$ elements of the nested sequence of patches of the stationary tiling construction associated with the choice of the central copy of $\frac{1}{5}U$ in $\varrho_\sigma(U)$, where $\sigma$ is the substitution scheme on two triangles illustrated in Figure \ref{fig:triangle substitution rule}. 
 \end{example}

  \begin{figure}[ht!]
  	\includegraphics[scale=0.25]{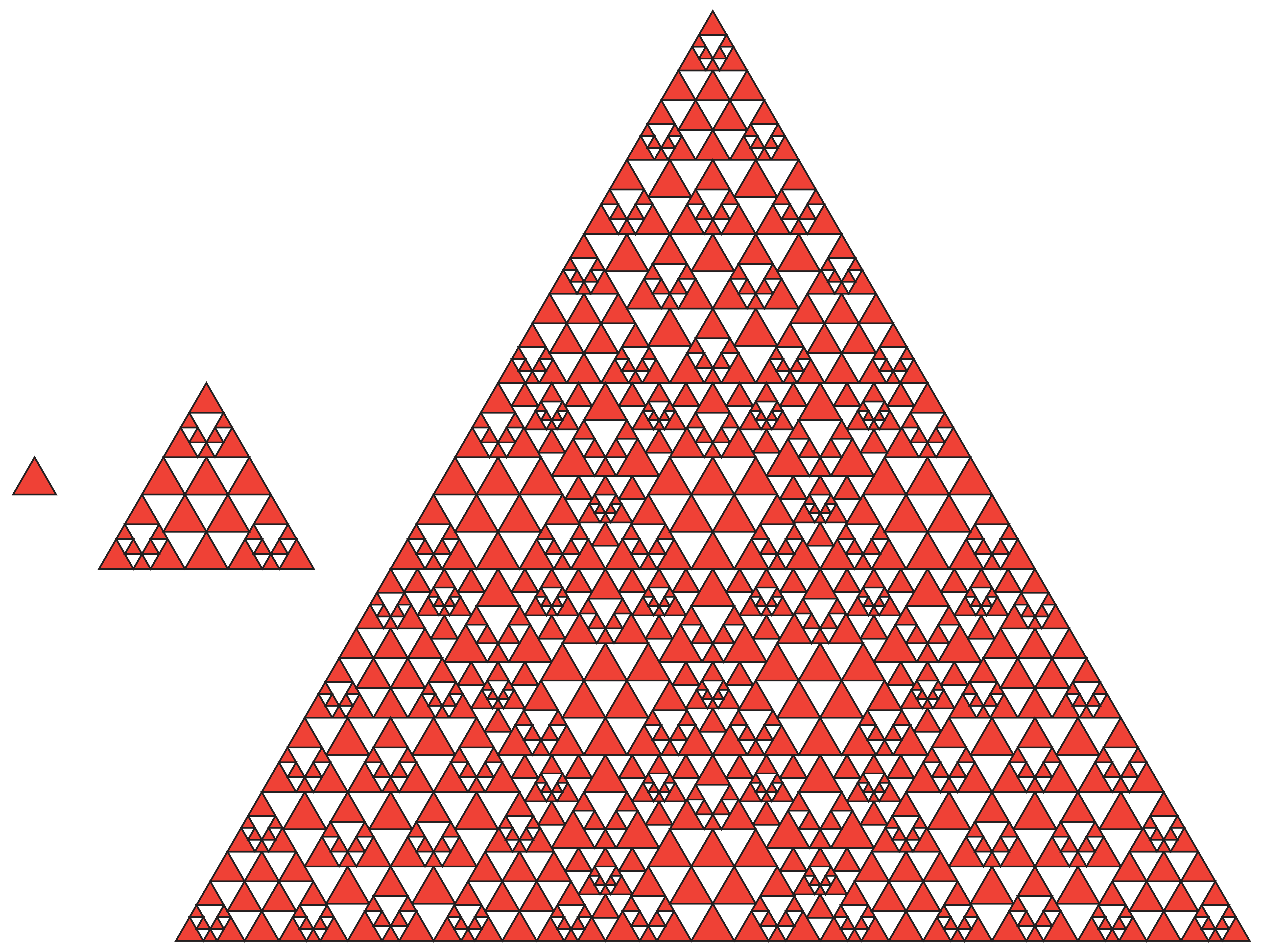}\caption{The patches $F_0(U)$, $F_{1\cdot\log5}(U)$ and $F_{2\cdot\log5}(U)$,  the first three elements of the nested sequence of patches that define a stationary tiling.\label{fig:triangle stationary}}	
  \end{figure} 

\begin{remark}\label{rem: stationary with isometries}
	Stationary tilings can be generated using the more general definition of substitution schemes in which isometries are allowed in the substitution rule, and not only translations, see Remark \ref{rem isometries instead of translations}. Assume $F_s(T_i)$ contains a copy of $\varphi(T_i)$, where $\varphi$ is an isometry of $\R^d$. The patch $\varphi^{-1}(F_s(T_i))$ contains a copy of $T_i$ as a sub-patch, and similarly as described above, we get a sequence of patches $\varphi^{-k}(F_{ks}(T_i))$ which define a stationary tiling $\TT$ in the sense that $\varphi^{-1}(F_s(\TT))=\TT$. 
	Generalized pinwheel tilings, that were introduced and studied in a most illuminating way in \cite{Sadun - generalized Pinwhell}, can be formulated as stationary tilings in this way, with $\varphi$ taken to be rotations by an appropriate angle.
\end{remark}

We end this section with a remark on tilings generated by commensurable schemes.

\begin{remark}\label{rem: commensurable spaces and fixed scale}
	Let $\sigma$ be a (not necessarily normalized) fixed scale scheme with scaling constant $\alpha$, then all edges in the associated graph $G_\sigma$ are of length $\log\frac{1}{\alpha}$. Define a type of substitution semi-flow $\widetilde{F_t}$ in which the tiles are substituted simultaneously when they all reach the volumes of the original prototiles they are associated with. The patches $F_{k\log(1/\alpha)}(T_i)$ are simply the patches $(1/\alpha)^{k}\varrho_\sigma^k(T_i)$, and all patches in $\PPP_{i}$ are inflations of such patches for some $k\in\N$. 
	
	Recall that in the classical theory of fixed scale tilings, the well studied tiling space, which we denote here by $\X_\sigma$, can be defined using the patches $(1/\alpha)^{k}\varrho_\sigma^k(T_i)$, see \cite{BaakeGrimm} and references within. It follows that in such a case, the space $\X^{\widetilde{F}}_\sigma$, defined using the semi-flow $\widetilde{F_t}$, can be expressed simply as the product $\X_\sigma\times\left(\alpha,1\right]$. It was shown in \cite{Yotam Kakutani} that Kakutani sequences of partitions generated by commensurable schemes can be represented as subsequences of generation sequences of partitions generated by fixed scale schemes. In our setting and language, this simply means that given a commensurable scheme $\sigma$ there exists a fixed scale scheme $\widetilde{\sigma}$ so that $\X^F_\sigma= \X^{\widetilde{F}}_{\widetilde{\sigma}}$, and so for the study of commensurable multiscale tilings, one should refer to results on standard fixed scale substitution tilings and the standard spaces of tilings. As we will see below, incommensurable tilings and multiscale tiling spaces differ from the classical setup in various ways. 
\end{remark}

\section{Scales and complexity in stationary tilings}\label{sec: scales and complexity}

From here on \textbf{all schemes are assumed to be incommensurable}. We show that an incommensurable stationary tiling has tiles in a set of scales which  is dense within a certain interval of scales. In particular, it is of infinite local complexity, see e.g. \cite{BaakeGrimm}.

\begin{definition}\label{def:possible_scales}
	Given a substitution scheme $\sigma$ and a prototile $T_j\in\tau_\sigma$, denote 
	\begin{equation}\label{eq:beta_j,min}
	\beta^{\min}_{j}:=\min\left\{\alpha:\ \exists i\in\{1,\ldots,n\}, \varrho_\sigma(T_i) \text{ contains a copy of } \alpha T_j \right\}.
	\end{equation}
	The interval $(\beta^{\min}_{j},1]$ is called the \emph{interval of legal scales of tiles of type $j$}.
	Note that the prototiles $T_i,T_j\in\tau_\sigma$ on the right-hand side of \eqref{eq:beta_j,min} are assumed to be of unit volume.  A \textit{tile of legal type and scale} is any translated copy of $\alpha T_j$ for $T_j\in\tau_\sigma$ and $\alpha\in(\beta^{\min}_{j},1]$. Note that tiles of type $j$ and scale $\beta^{\min}_{j}$ can appear in tilings that are admitted in the limit, but never in any legal patch or in a (translation of a) stationary tiling.
\end{definition}

Recall that a fixed scale substitution scheme is called \textit{primitive} if there exists $k\in\N$ so that all types of tiles belong to the patch that is the result of $k$ applications of the substitution rule on any initial prototile. The following theorem demonstrates how incommensurability takes the part of primitivity in multiscale substitution tilings. It also plays an important role in the proof of minimality of the tiling dynamical system in \S\ref{sec:Dynamics}.

\begin{thm}\label{thm:The_scales_are_dense_in_SS}
	Let $\sigma$ be an irreducible incommensurable substitution scheme, let $\SS\in\X_\sigma^F$ and $s\in\R^+$ be so that $\SS=F_s(\SS)$, and let $\varepsilon>0$. Then there exists $K\in\N$ so that for every integer $k\ge K$, a tile $T$ in $\SS$ and $j\in\{1,\ldots,n\}$, the set 
	\[\emph{Scales}(e^{ks}T,j):= \left\{\alpha: \text{the patch supported on } e^{ks}T \text{ in } F_{ks}(\SS)=\SS \text{ contains a copy of } \alpha T_j \right\}\]
	is $\varepsilon$-dense in $(\beta^{\min}_{j},1]$. In particular, the set of scales in which tiles of each type appear in stationary tilings is dense within their intervals of legal scales.
\end{thm}

\begin{proof}
	Since every tile $T$ in $\SS$ is a translated copy of $\alpha T_i$ for some $T_i\in\tau_\sigma$ and some $\alpha\in  (\beta^{\min}_{i},1]$, the patch supported on $e^{ks}T$ is a translated copy of $F_{ks}(\alpha T_i)=F_{ks-\log(1/\alpha)}(T_i)$. Note that since $F_s(\SS)=\SS$, the set $e^{ks}T$ is indeed a support of a patch in $\SS$.
	
	Fix $\varepsilon>0$ and $j\in\{1\ldots,n\}$. We show that there exists $K_j\in\N$ so that $\text{Scales}(e^{ks}T,j)\cap (c,c+\varepsilon]\neq\varnothing$ for every $c\in(\beta^{\min}_{j},1-\varepsilon]$, every tile $T$ in $\SS$, and every integer $k\ge K_j$. Set $\ell=\ell(j)\in\{1\ldots,n\}$ so that $\varrho_\sigma(T_\ell)$ contains a copy of $\beta^{\min}_{j} T_j$. First observe that for every $c\in(\beta^{\min}_{j},1-\varepsilon]$ and $t\in\R^+$, the following implication holds:
	\begin{align}\label{eq:implication_in_'scales_are_dense'}
	F_{t}(T_i) &\text{ contains a copy of }T_\ell \qquad \Longrightarrow \\ \notag
	F_{t+\log c-\log \beta^{\min}_{j}}(T_i) &\text{ contains a copy of }F_{\log c-\log \beta^{\min}_{j}}(\beta^{\min}_{j}T_j) = cT_j 
	\end{align}

	For $i\in\{1\ldots,n\}$ let $s_{1}(i)<s_{2}(i)<\ldots$ 
	be an increasing enumeration of $\SSS_{i\rightarrow \ell}$, then by Corollary \ref{cor:SSS_(i-to-j)_becomes_dense} 
	\begin{equation}
	\lim_{m\to\infty}s_{m+1}(i)-s_{m}{(i)}=0.\notag
	\end{equation}
	It follows that for every $\delta>0$ there exists $K_j\in\N$ so that for every $k\ge K_j$, $i\in\{1,\ldots,n\}$, $\alpha\in(\beta_i^{\min},1]$ and $c\in(\beta_j^{\min},1-\varepsilon]$ there exists some $m\in\N$ with 
	\begin{equation}\label{eq:in_"scales are dense"2}
	ks-\log c+\log \beta^{\min}_{j}-\delta<s_m{(i)}+\log(1/\alpha)\le ks-\log c+\log \beta^{\min}_{j}.
	\end{equation}
	Put $h:=s_m(i)+\log(1/\alpha)+\log c-\log \beta^{\min}_{j}$, then \eqref{eq:in_"scales are dense"2} simply says $h\in(ks-\delta,ks]$.
	
	Suppose that $T$ is of type $i$ and scale $\alpha$. Since $s_m{(i)}\in\SSS_{i\rightarrow \ell}$, the patch $F_{s_m{(i)}+\log(1/\alpha)}(T)=F_{s_m{(i)}}(T_i)$ contains a copy of $T_\ell$. By the implication in  \eqref{eq:implication_in_'scales_are_dense'},  $F_h(T)$ contains a copy of $cT_j$. Write $\eta:=ks-h$, then $\eta<\delta$ and $F_{ks}(T)=F_\eta(F_h(T))$ contains a copy of $F_\eta(cT_j)=e^\eta cT_j$, where the equality holds whenever $\delta$ is small enough. In particular, for $\delta\le\log\left(1+\varepsilon/\beta\right)$, where $\beta:=\min_j\beta_j^{\min}$, we have $e^\eta c \in (c,c+\varepsilon]$, because $c\ge \beta$. Taking $K=\max_j K_j$ completes the proof.
\end{proof}

\subsection{Tile complexity}
Let $\sigma$ be an irreducible substitution scheme, and consider a stationary tiling $\SS=\bigcup_{k=0}^{\infty}F_{ks}(T)\in \X^F_\sigma	$. For every $k\ge0$, denote by $c_{\SS,j}(k)$ the number of distinct scales in which a tile of type $j$ appears in $F_{ks}(T)$, and set $c_{\SS}(k)=\sum_{j=1}^n c_{\SS,j}(k)$. So $c_{\SS}(k)$ is the number of distinct tiles up to translation in $F_{ks}(T)$, and $c_{\SS}(k)$ is called the \textit{tile complexity function} of $\SS$. Note that it is well defined because $s$ is assumed to be minimal with the property $F_s(\SS)=\SS$. Since $F_{ks}(T)$ is a nested sequence of patches, clearly $c_{\SS}(k)$ is non-decreasing.

\begin{thm}\label{thm:scale complexity}
	Let $\sigma$ be an irreducible substitution scheme and let $\SS=\bigcup_{k=0}^{\infty}F_{ks}(T)\in \X^F_\sigma	$ be a stationary tiling. If $c_{\SS}(\ell+1)=c_{\SS}(\ell)$ for some $\ell\in\N$, then $c_{\SS}(k)=c_{\SS}(\ell)$ for all $k\ge \ell$. Moreover, this is the case if and only if the scheme $\sigma$ is commensurable. 
\end{thm}

\begin{proof}
	Assume $c_{\SS}(\ell+1)=c_{\SS}(\ell)$. This means that every type and scale that appears in the patch $F_{(\ell+1) s}(T)$ also appears in the patch $F_{\ell s}(T)$. Since the tiles in $F_{(\ell+1) s}(T)$ are all sub-tiles in patches defined by applying $F_s$ to the tiles of $F_{\ell s}(T)$, we deduce that applying $F_s$ to the tiles of $F_{(\ell+1) s}(T)$ defines a patch with tiles of the same set of types and scales. Therefore, repeated applications of $F_s$ will always result with patches whose tiles are of the same types and scales, and so $c_{\SS}(k)=c_{\SS}(\ell)$  for every $k\ge\ell$.

	For the second part of the theorem, if $\sigma$ is incommensurable then Theorem \ref{thm:The_scales_are_dense_in_SS} in particular implies that $c_{\SS}(k)$ tends to infinity. Conversely, assume $c_{\SS}(k)$ is strictly increasing, then in particular it is unbounded as an increasing sequence of integers. If $\sigma$ is commensurable, then Remark \ref{rem: commensurable spaces and fixed scale} implies that there exists a standard substitution tiling $\mathcal{T}$ generated by a fixed scale substitution scheme, so that for any $k\in\N$ the patches $F_{ks}(T)$ are sub-patches of $\mathcal{T}$ (see also \cite[Theorem 7.2]{Yotam Kakutani}). Since the fixed scale substitution tiling $\mathcal{T}$ has only finitely many tiles up to translation, the complexity function $c_{\SS}(k)$ is bounded, contradicting our assumption. 
\end{proof}

Given a finite alphabet $\mathcal{A}$ and an infinite sequence $u\in\mathcal{A}^{\N}$,  the cardinality $p_k(u)$ of the set of \textit{words} in $u$ of length $k$ defines the \textit{complexity function} of $u$. The statement of Theorem \ref{thm:scale complexity} resembles that of the well-known result that a sequence $u\in\mathcal{A}^{\N}$ is periodic if and only if there exists $k\in\N$ for which $p_k(u)=p_{k+1}(u)$, see e.g. \cite[Proposition 4.11]{BaakeGrimm}. Non-periodic sequences with minimal complexity function $p_k(u)=k+1$ for all $k\in\N$ are called \textit{Sturmian sequences}. The following Proposition \ref{prop: Sturmian tilings} establishes the existence of ``Sturmian'' incommensurable stationary tilings.

\begin{figure}[ht!]
	\includegraphics[scale=3.5]{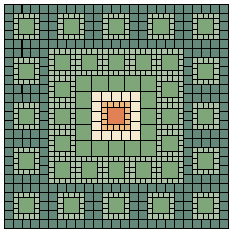}\caption{The patch $F_{5s}(S)$. Squares are colored according to the $c_{\SS}(5)=6$ distinct scales in which the appear in $F_{5s}(S)$.}	\label{fig:tile complexity}	
\end{figure}

\begin{prop}\label{prop: Sturmian tilings}
	There exist incommensurable stationary tilings so that for all $k\in\N$ we have $c_{\SS}(k+1)=c_{\SS}(k)+1$. 
\end{prop}

\begin{proof}
	Let $\mathcal {S}$ be the stationary tiling with square tiles, constructed in Example \ref{ex: stationary squares} using the substitution scheme on the unit square $S$. Here $s=\log(5/3)$ and $T=S$, and we denote by $\mathcal{S}_k$ the patch $F_{k\log(5/3)}(S)$. Clearly $c_{\SS}(0)=1$, and from Figure \ref{fig:squares stationary} we deduce $c_{\SS}(1)=2$. Assuming, by induction, that $c_{\SS}(m)=m+1$ for all $1\le m<k$, we prove that $c_{\SS}(k)=k+1$. 
	
	Since $\SS$ is stationary, for every $m$ the set of $c_{\SS}(m)$ different scales that appear in $F_{ms}(T_i)$ contains the $c_{\SS}(m-1)$ scales that appear in $F_{(m-1)s}(T_i)$. In particular, using the induction hypothesis, the set of $k$ scales that appear in $\SS_{k-1}$ consists of the $k-1$ scales that appear in $\SS_{k-2}$, and one additional scale $\alpha\in(\frac{1}{5},1]$.  Therefore, squares in $\mathcal{S}_{k}$ that are descendants of squares of these $k-1$ scales in $\mathcal{S}_{k-1}$, appear in $k$ scales. Regarding the scale $\alpha$, we consider two cases: 	If $\alpha\le\frac{3}{5}$, then applying $F_s$ to squares of scale $\alpha$ results in squares of side length $e^s\alpha = \frac{5}{3}\alpha \le1$, and so at most one new scale appears in $\mathcal{S}_{k}$. Otherwise $\alpha>\frac{3}{5}$, and so applying $F_s$ to squares of scale $\alpha$ results in a patch consisting of a single central square of scale $\frac{3}{5}\cdot e^s\alpha =\alpha$, and $16$ smaller squares of scale  $\frac{1}{5}\cdot e^s\alpha $. And so once again, at most one new scale appears in $\mathcal{S}_{k}$. Since the scheme is incommensurable, by Theorem \ref{thm:scale complexity}, at least one new scale must appear in $\mathcal{S}_{k}$, finishing the proof.
\end{proof}

As an immediate corollary of Lemma \ref{lem:poly(log)_lemma} that is more naturally suited in \S \ref{sec: BD BL}, we deduce the following result.

\begin{prop}\label{prop: polynomial growth}
	The tile complexity function growth rate is at most polynomial, that is, for any stationary tiling $\SS$ there exists $m\in\N$ for which $c_{\SS}(k)=O\left(k^m\right)$, where $m$ depends on the parameters of the generating substitution scheme and on the parameter $s$ of the stationary tiling. 
\end{prop}

\subsection{Patches and scales}
Recall that given a substitution scheme $\sigma$, a patch is called legal if it appears as a sub-patch of some element of $\PPP_\sigma$. We show here that if $\sigma$ is irreducible and incommensurable, then for every legal patch, every stationary tiling contains rescaled copies of the patch, with a dense set of scales.  

\begin{lem}\label{lem: legal patches appear in stoationary tilings}
	Let $\sigma$ be an irreducible incommensurable substitution scheme, and let $\PP$ be a legal patch. Then every stationary tiling $\SS\in\X_\sigma^F$ contains a translated copy of $\alpha\PP$ for some scale $\alpha>0$.
\end{lem}

\begin{proof}
	Since $\PP$ is legal, there is $t\in\R^+$ so that $\PP$ is a sub-patch of $F_t(T_i)$ for some $T_i\in\tau_\sigma$. Since $\PP$ contains finitely many tiles, there are $t_{\min}<t\le t_{\max}$ so that every element of $\{F_t(T_i):\,t_{\min}<t\le t_{\max}\}$ contains a rescaled copy of $\PP$ as a sub-patch. We remark that in $F_{t_{\max}}(T_i)$ the dilation of $\PP$ is such that the tiles of maximal volume are of volume $1$, and that in $F_{t_{\min}}(T_i)$ an ancestor of at least one of the tiles of the patch $\PP$ appears at volume exactly $1$.
	
	Let $\SS\in\X_\sigma^F$ be a stationary tiling for which $F_s(\SS)=\SS$. Fix $m\in\N$ so that $ms>t_{\max}$. By Theorem \ref{thm:The_scales_are_dense_in_SS} the tiles in $\SS$ appear in a dense set of scales within the intervals of legal scales. In particular, by irreducibility, there exists a tile $T$ in $\SS$ so that $F_{ms-t_{\max}}(T)$ contains a translated copy of $\alpha T_i$ for some $\alpha\in(e^{t_{\min}-t_{\max}},1]$. This can be deduced by considering a segment associated with the interval of scales $(e^{t_{\min}-t_{\max}},1]$, which is a subset of an edge of the associated graph $G_\sigma$ that terminates at vertex $i\in\mathcal{V}$. This segment can be mapped to another segment on the edges of $G_\sigma$ by choosing a ``reverse walk'' of length $ms-t_{\max}$. Any point in the resulting segment corresponds to a choice of some rescaled prototiles, that is, some tile of legal type and scale. Since tiles appear in $\SS$ in a dense set of scales within all legal scales, it is guaranteed that indeed there is a tile $T$ in $\SS$ which is a copy of $\alpha T_i$ with $\alpha\in(e^{t_{\min}-t_{\max}},1]$.
	
	It follows that 
	$F_{ms}(T)=F_{t_{\max}}\left(F_{ms-t_{\max}}(T)\right)$ contains $F_{t_{\max}}(\alpha T_i)=F_{t_{\max}-\log(1/\alpha)}(T_i)$ as a sub-patch. Clearly
	\begin{equation*}
	t_{\min}\le t_{\max}-\log(1/\alpha)\le t_{\max},
	\end{equation*} 
	and so $F_{ms}(T)$ contains a rescaled copy of $\PP$. Since $F_{ms}(T)$ is a patch in $F_{ms}(\SS)=\SS$, the result follows. 
\end{proof}

Let $\SS\in\X_\sigma^F$ be a stationary tiling with $\SS=\bigcup_{k=0}^{\infty}F_{ks}(T)$. For any fixed patch $\PP_0$ in $\SS$, there exists a smallest $k\ge1$ for which $\PP_0$ is a sub-patch of $F_{ks}(T)$. The patch $\PP_0$ can also be considered in the context of the continuous family $\{F_t(T):\,t\in\R^+\}$. As in the proof of Lemma \ref{lem: legal patches appear in stoationary tilings}, we define a maximal non-trivial interval $t_{\PP_0}:=(t_{\min},t_{\max}]$ so that $ks\in t_{\PP_0}$, and such that for any $t\in t_{\PP_0}$ the patch $F_t(T)$ contains a rescaled copy of $\PP_0$ as a sub-patch. If the original patch $\PP_0$ is thought of as being of scale $1$, then these rescaled copies of $\PP_0$ in  $\{F_t(T):\,t\in t_{\PP_0}\}$ appear in scales within an interval $I'_{\PP_0}$ of scales that contains $1$, and by left continuity also some left neighborhood of $1$. 

Now let $\PP$ be a patch in $\SS$, consider all of its translated copies in $\SS$ and let $\left(\PP_j\right)_{j\ge1}$ be an enumeration of them. Each $\PP_j$ is a patch in $\SS$ and has its own interval of times $t_{\PP_j}$ and interval of scales $I'_{\PP_j}$ as described above. Clearly all the intervals of the form $I'_{\PP_j}$ have the same maximum, and the associated rescaled copies of these patches have maximal tiles of volume $1$. Also note that since there are only finitely many tiles in a patch, and since tile scales are bounded, the intervals $I'_{\PP_j}$ are all bounded uniformly from below. We can thus define  the \textit{interval of legal scales} for the patch $\PP$ in $\SS$ as the union of all these intervals, and denote it by  $I_{\PP}$. Clearly $I_{\PP}$ contains $1$ and a left neighborhood of $1$.

\begin{lem}\label{thm:The_scales_for_patches_are_dense}
	Let $\sigma$ be an irreducible incommensurable substitution scheme, and let $\SS\in\X_\sigma^F$ be a stationary tiling. Let $\PP$ be a patch in $\SS$, and let $I_\PP=(\beta^{\min}_\PP,\beta^{\max}_\PP]$ be the interval of scales in which $\PP$ appears in $\SS$. Then the set 
	\[
	\emph{Scales}(\SS,\PP):= \left\{\alpha:\ \SS \text{ contains a copy of } \alpha\PP \right\}\]
	is dense in $I_\PP$.
\end{lem}

\begin{proof}
	Let $\left(\PP_j\right)_{j\ge1}$ be an enumeration of translated copies of $\PP$ in $\SS$. Fix a patch $\PP_j$, and as in the discussion above let $t_{\PP_j}$ and   $I'_{\PP_j}$ be the associated intervals of times and of scales, respectively. 
	
	Assume $s\in\R^+$ is such that $F_s(\SS)=\SS$ and write $t_{\PP_j}=(t_{\min},t_{\max}]$. As in the proof of Lemma \ref{lem: legal patches appear in stoationary tilings}, if we pick $m\in\N$ so that $ms>t_{\max}$, then $F_{ms}(\SS)=\SS$ contains a rescaled copy of $\PP_j$. In fact, it follows from the proof of Lemma \ref{lem: legal patches appear in stoationary tilings}, that $\SS$ contains rescaled copies of $\PP_j$ with scales which are dense in the interval $I'_{\PP_j}$. Since $I_{\PP}$ is defined as the union of the intervals $I'_{\PP_j}$, the proof is complete.
\end{proof}

\begin{cor}\label{cor: every stationary tiline contains rescaled copies of legal patches}
	Every stationary tiling contains rescaled copies of every legal patch, and the patches appear in a dense set of scale.
\end{cor}

It  will follow from Corollary \ref{cor: patch frequency for all tilings in the space} that the above holds for any tiling  $\TT\in\X_\sigma^F$.

\section{The tiling dynamical system}\label{sec:Dynamics}

A \textit{dynamical system} is a pair $\left(X,G\right)$ where $X$
is a topological space and $G$ is a group that
acts on $X$. The \textit{orbit} of a point $x\in X$ is the set
$\OO(x):=\left\{ g.x\,:\,g\in G\right\} $. A \textit{subsystem}
of $\left(X,G\right)$ is a closed, non-empty and $G$-invariant set
$Y\subset X$. Examples of subsystems include orbit closures $\overline{\OO(x)}$,
for $x\in X$. 

We study the dynamical system $\left(\X_\sigma^F,\mathbb{R}^{d}\right)$,
where the group $\R^{d}$ acts on tilings of $\R^{d}$ by translations,
and the topology on $\X_\sigma^F$ is determined by the metric $D$, defined
in \eqref{eq:CF-metric}. Traditionally, we call it the \textit{tiling dynamical system}.

\subsection{Minimality} 
Recall that a dynamical system $\left(X,G\right)$ is called \textit{minimal}
if it contains no proper subsystems, or equivalently, if every orbit is dense. 

\begin{thm}\label{thm: tiling space is minimal}
	Let $\sigma$ be an irreducible incommensurable substitution scheme in $\R^d$. Then the dynamical system $\left(\X_\sigma^F,\mathbb{R}^{d}\right)$ is minimal.
\end{thm}

The following result is a key step in the proof of Theorem \ref{thm: tiling space is minimal}.
\begin{lem}
	\label{lem:Stationary tilings has dense orbit}
	Let $\sigma$ be an irreducible incommensurable substitution scheme in $\R^d$, and let $\SS=\bigcup_{k=0}^{\infty}F_{ks}(T_i)\in\X_\sigma^F$ be
	a stationary tiling. Then $\OO(\SS)$ is dense in $\X_\sigma^F$.
\end{lem}

\begin{proof}
	Let $\varepsilon>0$. We need to show that for every $\TT\in \X_\sigma^F$, there
	exists $u\in\R^{d}$ so that 
	\begin{equation}\label{eq: S-u and T are close} 
	D\left(\TT, \SS-u\right)<\varepsilon.
	\end{equation}
	First, in view of Proposition \ref{prop:Equivalent def for X^F_sigma}, the tiling $\TT$ may be described as the limit $\lim_{\ell\to\infty}\PP_\ell-v_\ell$, where $\PP_\ell\in\PPP_\sigma$ and $v_\ell\in\R^d$. Let $\PP_\ell\in\PPP_\sigma$ and $v_\ell\in\R^d$ be such that 
	\begin{equation}\label{eq: P-v and T are close}
	D(\TT,\PP_\ell-v_\ell)<\varepsilon/2.	
	\end{equation}
	Note that in particular, the support of $\PP_\ell-v_\ell$ contains the ball $B(0,2/\varepsilon)$.
	
	Write $\PP_\ell=F_a(T_j)$ for some $T_j\in\tau_\sigma$ and $a\in\R^+$. Let $b\in\SSS_{i\rightarrow j}$, 
	and let $s_{1}<s_{2}<\ldots$ be an increasing enumeration of the set $\SSS_{i\rightarrow i}$. By Lemma \ref{lem:lengths of loops are getting mroe and more dense}, for any $\delta>0$ there exists $M$ so that $s_{m+1}-s_{m}<\delta$,
	for every $m\ge M$. Choose a minimal $k\in\N$ such that $\left(ks-b-s_{M}\right)-a\ge0$.
	Since $s_m\rightarrow\infty$ and the gap between two sequential elements in $\SSS_{i\rightarrow i}$
	with indices greater than $M$ is smaller than $\delta$, there exists
	some $m\ge M$ for which
	\begin{equation*}\label{eq:in proving dense orbit of stationary (internal eq 1')}
	0\le a- \left(ks-b-s_{m}\right)<\delta
	\end{equation*}
	holds. Since the patch $\PP_\ell$ has finitely many tiles, for small enough values of $\delta$ the patch $\PP_\ell$ is a small inflation of the patch $F_{ks-b-s_{m}}\left(T_{j}\right)$, and by taking small values of $\delta$ the inflation can be made arbitrarily small. In particular, by appropriate choices of $\delta$ there exist choices of $m$ for which the two patches have arbitrarily small Hausdorff distance, and this is also true for their translations by $v_\ell$. Since the support of $\PP_\ell-v_\ell$ contains $B(0,2/\varepsilon)$, it follows that
	\begin{equation*}
	D(\PP_\ell-v_\ell, F_{ks-b-s_{m}}\left(T_{j}\right)-v_\ell)<\varepsilon/2.	
	\end{equation*}  
	
	By our choice of $b$ and $s_{m}$ we have $b+s_{m}\in \SSS_{i\rightarrow j}$,
	and so $F_{b+s_{m}}\left(T_{i}\right)$ contains a translated copy of $T_j$. Hence $F_{ks}\left(T_{i}\right)=F_{ks-b-s_{m}}\left(F_{b+s_{m}}\left(T_{i}\right)\right)$
	contains a translated copy of the patch $F_{ks-b-s_{m}}\left(T_{j}\right)$ as a sub-patch.
	In turn, $F_{ks}\left(T_{i}\right)-v_\ell$ contains a translated copy of the patch $F_{ks-b-s_{m}}\left(T_{j}\right)-v_\ell$ as a sub-patch. Therefore, 
	there exists some $w_{\ell}\in\R^{d}$ with 
	\[
	D\left(\PP_\ell-v_\ell, F_{ks}(T_{i})-v_\ell-w_{\ell}\right)<\varepsilon/2.
	\]
	Then setting $u:=v_{\ell}+w_{\ell}$, we have 
	\begin{equation} \label{eq: P-v and S-u are close}
		D\left(\PP_\ell-v_\ell, \SS-u\right)<\varepsilon/2.
	\end{equation} 
	Combining \eqref{eq: P-v and T are close} and \eqref{eq: P-v and S-u are close}, the triangle inequality implies \eqref{eq: S-u and T are close}, finishing the proof.
\end{proof}

\begin{proof}[Proof of Theorem \ref{thm: tiling space is minimal}]
	Given any $\TT_1,\TT_2\in \X_\sigma^F$, we show that $\TT_2\in \overline{\OO(\TT_1)}$. Let $\SS=\bigcup_{k=0}^{\infty}F_{ks}(T_i)$ be a stationary tiling, for some $s>0$ and $T_i\in\tau_\sigma$. By Lemma \ref{lem:Stationary tilings has dense orbit}, $\TT_2\in \overline{\OO(\SS)}$ and hence it suffices to show that $\SS\in \overline{\OO(\TT_1)}$. We show that for every $\varepsilon>0$ there exists $u\in\R^d$ such that 
	$D(\TT_1-u,S)<\varepsilon$. 
	
	Let $\varepsilon>0$. Fix $m\in\N$ such that $\supp(F_{ms}(T_i))\supset B(0,3/\varepsilon)$. By left-continuity (see Proposition \ref{prop:F_t is left-continuous}), there exists $0<\delta<\log(\varepsilon/2)$ so that for every $0\leq\eta<\delta$ the support of $F_{ms-\eta}(T_i)$ contains $B(0,2/\varepsilon)$ and
	\begin{equation}\label{eq:X^F_sigma is minimal (internal 1)}
	D\left(F_{ms-\eta}\left(T_{i}\right), F_{ms}\left(T_{i}\right)\right)<\varepsilon/2.
	\end{equation}
	We first show that there exists some $R=R(\varepsilon)>0$ so that for every $z\in\R^d$ the patch $[B(z,R)]^\SS$ contains a translated copy of the patch $F_{ms-\eta}(T_i)$, for some $0\leq\eta<\delta$. 
	By Theorem \ref{thm:The_scales_are_dense_in_SS}, in particular, there exists $K>0$ so that for every tile $T$ is $\SS$ the set 
	\[\left\{\alpha: \text{the patch supported on } e^{Ks}T \text{ in } \SS \text{ contains a copy of } \alpha T_i \right\}\]
	is $\delta$-dense in the interval $(\beta_i^{\min},1]$. In particular, for every tile $T$ in $\SS$, the patch supported on $e^{Ks}(T)$ contains a copy of $\alpha T_i$ for $\alpha\in(1-\delta,1]$, and hence the patch supported on $e^{(K+m)s}(T)$ contains a copy of $F_{ms-\eta}(T_i)$, where here $\eta=1-\alpha\in[0,\delta)$. Set 
	\[R=R(\varepsilon):=e^{(K+m)s}\cdot\max_j\{{\rm diameter}(T_j)\}.\] 
	Then every patch $[B(z,R)]^\SS$, for $z\in\R^d$, contains a translated copy of the patch $F_{ms-\eta}(T_i)$, for some $0\leq\eta<\delta$.
	
	Invoking Lemma \ref{lem:Stationary tilings has dense orbit} once again, we find some $z\in\R^d$ so that $D(\SS-z,\TT_1)<1/R$. This implies that $[B(z,R)]^\SS$ and $[B(0,R)]^{\TT_1}$ are $1/R$-close, with respect to the Hausdorff metric. Since $[B(z,R)]^\SS$ contains a copy of $F_{ms-\eta}(T_i)$, for $0\leq\eta<\delta$, the patch $[B(0,R)]^{\TT_1}$ contains a patch $\PP$ whose Hausdorff distance is less than $1/R$ from $F_{ms-\eta}(T_i)$. Recall that $F_{ms-\eta}(T_i)$ contains a ball of radius $2/\varepsilon$, then in particular there exists some $u\in B(0,R)$, with $B(u,2/\varepsilon)\subset B(0,R)$, such that 
	\begin{equation}\label{eq:X^F_sigma is minimal (internal 2)}
	D(\TT_1-u,F_{ms-\eta}(T_i))\le D(\PP-u, F_{ms-\eta}(T_i))<\varepsilon/2.	
	\end{equation}
	Combining  \eqref{eq:X^F_sigma is minimal (internal 1)} and \eqref{eq:X^F_sigma is minimal (internal 2)} we obtain $D(\TT_1-u,F_{ms}(T_i))<\varepsilon$. But since $\SS\supset F_{ms}(T_i)$, this implies $D(\TT_1-u,\SS)\le D(\TT_1-u,F_{ms}(T_i))<\varepsilon$, and the proof is complete. 
\end{proof}

\subsection{Geometric interpretation of dynamical properties}
A finite local complexity tiling $\TT$ of $\R^{d}$ is called \textit{repetitive} if for every $r>0$ there is a radius $R:=R(r)>0$,
so that for every $x\in\R^{d}$ the patch $\left[B(x,R)\right]^{\TT}$
contains translated copies of all the patches of the form $\left[B(y,r)\right]^{\TT}$,
$y\in\R^{d}$ (see e.g. \cite[Section 5.3]{BaakeGrimm}, 
\cite{Solomyak-dynamics}, and also \cite{Frettloh Richards} and \cite{Lagarias-Pleasants} for a clearer distinction between similar, common, definitions of repetitivity).

Recall that a set $S\subset\R^{d}$ is called \textit{relatively dense}, or \textit{syndetic} in the topological dynamics setup, if there exists some $R>0$ so that $S$ intersects every ball of radius $R$ in $\R^d$. A relatively dense set with parameter $R$ is called \textit{$R$-dense}. Note that repetitivity means that for every $r>0$ and every $y\in\R^{d}$, the set 
\[
Z_{y}:=\left\{ t\in\R^{d}\,:\,\left[B(y,r)\right]^{\TT}-y=\left[B(t,r)\right]^{\TT}-t\right\} 
\]
of \textit{return times} to $\left[B(y,r)\right]^{\TT}$ is relative dense. Since $\left[B(y,r)\right]^{\TT}-y=\left[B(0,r)\right]^{\TT-y}$,
the set $Z_{y}$ is simply the collection of $t\in\R^{d}$ for which $\TT-y$ and $\TT-t$ agree on the ball of radius $r$ around the origin.

An immediate corollary of Theorems \ref{thm: tiling space is minimal} and \ref{thm:The_scales_are_dense_in_SS} is that if $\sigma$ is incommensurable then every $\TT\in \X_\sigma^F$ is of infinite local complexity. Since repetitivity is clearly impossible for such tilings, we consider the following suitable variant.

\begin{definition}
	\label{def:almost repetitive}Given $\varepsilon>0$, a tiling $\TT\in \X_\sigma^F$
	is called \textit{$\varepsilon$-repetitive} if there exists some
	$R=R(\varepsilon)>0$ such that for every $y\in\R^{d}$ the set 
	\[
	A_{y}:=\left\{ x\in\R^{d}\,:\,D(\TT-y,\TT-x)<\varepsilon\right\}
	\]
	of return times to the $\varepsilon$-neighborhood of $\TT-y$ is $R$-dense. $\TT$ is \textit{almost repetitive} if it is
	$\varepsilon$-repetitive for every $\varepsilon>0$ (compare \cite[Definitions 2.13 and 3.5]{Frettloh Richards}).

\end{definition}

Recall that two tilings $\TT,\TT'$ of $\R^{d}$ are called \textit{locally
	indistinguishable} (LI)
if $\TT$ and $\TT'$ have the same collection of patches of compact
support (see e.g. \cite[Definition 5.5]{BaakeGrimm}, \cite[Definition 1.6]{Lagarias-Pleasants}).
This notion induces an equivalence relation on tilings of $\R^{d}$
and one denotes by LI($\TT$) the equivalence class of the tiling
$\TT$. As in the case of repetitivity, we consider the suitable variant of local indistinguishability.
\begin{definition}
	\label{def:ALI}Given $\varepsilon>0$, we say that two tilings $\TT, \TT'\in \X_\sigma$
	are \textit{$\varepsilon$-locally indistinguishable} ($\varepsilon$-LI),  if for
	every $y\in\R^{d}$ there exist $x_{1},x_{2}\in\R^{d}$ such that
	\[
	D\left(\TT-y,\TT'-x_{1}\right),D\left(\TT'-y,\TT-x_{2}\right)<\varepsilon.
	\]
	$\TT,\TT'\in \X_\sigma^F$ are called \textit{almost locally indistinguishable}
	(ALI),  if they are
	$\varepsilon$-LI for every $\varepsilon>0$ (compare \cite[Definition 3.9]{Frettloh Richards}).
	We denote by ALI($\TT$) the collection of tilings in $\X_\sigma^F$
	that are \text{ALI} with $\TT$.
\end{definition}

A standard theorem in the theory of tilings, or discrete patterns,
states that repetitivity, having a closed LI equivalence
class, and having a minimal orbit closure with respect to translations
are all equivalent properties (see e.g. \cite[Theorem 5.4]{BaakeGrimm},
\cite[Theorem 3.2]{Lagarias-Pleasants}). A variant of this theorem
for $r$-separated point sets in a setting that resembles the one studied here is given
in \cite[Theorem 3.11]{Frettloh Richards}.  Our proof of the multiscale substitution tilings variant is similar, and included here due to small differences in the definitions.  
\begin{thm}
	\label{thm:almost repetitive iff minimal iff ALI} 	Let $\sigma$ be an irreducible incommensurable substitution
	scheme and let $\TT\in \X_\sigma^F$. The following are
	equivalent:
\end{thm}

\begin{enumerate}
	\item The orbit closure $\overline{\OO(\TT)}:=\overline{\left\{ \TT+x\,:\,x\in\R^{d}\right\} }$
	of $\TT$ is minimal. 
	\item ALI($\TT$) is closed in $\left(\X_\sigma^F, D\right)$. 
	\item $\TT$ is almost repetitive.
\end{enumerate}

\begin{proof}
	$(1)\Rightarrow(2):$ It suffices to show that $\overline{\OO(\TT)}=\text{ALI}(\TT)$.
	First, clearly $\overline{\OO(\TT)}\supset\mbox{ALI}(\TT)$, because if $\TT'\in\text{ALI}(\TT)$, putting $y=0$ in Definition \ref{def:ALI} implies that for every $k\in\N$ there exists some $z_{k}\in\R^{d}$ with $D\left(\TT-z_{k},\TT' \right)<1/k$,
	and so $\TT' \in\overline{\OO(\TT)}$. 
	
	Let $\TT'\in\overline{\OO(\TT)}$, $\varepsilon>0$ and $y\in\R^{d}$. Then  $\TT'-y\in\overline{\OO(\TT)}$, hence there exists some $x_{1}\in\R^{d}$ so that $D\left(\TT'-y,\TT-x_{1}\right)<\varepsilon$.
	By minimality $\TT$, and hence $\TT-y$, belong to $\overline{\OO(\TT')}$, so there exists
	some $x_{2}\in\mathbb{R}^{d}$ with $D\left(\TT-y,\TT'-x_{2}\right)<\varepsilon$.
	We conclude that $\TT'\in\text{ALI}(\TT)$.
	
	$(2)\Rightarrow(3):$ Assume that $(3)$ does not hold. Then there
	is some $\varepsilon>0$ for which $\TT$ is not $\varepsilon$-repetitive.
	Namely, there is some $y\in\R^{d}$ so that the set of return times
	to the $\varepsilon$-neighborhood of $\TT-y$ is not $R$-dense for any $R>0$. Then there is a sequence $x_{k}\in\R^{d}$ for which the sequence
	of patches $\left[B(x_{k},k)\right]^{\TT}$ does not contain
	a patch that, after translation, is $\varepsilon$-close to $\TT-y$. The inclusion 
	\[
	\overline{\OO(\TT)}\supset{\rm ALI}(\TT)\supset\OO(\TT),
	\]
	combined with the assumption that ALI($\TT$) is closed, implies that $\overline{\OO(\TT)}={\rm ALI}(\TT)$.
	
	Let $\TT'\in \X_\sigma^F$ be a limit of some subsequence of $\left[B(x_{k},k)\right]^{\TT}-x_{k}$.
	Then no translation of $\TT'$ is $\varepsilon$-close to $\TT-y$.
	Since $\TT'$ is a limit of the sequence $\TT-x_{k}$, then
	$\TT'\in\overline{\OO(\TT)}={\rm ALI}(\TT)$. This means that there is some
	$x\in\R^{d}$ so that $D\left(\TT-y,\TT'-x\right)<\varepsilon$,
	a contradiction. 
	
	$(3)\Rightarrow(1):$ Let $\TT'\in\overline{\OO(\TT)}$ and let $\varepsilon>0$, then there exist $y_{k}\in\R^{d}$ with $D\left(\TT-y_{k},\TT'\right)<1/k$.
	By $(3)$, $\TT$ is $\frac{\varepsilon}{2}$-repetitive. In
	particular, for $y=0$ in Definition \ref{def:almost repetitive},
	there is some $R>0$ for which the set 
	\[
	A_{0}=\left\{ x\in\R^{d}\,:\,D\left(\TT,\TT-x\right)<\varepsilon/2\right\} 
	\]
	is $R$-dense, and we may assume that $R>2/\varepsilon$. Fix $K\ge2R$.
	Since 
	\begin{equation}\label{eq: T-y_K and T' are close}
	D\left(\TT-y_{K},\TT'\right)<1/K\le1/2R,
	\end{equation}
	by the definition of $D$, the patches supported on the $2R$-ball around the origin in $\TT'$
	and in $\TT-y_{K}$ are $1/2R$-close. Since $A_{0}$ is $R$-dense
	and $R>2/\varepsilon$, there is some $x_{0}\in B(y_{K},R)$ so that $D\left(\TT,\TT-x_{0}\right)<\varepsilon/2$ and 
	\begin{equation}\label{eq: ball around x_0 is in ball around y_K}
	B(x_{0},2/\varepsilon)\subset B(y_{K},2R).
	\end{equation}
	Set $x_{1}=x_{0}-y_{K}$. Since $||x_{1}||<R$,
	combining \eqref{eq: T-y_K and T' are close} and \eqref{eq: ball around x_0 is in ball around y_K} implies that 
	\begin{equation*}
	D\left(\TT-x_{0},\TT'-x_{1}\right)=D\left(\TT-y_K-x_1,\TT'-x_{1}\right)<\varepsilon/2,
	\end{equation*}
	and by the triangle inequality this yields
	\[
	D\left(\TT,\TT'-x_{1}\right)\le D\left(\TT,\TT-x_{0}\right)+D\left(\TT-x_{0},\TT'-x_{1}\right)<\varepsilon.
	\]
	It follows that $\TT\in\overline{\OO(\TT')}$, and minimality is established.
\end{proof}

\begin{cor}
	Let $\sigma$ be an irreducible incommensurable substitution
	scheme. Then every $\TT\in \X_\sigma^F$ is almost repetitive, and every
	$\TT,\TT'\in \X_\sigma^F$ are ${\rm ALI}$-equivalent. 
\end{cor}

\subsection{Supertiles}
Consider an irreducible incommensurable substitution scheme $\sigma$ in $\R^d$. Minimality of $\left(\X_\sigma^F,\R^d\right)$, combined with the construction of stationary tilings, allows for the existence of the powerful hierarchical structure known as supertiles on all tilings in the multiscale tiling space $\X_\sigma^F$.

\begin{definition}\label{def: supertiles}
	Let $\SS\in\X_\sigma^F$ and $s\in\R^+$ be so that $F_s(\SS)=\SS$. For every $m\in\N$ we have $F_{ms}(\SS)=\SS$. We denote by $\SS^{-m}$ the tiling whose tiles are $\{e^{ms}T\: :\: T\text{ is a tile in } \SS \}$. These tiles are called \textit{order $m$ supertiles} of $\SS$, or simply \textit{$m$-supertiles}, and are denoted by $T^{(m)}$. The type of the supertile $e^{ms}T$ is inherited from the type of $T$, and for every $\ell\ge m$ every $\ell$-supertile can be decomposed into a union of $m$-supertiles. 
\end{definition}

As a consequence of minimality, every tiling $\TT \in \X^F_\sigma$ of $\R^d$ inherits a hierarchical structure of supertiles from any given stationary tiling. 
Given $\TT \in \X^F_\sigma$ and a stationary tiling $\SS$ with $F_s(\SS)=\SS$, for $s>0$, fix $m\in \N$ and $v_n\in\R^d$ so that $\TT = \lim_{n\to\infty} \SS + v_n$. Since $\partial \SS^{-m} \subset \partial \SS$ and $\partial \SS +v_n \to \partial \TT$, the sequence $\partial \SS^{-m} +v_n$ is convergent and the limit is a subset of $\partial \TT$, which defines $m$-supertiles of $\TT$.  Note that by \eqref{eq: flow on translation is translation of flow}, for any tiling in the $\R^d$-orbit of a stationary tiling $F_s(\SS)=\SS$, supertiles are of the form $e^{ms}T$ for some tile $T$ of legal type and scale.

\begin{example}
	Figure \ref{fig:Newtrianglessupertiles} illustrates tiles, thought of as $0$-supertiles, within $1$-supertiles and a single $2$-supertile, in a stationary tiling generated by the scheme on two triangles illustrated in Figure \ref{fig:triangle substitution rule}.
	
\end{example}

\begin{figure}[ht!]
	\includegraphics[scale=0.25]{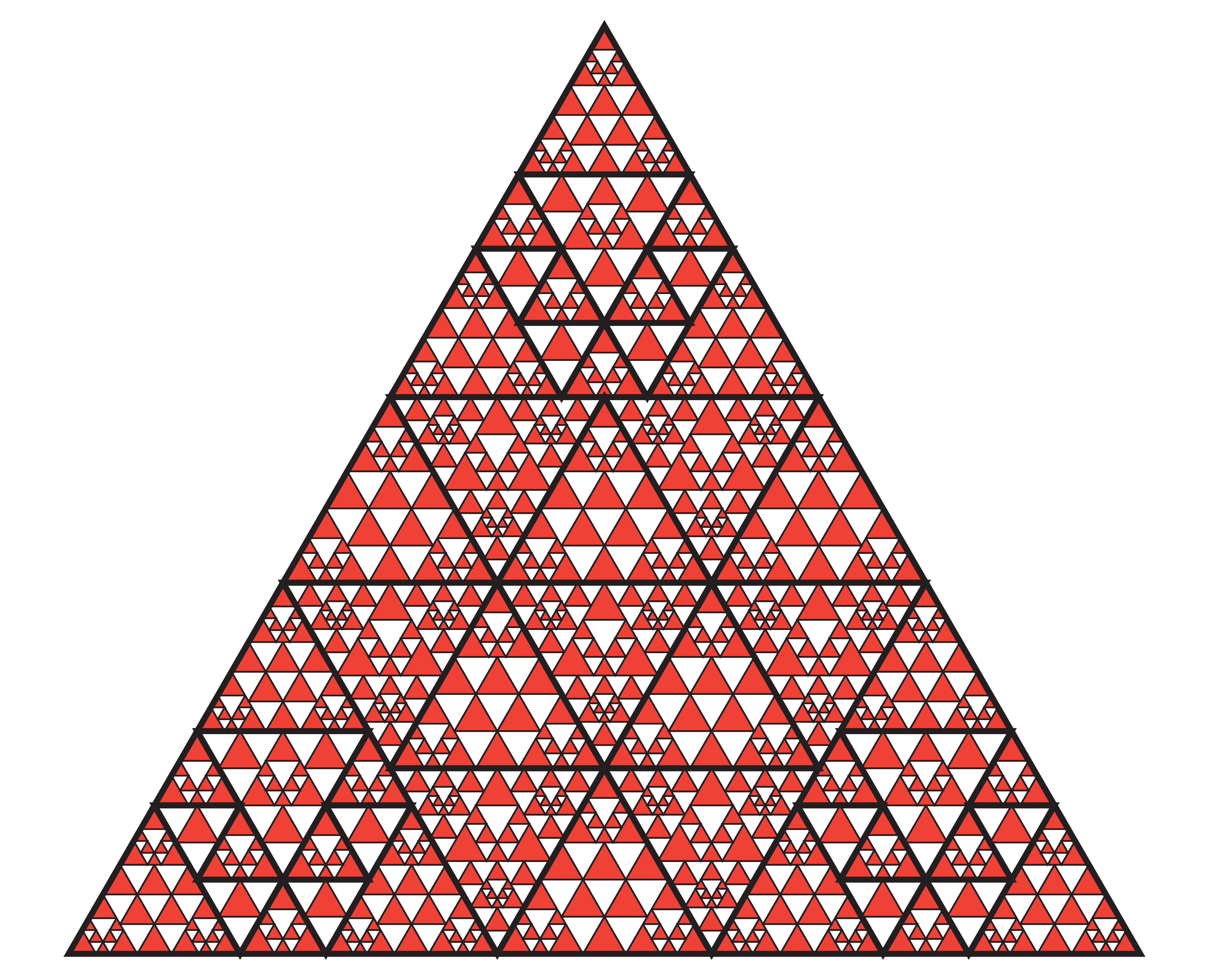}\caption{Tiles as $0$-supertiles, within $1$-supertiles (in bold boundaries), all within a single $2$-supertile. Compare also with Figure \ref{fig:triangle stationary}.}	\label{fig:Newtrianglessupertiles}	
\end{figure}

\section{Explicit tile counting formulas}\label{sec: Frequencies}

This section contains various explicit asymptotic counting formulas for tiles in incommensurable multiscale tilings, according to their types and scales. Results on the distribution of metric paths on incommensurable graphs \cite{Yotam graphs} play an important role.

\begin{definition}
	Let $G$ be a directed weighted graph with a set of vertices $\mathcal{V}=\left\{ 1,\ldots,n\right\} $.
	Let $i,j\in\mathcal{V}$ be a pair of vertices in $G$, and assume
	that there are $k_{ij}\geq0$ edges $\varepsilon_{1},\ldots,\varepsilon_{k_{ij}}$
	with initial vertex $i$ and terminal vertex $j$. The \textit{graph
		matrix function} of $G$ is the matrix valued function $M:\mathbb{C}\rightarrow M_{n}\left(\mathbb{C}\right)$, its $(i,j)$ entry the exponential polynomial
	\[
	M_{ij}\left(s\right)=e^{-s\cdot l\left(\varepsilon_{1}\right)}+\cdots+e^{-s\cdot l\left(\varepsilon_{k_{ij}}\right)}.
	\]
	If $i$ is not connected to $j$ by an edge, put $M_{ij}\left(s\right)=0$. Note that the restriction of $M$ to $\mathbb{R}$ maps $\R$ to real valued matrices.
\end{definition}

Note that given a substitution scheme $\sigma$, the $(i,j)$ entry of $M_\sigma(s)$, where $M_\sigma$ is the graph matrix function of the associated graph $G_\sigma$, is given by 
\begin{equation*}
\left(\alpha_{ij}^{(1)}\right)^s+\cdots+ \left(\alpha_{ij}^{(k_{ij})}\right)^s.
\end{equation*}

\subsection{The number of tiles of given types and scales}

First we need the following counting result on metric paths in graphs associated with incommensurable schemes.

\begin{thm}
	\label{thm: graph path counting in the case of multiscale}Let $\sigma$
	be a normalized irreducible incommensurable scheme in $\R^d$, and let $G_\sigma$ be the associated graph. Let $I$ be an interval
	that is contained in an edge $\varepsilon\in\mathcal{E}$ with initial
	vertex $h\in\mathcal{V}$, and assume that $I$ is of length $\delta>0$
	and of distance $\alpha_{0}$ from the vertex $h$. Then the number of
	metric paths in $G_\sigma$ of length $x$, with initial fixed vertex $i\in\mathcal{V}$,
	that terminate at a point in $I$, grows as 
	\[
	e^{-\alpha_{0}d}\frac{1-e^{-\delta d}}{d}q_{h}e^{dx}+o\left(e^{dx}\right),\quad x\rightarrow\infty,
	\]
	independent of $i$, where $q_{h}\cdot\boldsymbol{1}$ is column $h$
	of the rank 1 matrix
	\begin{equation*}
	Q_\sigma=\frac{{\rm adj}\left(I-M_\sigma\left(d\right)\right)}{-{\rm \tr}({\rm adj}\left(I-M_\sigma\left(d\right)\right)\cdot M_\sigma^{\prime}\left(d\right))},		
	\end{equation*} 
	with $\boldsymbol{1}=(1,\ldots,1)^T\in\R^n$, and  $M_\sigma$ is the graph matrix function of $G_\sigma$.
\end{thm}

\begin{proof}[Sketch of proof]
	The proof follows from a slight adjustment of the proof of Theorem 1 in \cite{Yotam graphs}, and so we only give an outline of the proof. A main tool is the Wiener-Ikehara Tauberian theorem, by which the exponential growth rate of our counting function follows from special properties of its Laplace transform, and in particular the location of its poles, see \cite{Montgomery Vaughan}. Roughly speaking, if there exists $\lambda\in\R$ for which the Laplace transform has a simple pole at $s=\lambda$, and there are no other poles in the half plane $\text{Re}(s)\ge\lambda$, then the counting function grows exponentially with exponent $\lambda$.
	
	The counting function we are interested in can be written as
	\begin{equation}\label{eq: counting function}
		B_{i,I}\left(x\right)=\sum_{\gamma\in\Gamma\left(i,h\right)}\chi_{(l\left(\gamma\right)+\alpha_{0},l\left(\gamma\right)+\alpha_{0}+\delta)}(x),
	\end{equation}
	where $\Gamma\left(i,h\right)$ is the set of paths in $G_\sigma$ with initial vertex
	$i$ and terminal vertex $h$. Note that although \eqref{eq: counting function} is written as if $I$ is assumed to be open, the arguments and results presented here are not changed if $I$ is assumed to contain one or both of its endpoints, as the Laplace transform is not changed if only countably many values of the counting function are varied. A direct computation, details of which can be found in \cite{Yotam graphs}, shows that the Laplace transform is
	\begin{equation*}\label{eq: Laplace tranform}
		\mathcal{L}\left\{ B_{i,I}\left(x\right)\right\} \left(s\right):=\int_{0}^{\infty}B_{i,I}\left(x\right)e^{-xs}dx=e^{-\alpha_{0}s}\frac{1-e^{-\delta s}}{s}\cdot\frac{\left({\rm adj}\left(I-M_\sigma\left(s\right)\right)\right)_{ih}}{\det\left(I-M_\sigma\left(s\right)\right)}.
	\end{equation*}

	As a result of Proposition \ref{prop: volumes sum up to 1}, since $G_\sigma$ is associated with a substitution scheme, the value of $\lambda$, which due to the form of the Laplace transform is the maximal real value for which the spectral radius of $M_\sigma(\lambda)$ is exactly $1$, is the dimension $d$, see \cite[Lemma 8.1]{Yotam Kakutani}. The Perron-Frobenius eigenvalue of the non-negative matrix $M_\sigma(d)$ is  indeed $1$, and a corresponding eigenvector can be chosen to be the vector of volumes of prototiles in $\tau_\sigma$. Since $\sigma$ is normalized,  this vector is $\boldsymbol{1}$, from which it follows that $Q_\sigma$ is a positive matrix. Since the columns of $Q_\sigma$ are multiples of the eigenvector $\boldsymbol{1}$, all the rows of $Q_\sigma$ are identical, from which the independence of $i$ stems in the statement of the result. 
	
	The rest of the proof concerns the careful analysis of the poles of the Laplace transform, and more specifically the zeros of the exponential polynomial $\det (I-M_\sigma(s))$, to which the bulk of \cite{Yotam graphs} is dedicated. The steps taken there provide proof that there are no poles in the half plane $\text{Re}(s)\ge d$, except for a simple pole located at $s=d$, establishing the properties required for the application of the Wiener-Ikehara theorem.
\end{proof}

\begin{thm}\label{thm: type and scale frequencies}
	Let $\sigma$
	be an irreducible incommensurable scheme in $\R^d$. For every
	$j=1,\ldots,n$ and $0\leq a<b\le 1$, the number of tiles of type $j$
	in $F_{t}(T_i)$ with scale in $[a,b]$, or equivalently with volume in $[a^d,b^d]$, grows as
	\[
	\varphi_{j,[a,b]}e^{dt}+o\left(e^{dt}\right),\quad t\rightarrow\infty,
	\]
	independent of $i$, where 
	$	\varphi_{j,[a,b]}:=\sum_{h=1}^{n}c_{hj,[a,b]}q_{h}$ and
	$q_h$ is as in Theorem \ref{thm: graph path counting in the case of multiscale},
	\begin{equation}\label{eq:a_hj}
	c_{hj,[a,b]}:=\frac{1}{d}\sum_{k=1}^{k_{hj}}\left(\alpha_{hj}^{\left(k\right)}\right)^{d}\left(\left(\eta_{hj}^{\left(k\right)}(a)\right)^{-d}-\left(\mu_{hj}^{\left(k\right)}(b)\right)^{-d}\right),
	\end{equation}
	$\alpha_{hj}^{\left(k\right)}$ are the constants of substitution as in Definition \ref{def: substitution scheme}, and 
	\begin{equation}\label{eq:eta_and_mu_hj}
	\eta_{hj}^{\left(k\right)}(a):=\max\left\{ a,\alpha_{hj}^{\left(k\right)}\right\} ,\,\,\mu_{hj}^{\left(k\right)}(b):=\max\left\{ b,\alpha_{hj}^{\left(k\right)}\right\}.
	\end{equation}
\end{thm}

Although the formulas given in Theorem \ref{thm: type and scale frequencies} seem complicated, in fact they are easily evaluated in examples, as will be demonstrated in Example \ref{ex: frequencies for triangles} at the end of this section.

\begin{proof}
	Tiles of type $j$ in $F_{t}\left(T_i\right)$ correspond to metric paths of length $t$ that terminate on edges with end point $j$. Let $\varepsilon\in\mathcal{E}$ be such an edge in $G_\sigma$, and assume $\varepsilon$ is of length $l\left(\varepsilon\right)=\log\frac{1}{\alpha}$, corresponding to tiles of scale in $(\alpha,1]$. Let $\delta$ and $\alpha_0$ be as in the statement of Theorem \ref{thm: graph path counting in the case of multiscale}.	There are three distinct cases:
	\begin{enumerate}
		\item If $\alpha<a<b$ then 
		\begin{equation*}
		\begin{split}
		\alpha_0&=\log\tfrac{1}{\alpha}-\log\tfrac{1}{a}=\log\tfrac{a}{\alpha}\\
		\delta&=\log\tfrac{1}{\alpha}-\log\tfrac{1}{b}-\alpha_0=\log\tfrac{b}{a}.
		\end{split}
		\end{equation*}
		Therefore 
		\begin{equation*}
		e^{-\alpha_{0}d}\frac{1-e^{-\delta d}}{d}=\frac{1}{d}\alpha^d\left(a^{-d}-b^{-d}\right).
		\end{equation*}
		\item If $a\le \alpha<b$ then $\alpha_0=0$ and
		\begin{equation*}
		\delta=\log\tfrac{1}{\alpha}-\log\tfrac{1}{b}-\alpha_0=\log\tfrac{b}{\alpha}.
		\end{equation*}
		Therefore 
		\begin{equation*}
		e^{-\alpha_{0}d}\frac{1-e^{-\delta d}}{d}=\frac{1}{d}\alpha^d\left(\alpha^{-d}-b^{-d}\right).
		\end{equation*}
			\item If $a<b\le \alpha$ then $\alpha_0=0$ and $\delta=0$ and there is no contribution to the counting.
	\end{enumerate}

	Summing the contributions from all edges in $G_\sigma$ that terminate at the vertex $j$, we establish the required formula.
\end{proof}

Note that by Theorem \ref{thm: graph path counting in the case of multiscale} and its proof, the exact same formula arises for scales in the interval $[a,b], (a,b]$ or $[a,b)$. In particular, $\varphi_{j,[a,a]}=0$ and we deduce the following corollary.

\begin{cor}\label{cor: number of tiles of fixed scale in o(e^{dt})}
	For every $j=1,\ldots,n$ and $\alpha\in(\beta_j^{\min},1]$, the number of tiles of type $j$
	and scale exactly $\alpha$ in $F_{t}(T_i)$ is $o\left(e^{dt}\right)$ as $t$ tends to infinity.
\end{cor}

The results stated above can be described in terms of supertiles.

\begin{cor}\label{cor: counting in supertiles} 
	Under the notation of Theorem \ref{thm: type and scale frequencies}, for every tile $T$ of legal type and scale (see Definition \ref{def:possible_scales}), the number of tiles of type $j$ in $F_t(T)$ with scale in $[a,b]$ grows as
	\begin{equation*}
	\varphi_{j,[a,b]}\vol\left(e^{t}T\right)+o\left(\vol\left(e^{t}T\right)\right),\quad t\rightarrow\infty.
	\end{equation*}
	In particular, for every sequence $\left(T^{(m)}\right)_{m\ge0}$  of supertiles of growing order in $\TT\in\X_\sigma^F$, the number of tiles of type $j$ with scale in $[a,b]$ in the patch $[T^{(m)}]^\TT$,  grows as
	\begin{equation*}
		\varphi_{j,[a,b]}\vol\left(T^{(m)}\right)+o\left(\vol\left(T^{(m)}\right)\right),\quad m\rightarrow\infty.
	\end{equation*} 
\end{cor}

\begin{proof}
	If $T$ is a rescaled copy of $\alpha T_i$, then 
	\[
	F_t(T)=F_t(\alpha T_i)=F_{t-\log(1/\alpha)}(T_i),
	\]
	and by Theorem \ref{thm: type and scale frequencies} the number of tiles we are interested in grows as 
\[
\varphi_{j,[a,b]}e^{d(t-\log(1/\alpha))}+o\left(e^{dt}\right),\quad t\rightarrow\infty.
\]
But $e^{d(t-\log(1/\alpha))}=e^{dt}\alpha=\vol(e^tT)$, and the result follows.
\end{proof}

We remark that Corollary \ref{cor: counting in supertiles} gives an alternative proof that the scales in which tiles appear in incommensurable tilings are dense within the intervals of legal scales, see Theorem \ref{thm:The_scales_are_dense_in_SS}. For the next results, recall that for a tiling $\TT$ and a set $B\subset\R^d$ we denote $[B]^{\TT}$ to be the patch of all tiles in $\TT$ that intersect $B$.

\begin{cor}\label{cor: frequency of types}
	Let $\left(T^{(m)}\right)_{m\ge0}$  be a sequence of supertiles of growing order in $\TT\in\X_\sigma^F$. The number of tiles of type $j$ in $[T^{(m)}]^\TT$
	grows as 
	\[
	\varphi_{j,(\beta_j^{\min},1]}\vol\left(T^{(m)}\right)+o\left(\vol\left(T^{(m)}\right)\right),\quad m\rightarrow\infty
	\]
	with 
	\begin{equation}\label{eq: number of tiles of type j}
	\varphi_{j,(\beta_j^{\min},1]}=\sum_{h=1}^{n}\frac{1}{d}\sum_{k=1}^{k_{hj}}\left(1-\left(\alpha_{hj}^{\left(k\right)}\right)^{d}\right)q_{h}.
	\end{equation}
	In particular, the total number of tiles in $[T^{(m)}]^\TT$ grows as 
	\begin{equation}\label{eq: total number of tiles in supertiles}
	\sum_{j=1}^{n}\varphi_{j,(\beta_j^{\min},1]}\vol\left(T^{(m)}\right)+o\left(\vol(T^{(m)})\right),\quad t\rightarrow\infty,
	\end{equation}
	independent of $i$.
\end{cor}

\begin{proof}
	Simply note that the parameters in \eqref{eq:eta_and_mu_hj} satisfy $\eta^{(k)}_{hj}(0)=\alpha^{(k)}_{hj}$ and $\mu^{(k)}_{hj}(1)=1$, for all $i,j,k$. 
\end{proof}

Manipulations of the formulas given above allow for computations of other quantities. For example, we give below the asymptotic formula for the relative number of tiles of a given type and scales within a fixed interval of scales, within all tiles in supertiles. Various other quantities can be similarly derived. 

\begin{cor}\label{cor: relative number of tiles of type and scale}
	For every
	$j=1,\ldots,n$ and $0\leq a<b\le 1$, the relative number of tiles of type $j$
	with scale in $[a,b]$ within the total number of tiles in a patch supported on an $m$-supertile, tends to
	\begin{equation*}
		\frac{\varphi_{j,[a,b]}}{\sum_{j=1}^n\varphi_{j,(\beta_j^{\min},1]}}+o\left(1\right),\quad m\rightarrow\infty,
	\end{equation*}
	where $\varphi_{j,[a,b]}$ and $\varphi_{j,(\beta_j^{\min},1]}$ are as in Theorem \ref{thm: graph path counting in the case of multiscale} and Corollary \ref{cor: frequency of types}, respectively.
\end{cor}

\begin{example}\label{ex: frequencies for triangles}
	Let $\sigma$ be the substitution scheme in $\R^2$  on $T_1=U$ and $T_2=D$ the two equilateral triangles, as described in Figure \ref{fig:triangle substitution rule}. By a direct computation 
	\begin{equation*}
	Q_\sigma=\frac{{\rm adj}\left(I-M_\sigma\left(2\right)\right)}{-{\rm \tr}({\rm adj}\left(I-M_\sigma\left(2\right)\right)\cdot M_\sigma^{\prime}\left(2\right))}=\frac{1}{\frac{4}{25}\log 2+\frac{1}{4}\log 5}\begin{pmatrix}
	\frac{1}{4} & \frac{8}{25} \\ \frac{1}{4} & \frac{8}{25}
	\end{pmatrix}
	\end{equation*}
	and so 
	\begin{equation*}
	q_1=\frac{1}{4}\cdot \frac{1}{\frac{4}{25}\log 2+\frac{1}{4}\log 5}\quad\text{and}\quad q_2=\frac{8}{25}\cdot \frac{1}{\frac{4}{25}\log 2+\frac{1}{4}\log 5}.
	\end{equation*}

	Consider rescaled copies of $U$ (tiles of type $1$) with scales within the interval $\left[\frac{3}{5},\frac{4}{5}\right]$. Since $\alpha\le\frac{1}{2}$ for any scale $\alpha$ in which tiles appear in the substitution scheme $\sigma$, we have $\eta=\frac{3}{5}$ and $\mu=\frac{4}{5}$ in \eqref{eq:eta_and_mu_hj}. Plugging this into \eqref{eq:a_hj} we get
	
	\begin{equation*}
	c_{11,\left[\frac{3}{5},\frac{4}{5}\right]}=\frac{119}{288}\quad\text{and}\quad c_{21,\left[\frac{3}{5},\frac{4}{5}\right]}=\frac{175}{1152},	
	\end{equation*}
	
	and so\begin{equation*}
	\varphi_{1,\left[\frac{3}{5},\frac{4}{5}\right]}=c_{11,\left[\frac{3}{5},\frac{4}{5}\right]}q_1+c_{21,\left[\frac{3}{5},\frac{4}{5}\right]}q_2\approx 0.296.
	\end{equation*}

	We now use the above computation to derive a couple of results about tiles that are rescaled copies of $U$ with scales in the interval $\left[\frac{3}{5},\frac{4}{5}\right]$. First, according to Corollary \ref{cor: counting in supertiles}, for every tiling $\TT\in\X_\sigma^F$ and every sequence $\left(T^{(m)}\right)_{m\ge0}$  of supertiles of growing order in $\TT$, the number of such tiles in the patch supported on $T^{(m)}$ grows approximately as
	\begin{equation*}
	0.296\cdot\vol\left(T^{(m)}\right)+o\left(\vol\left(T^{(m)}\right)\right),\quad m\rightarrow\infty.
	\end{equation*}
	Next, using \eqref{eq: number of tiles of type j} we have
	\begin{equation*}
	\varphi_{1,(\beta_1^{\min},1]}\approx 2.016\quad\text{and}\quad \varphi_{2,(\beta_2^{\min},1]}\approx 1.841,
	\end{equation*}

	and by Corollary \ref{cor: relative number of tiles of type and scale} we deduce that the relative number of such tiles within the total number of tiles in a patch supported on an $m$-supertile, tends to 
	\begin{equation*}
	\frac{\varphi_{1,\left[\frac{3}{5},\frac{4}{5}\right]}}{\varphi_{1,(\beta_1^{\min},1]}+\varphi_{2,(\beta_2^{\min},1]}}	+o(1)\approx 0.076+o(1),\quad m\rightarrow\infty.
	\end{equation*}
	
\end{example}

\subsection{The volume occupied by tiles of given types and scales}

For the main results of this section we turn to probabilistic results on graphs. 

\begin{thm}
	\label{thm: prob graph path counting in the case of multiscale-1}Let $\sigma$
	be a normalized irreducible incommensurable scheme in $\R^d$, and let $G_\sigma$ be the associated graph. For any $i\in\mathcal{V}$
	and $\varepsilon\in\mathcal{E}$ with initial vertex $i$ let $p_{i\varepsilon}$ be the probability that a walker who is passing
	through vertex $i$ chooses to continue his walk through the edge $\varepsilon$, and assume that the sum of the probabilities over all edges originating
	at any vertex is equal to $1$. Let $I$ be an interval that is contained in some edge $\varepsilon\in\mathcal{E}$ with initial vertex $h\in\mathcal{V}$, and assume that $I$ is of length $\delta>0$ and distance $\alpha_{0}$ from the vertex $h$. Then the probability that a walker originating at vertex $i\in\mathcal{V}$ and advancing at unit speed is on the interval $I$ after walking along a metric path of length $x$, tends to 
	\[
	p_{h\varepsilon}\delta q_{h}+o(1),\quad x\rightarrow\infty,
	\]
	independent of $i$ and of $\alpha_0$, where $q_{h}$ is as in Theorem \ref{thm: graph path counting in the case of multiscale}.
\end{thm}

\begin{proof}
	The proof is very similar to that of Theorem \ref{thm: graph path counting in the case of multiscale} given above for counting paths terminating in given intervals. Here we adjust Theorem 2 in \cite{Yotam graphs} in the same way Theorem 1 in \cite{Yotam graphs} is adjusted for the proof of Theorem \ref{thm: graph path counting in the case of multiscale}.
\end{proof}

\begin{thm}
	Let $\sigma$ be an irreducible incommensurable scheme in $\R^d$. For every
	$j=1,\ldots,n$ and $0\leq a<b\le 1$, the volume covered by tiles of type $j$
	in $F_{t}(T_{i})$ with scale in $[a,b]$, or equivalently with volume in $[a^d,b^d]$, grows as 
	\[
	\nu_{j,[a,b]}e^{dt}+o\left(e^{dt}\right),\quad t\rightarrow\infty,
	\]
	where $\nu_{j,[a,b]}:=\sum_{h=1}^{n}d_{hj,[a,b]}q_{h}$ and $q_h$ is as in Theorem \ref{thm: graph path counting in the case of multiscale}, 
	\[
	d_{hj,[a,b]}:=\sum_{k=1}^{k_{hj}}\left(\alpha_{hj}^{\left(k\right)}\right)^{d}\log\tfrac{\eta_{hj}^{\left(k\right)}(a)}{\mu_{hj}^{\left(k\right)}(b)},
	\]
	with $\eta_{hj}^{\left(k\right)}(a)$ and $\mu_{hj}^{\left(k\right)}(b)$ as in \eqref{eq:eta_and_mu_hj}, and independent of $i$.
\end{thm}

\begin{proof}
	We assign probabilities to the graph $G_\sigma$ in the following way. Assume $\varepsilon$ is an edge with initial vertex $h$ that is associated to a tile $T=\alpha T_j\in\omega_\sigma(T_h)$. Then put $p_{h\varepsilon}=\vol T=\alpha^{d}$, which is the probability for
	a point in $T_i$ to belong to the tile $T$ after the substitution rule $\varrho_\sigma$ is applied to $T_i$. By \eqref{eq:volumes sum up to 1}, for every vertex the sum of the probabilities on its outgoing edges is $1$. 
	
	Let $\varepsilon$ be an edge as above, and recall that $l(\varepsilon)=\log\frac{1}{\alpha}$. We wish to calculate the probability that a metric path terminates at a point on $\varepsilon$ associated with a tile of scale in $[a,b]$. Once again there are three distinct cases:
	\begin{enumerate}
		\item If $\alpha<a<b$ then $\delta=\log\tfrac{b}{a}$, and so
		\begin{equation*}
			p_{h\varepsilon}\delta = \alpha^d \log\tfrac{b}{a}.
		\end{equation*}
		\item If $a\le \alpha<b$ then $\delta=\log\tfrac{b}{\alpha}$, and so 
		\begin{equation*}
		p_{h\varepsilon}\delta = \alpha^d \log\tfrac{b}{\alpha}.
		\end{equation*}
		\item If $a<b\le \alpha$ then $\delta=0$ and there is no contribution to the counting.
	\end{enumerate}
	
	Summing the contributions from all edges in $G_\sigma$ that terminate at the vertex $j$, we establish that the probability for a point in the patch $F_t(T_i)$ to be in a tile of type $j$ and scale in $[a,b]$ tends to 
	\begin{equation*}
 	\sum_{h=1}^{n}d_{hj,[a,b]}q_{h}+o\left(1\right),\quad t\rightarrow\infty,
	\end{equation*}	
	and since the volume of this patch is $e^{dt}$ we arrive at the required formula.
\end{proof}

The proof of Corollary \ref{cor: frequency of types} yields also the following analogous formula.
  
\begin{cor}\label{cor: volume covered by types}
	Let $\left(T^{(m)}\right)_{m\ge0}$  be a sequence of supertiles of growing order in $\TT\in\X_\sigma^F$. The volume of the region covered by tiles of type $j$ in $[T^{(m)}]^\TT$
	grows as
	\[
	\nu_{j,(\beta_j^{\min},1]}\vol\left(T^{(m)}\right)+o\left(\vol\left(T^{(m)}\right)\right),\quad m\rightarrow\infty.
	\]
	where 
	\[
	\nu_{j,(\beta_j^{\min},1]}=\sum_{h=1}^{n}\sum_{k=1}^{k_{hj}}\left(\alpha_{hj}^{\left(k\right)}\right)^{d}\log\tfrac{1}{\alpha_{hj}^{\left(k\right)}}q_h.
	\]
\end{cor}

\begin{remark}
	The results and formulas stated above can be used to recover the formulas originally obtained by Sadun for irrational generalized pinwheel tilings in \cite[Theorem 8]{Sadun - generalized Pinwhell}, as these tilings can also be described as stationary tilings generated by incommensurable substitution schemes on a single right triangle, as described in Remark \ref{rem: stationary with isometries}. We note that the computations for Corollaries \ref{cor: frequency of types} and \ref{cor: volume covered by types} appear also in \cite[\S 2, \S 8]{Yotam Kakutani}, which includes additional explicit formulas for various examples in the context of Kakutani sequences of partitions. 
\end{remark}

\section{Multiscale tilings are not uniformly spread}\label{sec: BD BL}
Following Laczkovich \cite{Laczk}, we say that a point set $Y\subset\R^d$ is \emph{uniformly spread} if there exists some $\alpha>0$ and a bijection $\phi:Y\to\alpha\Z^d$ satisfying
\[\sup_{y\in Y}\norm{y-\phi(y)}<\infty.\] 
Such a mapping $\phi$ is called a \emph{bounded displacement (BD)}.  Given a tiling $\TT$ of $\R^d$, by tiles of uniformly bounded diameter, we say that $\TT$ is \emph{uniformly spread} if there exists a point set $Y_\TT$, that is obtained by picking a point from each tile in $\TT$, which is uniformly spread. Note that since the tiles are of uniformly bounded diameter, if $\TT$ is uniformly spread then every point set $Y_\TT$ obtained in the above manner is uniformly spread. 

The following criterion, phrased here for tilings instead of point sets, was proved by Laczkovich, see \cite[Theorem 1.1]{Laczk}. Recall that for a set $A\subset\R^d$ and a constant $r>0$, we denote $A^{+r}=\{x\in\R^d:\ \dist(x,A)\le r\}$.  
\begin{thm}[\cite{Laczk}]\label{thm:Laczkovich_creterion}
	Let $\TT$ be a tiling of $\R^d$ by tiles of uniformly bounded diameter, then $\TT$ is uniformly spread if and only if there exist positive constants $c,\alpha>0$ so that 
	\begin{equation}\label{eq:Laczkovich_criterion}
	\absolute{\#[U]^\TT-\alpha\cdot \vol(U)}\le c\cdot \vol((\partial U)^{+1})
	\end{equation} 
	holds for every bounded, measurable set $U$. 		 
\end{thm}  
The question whether fixed scale substitution tilings are uniformly spread was studied in \cite{ACG,mixed BD,Solomon11} and \cite{Solomon14}, where it was shown that certain conditions on the eigenvalues and eigenvectors of the substitution matrix imply a positive answer to this question. Our goal in this section is to show that under a mild assumption on the geometry of the boundaries of the participating prototiles, incommensurable multiscale substitution tilings are never uniformly spread, emphasizing once again the difference between the tilings in the standard setup and the tilings being studied here.  

\begin{thm}\label{thm:not_BD_to_Z^d}
	Let $\sigma$ be an irreducible incommensurable substitution scheme in $\R^d$, and assume that the boundary of every prototile $T_i\in\tau_\sigma$ satisfies
	\begin{equation}\label{eq: BD assumption}
	\vol\left((\partial \supp {F_t(T_i)})^{+1}\right)=O(e^{(d-\eta)t})
	\end{equation} 
	for some fixed $\eta>0$. Then every $\TT\in\X_\sigma^F$ is not uniformly spread. 
\end{thm}

Note that for polygonal prototiles \eqref{eq: BD assumption} clearly holds with $\eta=1$. In a subsequent work \cite{non-BD} we show that as a result of Theorem \ref{thm:not_BD_to_Z^d} and the minimality of the dynamical system $\left(\X_\sigma^F,\mathbb{R}^{d}\right)$, the set of BD-equivalence classes that appear in $\X_\sigma^F$ has the cardinality of the continuum. 

We present two independent proofs of Theorem \ref{thm:not_BD_to_Z^d}. Both proofs rely on finding sufficiently large lower bounds on the discrepancy in \eqref{eq:Laczkovich_criterion}, allowing the application of Laczkovich's criterion from which Theorem \ref{thm:not_BD_to_Z^d} is implied. 

\subsection*{Proof 1 of Theorem \ref{thm:not_BD_to_Z^d}}
 The first proof follows from the lower bound given in Lemma \ref{lem:BD-lower_bound_for_error_term}. This is standard in the context of the dynamical and the fractal zeta functions, to which the books \cite{Parry Pollicott book} and \cite{Lapidus} are respectively dedicated, and so we only sketch a proof. 
 
 Let $E(t)$ denote the error term in Theorem \ref{thm: type and scale frequencies}.
\begin{lem}\label{lem:BD-lower_bound_for_error_term}
	 Let $\sigma$ be an irreducible, incommensurable substitution scheme in $\R^d$. 
	 Then for no constant $\beta<d$ do we have  
	\begin{equation*}
	E(t) = O\left(e^{\beta t}\right).
	\end{equation*}
	In particular, if $E_m$ denotes the error term in \eqref{eq: total number of tiles in supertiles}, for no $\varepsilon>0$ do we have
	\begin{equation*}
	E_m=O\left(\left(\vol\left(T^{(m)}\right)\right)^{1-\varepsilon}\right).
	\end{equation*} 
\end{lem}
\begin{proof}[Sketch of proof]
	The result follows from information on the location of  the poles of the Laplace transform of the counting function as appears in \eqref{eq: Laplace tranform}, which are the zeros of the exponential polynomial $\det(I-M_\sigma(s))$. 
	
	First, assume by contradiction that the error term is bounded by $e^{\beta t}$ for some $\beta<d$. Then by inverse Laplace transform theory, all poles $s\in\mathbb{C}$ of the Laplace transform other than $s=d$ have real part less or equal to $\beta<d$, see \cite[Proposition 6]{Pollicott}.  Therefore, they must be bounded away from the vertical line $\text{Re}(s)=d$.
	
	On the other hand, incommensurability of $\sigma$ implies that $d$ is a limit point of the real parts of zeros of $\det(I-M_\sigma(s))$. Similarly to \cite[Theorem 3.23]{Lapidus}, this is deduced by considering rational approximations to the exponential polynomial $\det(I-M_\sigma(s))$, and using Rouch\'e's theorem. See also \cite[Proposition 7]{Pollicott}.
\end{proof}

\begin{proof}[Proof 1 of Theorem \ref{thm:not_BD_to_Z^d}]
	Let $\TT\in\X_\sigma^F$ and let $\left(T^{(m)}\right)_{m\ge0}$  be a sequence of supertiles of growing order in $\TT$. Denote by $\#[T^{(m)}]^\TT$  the number of tiles in $[T^{(m)}]^\TT$. By Corollary \ref{cor: frequency of types}, $\#[T^{(m)}]^\TT$ grows as $C\vol(T^{(m)})+ E_m$ for some constant $C$. Hence for any $\alpha\neq C$, the inequality \eqref{eq:Laczkovich_criterion} clearly fails for $U=T^{(m)}$. For $\alpha=C$ we obtain that 
	\[
	\absolute{\#[T^{(m)}]^\TT-\alpha\cdot \vol\left(T^{(m)}\right)} = E_m. 
	\]
	$T^{(m)}$ are $m$-supertiles and hence by \eqref{eq: BD assumption} 
	\[
	\vol\left(\left(\partial T^{(m)}\right)^{+1}\right) \le C_\partial \left(\vol\left(T^{(m)}\right)\right)^{(d-\eta)/d}
	\]
	for some constant $C_\partial$. In view of Lemma \ref{lem:BD-lower_bound_for_error_term} we obtain that for any constant $c$  
	\[
	\absolute{\#[T^{(m)}]^\TT-\alpha\cdot \vol\left(T^{(m)}\right)}>c\cdot\vol\left(\left(\partial T^{(m)}\right)^{+1}\right)
	\]
	for $m$ large enough. We deduce that for any constant $c$ the inequality \eqref{eq:Laczkovich_criterion} in Laczkovich's criterion fails for $U=T^{(m)}$, for $m$ large enough, and the proof is complete.
\end{proof}

\subsection*{Proof 2 of Theorem \ref{thm:not_BD_to_Z^d}}
The following reasoning, including the proof of Lemma \ref{lem:discrepancy_lower_bound} below, was pointed out to us by the anonymous referee of the first version of the paper. We are truly grateful for the suggestions and for the opportunity to include a second approach in this revised version.

Given a multiscale substitution scheme $\sigma$, say that edges of the same length and with the same initial and terminal vertices are equivalent, and denote by $\EE$ the set of equivalence classes of edges of $G_\sigma$.

\begin{lem}\label{lem:poly(log)_lemma}
	There exists an absolute constant $C$, that depends on the parameters of $\sigma$, with the following property: For every $t>0$ and $T_i\in\tau_\sigma$ the patch $F_t(T_i)$ contains at most $C\cdot t^{\EEE}$ tiles that are pairwise translation non-equivalent. 	
\end{lem}
	
\begin{proof}
	First, observe that by Proposition \ref{prop: correspondence of paths and tiles}, tiles of type $j$ in $F_t(T_i)$ correspond to paths of length $t$ in $G_\sigma$ that initiate at vertex $i\in\mathcal{V}$ and terminate on an edge that terminates at vertex $j$. In addition, translation equivalent tiles are in particular of the same scale, and so the corresponding paths terminate at the same distance from $j$. Note that the number of edges in a path of length $t$ is bounded between $t/M$ and $t/m$, where $M$ and $m$ are the lengths of the longest and shortest edges of $G_\sigma$, respectively. Let $N_\gamma$ denote the number of edges in a path $\gamma$, not including the last partial edge (if there is one). In particular, if $\gamma$ and $\eta$ are two paths that initiate at vertex $i$, terminate on equivalent edges, and contain the same collection of edges counted with multiplicities, then $\gamma$ and $\eta$ correspond to tiles in $F_t(T_i)$ that differ by a translation. In other words, if $N_\gamma = N_\eta$ and the paths $\gamma$ and $\eta$ correspond to translation non-equivalent tiles, then these two paths contain different collections of edges, counted with multiplicities. 
	
	In view of the above, the number of different partitions of $N$ edges into $\EEE$ cells is an upper bound for the size of the largest set of translation non-equivalent tiles of type $j$ in $F_t(T_i)$ that correspond to paths with $N$ edges, that is, to paths $\gamma$ with $N_\gamma=N$. The number of such partitions is 
	\[\binom{N + \EEE -1}{\EEE - 1} = O(N^{\EEE -1}) = O(t^{\EEE -1}),\]
	where the implied constants depend only on the parameters of $\sigma$. Since there are finitely many types of tiles, there are $O(t^{\EEE -1})$ pairwise translation non-equivalent tiles in $F_t(T_i)$ that correspond to paths with $N$ edges. Recall that $N$ is an integer between $t/M$ and $t/m$, thus there are $O(t^{\EEE})$ pairwise translation non-equivalent tiles in $F_t(T_i)$, as required.
\end{proof}
	
	We denote by $\#\PP$ the number of tiles in a patch $\PP$. 
\begin{lem}\label{lem:discrepancy_lower_bound}
	For every $t_0>0$ and $T_i\in\tau_\sigma$ there exist $t\ge t_0$ and $\varepsilon_0>0$ such that for every $\varepsilon\in(0,\varepsilon_0]$ we have 
	\begin{equation}\label{eq:discrepancy_lower_bound}
		\# F_{t+\varepsilon}(T_i) - \# F_{t}(T_i) \ge 
		C\frac{\vol \big(\supp F_{t}(T_i)\big)}{\log^{\EEE}\vol \big(\supp F_{t}(T_i)\big)}  =C \frac{e^{dt}}{t^{\EEE}},	
	\end{equation}	
	where $C$ depends only on the parameters of $\sigma$.
\end{lem}
\begin{proof}
	Let $t_0>0$ and $T_i\in\tau_\sigma$ and consider the patch $\PP=F_{t_0}(T_i)$ . By the definition of the semi-flow $F_t$ we have $\vol(\supp(\PP)) = e^{dt_0}$, which implies that $\#\PP \ge e^{dt_0}$ since all tiles are of  volume at most 1. Combining Lemma \ref{lem:poly(log)_lemma} and the pigeonhole principle, there is a tile $T$ in $\PP$ and a constant $c>0$, so that $T$ has at least 	$c\frac{e^{dt_0}}{{t_0}^{\EEE}}$
	many translation equivalent copies in $\PP$. Suppose that $\alpha\le 1$ is the scale of $T$, and set $t=t_0+\log1/\alpha$. Then the patch $\PP_1:=F_t(T_i)$ contains a tile $F_{\log1/\alpha}(T)$ which is of unit volume, and has at least $c\frac{e^{dt_0}}{{t_0}^{\EEE}}$ many translation equivalent copies in $\PP_1$. Let $\beta<1$ be the maximal scale of a tile in $\PP_1$ whose scale is strictly smaller than 1, and set $\varepsilon_0:=\log1/\beta>0$. Then the number of tiles in the patch $F_{t+\varepsilon}(T_i)$ is the same for any choice of $\varepsilon\in(0,\varepsilon_0]$. Pick some $\varepsilon\in(0,\varepsilon_0]$ and let $\PP_2:= F_{t+\varepsilon}(T_i)$.
	Then all of the copies of $F_{\log1/\alpha}(T)$ in $\PP_1$ are of unit volume, where in $\PP_2$ all of these tiles were already subdivided via $\sigma$, and \eqref{eq:discrepancy_lower_bound} follows. 
\end{proof}

\begin{proof}[Proof 2 of Theorem \ref{thm:not_BD_to_Z^d}]
	 Observe that it follows from Theorem \ref{thm:Laczkovich_creterion} that for any minimal space $\X$ of Delone sets, or tilings, if there exists some $\TT\in\X$ that is not uniformly spread, then every $\TT\in\X$ is not uniformly spread (compare \cite[Theorem 3.2]{FG}, the argument holds in any minimal space). In fact, under the assumption of minimality, if some $\TT\in\X$ satisfies \eqref{eq:Laczkovich_criterion} for all measurable sets $U\subset \R^d$ with some positive constants $\alpha$ and $c>0$, then every $\TT\in\X$ satisfies \eqref{eq:Laczkovich_criterion} with the same constants $\alpha$ and $c$. 
	 So in view of Theorem \ref{thm: tiling space is minimal} we may assume by contradiction that there exist constants $\alpha,c>0$ for which
	 \begin{equation}\label{eq:assume_on_the_contrary_for_BD}
	 \forall \TT\in\X_\sigma^F\quad 
	 \forall U\subset\R^d \text{ measurable }:\quad
	 \absolute{\#[U]^\TT-\alpha\cdot \vol(U)}\le c\cdot \vol((\partial U)^{+1}).	 
	 \end{equation}
	  Consider sets $U$ of the form $\supp(F_t(T_i))$, with $t>0$ and $T_i\in\tau_\sigma$. By \eqref{eq: BD assumption} there exist $\delta>0$ and $t_0$ large enough so that for every $t\ge t_0$ and every $T_i\in\tau_\sigma$ 
	\begin{equation}\label{eq:discrepancy>boundary}
	 \frac{C}{3}\frac{e^{dt}}{t^{\EEE}} > c\cdot e^{(d-\eta+\delta)t} \quad\text{ and }\quad 
	 e^{(d-\eta+\delta)t} > \vol\left((\partial \supp(F_t(T_i)))^{+1}\right),
	 \end{equation}
	 where $C$ is the constant in \eqref{eq:discrepancy_lower_bound}. By Lemma \ref{lem:discrepancy_lower_bound}, there exist some $t \ge t_0$ and $\varepsilon>0$ that satisfy \eqref{eq:discrepancy_lower_bound}. Since    
	 \[\absolute{\vol\big(\supp(F_{t}(T_i))\big) -  \vol\big(\supp(F_{t+\varepsilon}(T_i))\big)}\le 
	 \vol\left((\partial \supp(F_t(T_i)))^{+1}\right) \]
	 we deduce from \eqref{eq:discrepancy_lower_bound} that either 
	\[
	 \absolute{\#F_{t+\varepsilon}(T_i) - \alpha \cdot \vol\big(\supp(F_{t+\varepsilon}(T_i))\big)} \ge \frac{C}{3}\frac{e^{dt}}{t^{\EEE}} > c\cdot e^{(d-\eta+\delta)t}, \]
	 or 
	 \[\absolute{\#F_{t}(T_i) - \alpha \cdot \vol\big(\supp(F_{t}(T_i))\big)} \ge \frac{C}{3}\frac{e^{dt}}{t^{\EEE}} > c\cdot e^{(d-\eta+\delta)t}.\]
	 Combined with \eqref{eq:discrepancy>boundary} this contradicts \eqref{eq:assume_on_the_contrary_for_BD}, which completes the proof.  
\end{proof}

\begin{remark}
	Let $d\ge 2$. A point set $Y\subset\R^d$ is called \emph{rectifiable}, or \emph{bi-Lipschitz equivalent to a lattice}, if there exists a bi-Lipschitz bijection 
	$\phi:Y\to\Z^d$. Namely, a bijection $\phi$ for which there is some constant $L\ge 1$ that satisfies 
	\[\frac{1}{L}\le \frac{ \norm{ \phi(y_1)-\phi(y_2) } }{\norm{y_1-y_2}} \le L\]
	for every two distinct points $y_1,y_2\in Y$.

	Rectifiability is often established using a sufficient condition of Burago and Kleiner, which requires an appropriate upper bound on the discrepancy $\absolute{\#[U]^\TT-\alpha\cdot \vol(U)}$ for large cubes $U$, see \cite{BK2}. Improving the lower bound of $E(t)$ from Lemma \ref{lem:BD-lower_bound_for_error_term} to a polynomial error term of the form $e^{dt}/{t^\delta}$ for some $\delta\le1$, would imply the failure of Burago-Kleiner condition for all point sets that arise from multiscale substitution tilings. Loh\"ofer and Mayer claim this with $\delta=1$ for what in our context would be constructions associated with the golden ratio, which heuristically is the case with the smallest error term, see \cite[p. 5]{Lohofer-Mayer}. Unfortunately, this appears without proof. 
	
	Consider the ``Sturmian'' tiling discussed in Proposition \ref{prop: Sturmian tilings}, for which $c_\SS(k) = k+1$ and the generating substitution scheme consists of a single square prototile. Namely, there are only $k+1$ tiles in the patch $F_{ks}(T)$ that are pairwise translation non-equivalent. Repeating the arguments of Lemma \ref{lem:discrepancy_lower_bound} and Proof 2 of Theorem \ref{thm:not_BD_to_Z^d} yields a sequence  $(U_k)_{k\ge0}$ of square patches that grow exponentially, with 
	\begin{equation*}
	\absolute{\#[U_k]^\TT-\alpha\cdot \vol(U_k)}= \Omega\left(\frac{\vol(U_k)}{\log(\vol(U_k))}\right).	
	\end{equation*} 
	This shows in particular that the Burago-Kleiner sufficient condition fails for this tiling! 
	
	We remark that the existence of non-rectifiable point sets in $\R^d$ is a non-trivial result, see \cite{BK1} and \cite{McMullen}. Finding concrete non-rectifiable examples, which currently include only those described in \cite{CortezNavas} and in \cite{Garber}, is a very interesting problem.

\end{remark}

\section{Uniform patch frequencies}\label{sec: patches}

This section is dedicated to the study of patch frequencies in incommensurable tilings. We define a multiscale tiling variant of patch frequency, where instead of counting appearances of patches up to translation equivalence as done in standard constructions, we group together patches that are dilations of each other. This is crucial in order to establish the existence of positive uniform patch frequencies for dilations of legal patches, which is key in our proof of unique ergodicity of incommensurable tiling dynamical system in \S\ref{sec: unique ergodicity}. 

Our proof follows the framework suggested in appendix A.1 of \cite{Lee-Moody-Solomyak2} for the existence of uniform patch frequencies of fixed scale substitution tilings. Indeed, some of the steps are identical to those of Lee, Moody and Solomyak, while others require new ideas and results that strongly depend on the incommensurability of the underlying substitution scheme. In particular, the assumption of  primitivity and the theory of Perron-Frobenius are replaced by our results on patches in incommensurable tilings and their scales developed in Sections \ref{sec: scales and complexity} and \ref{sec: Frequencies}.

\begin{definition}
	A sequence $\left(A_q\right)_{q\ge1}$ of bounded measurable subsets of $\R^d$  is \textit{van Hove} if
	\begin{equation*}
	\lim_{q\to\infty} \frac{\vol \left ( (\partial A_q)^{+r}\right )}{\vol(A_q)}=0
	\end{equation*}
	for all $r > 0$, where as before $A^{+r}=\{x\in \R^d:\ \dist(x,A)\le r\}$. In addition, we 
	denote
	\begin{equation*}
	A^{-r} := \{x \in A:\ \dist(x,\partial A) \ge r\}
	\end{equation*}
	for any $r > 0$ and a bounded set $A\subset \R^d$. 
\end{definition}

Let $\sigma$ be an irreducible incommensurable substitution scheme in $\R^d$. By our definition of tiles and by Proposition \ref{prop:limiting tiles are similar to prototiles}, all tiles in a tiling $\TT\in\X_\sigma^F$ have boundary of measure zero, and so every sequence of supertiles with growing order is van Hove. Given a patch $\PP$ in a tiling $\TT\in\X_\sigma^F$, a bounded interval $I\subset\R$ and a bounded set $A\subset \R^d$, denote
\begin{equation*}\label{eq:L_P,I_and_N_P,I}
\begin{split}
L_{\PP,I}(A,\TT)&:=\#\{g\in\R^d:\ \exists \alpha\in I \text{ s.t. } g+\alpha \PP\subset \TT, (g+\supp (\alpha \PP))\subset A\},\\
N_{\PP,I}(A,\TT)&:=\#\{g\in\R^d:\ \exists \alpha\in I \text{ s.t. } g+\alpha \PP\subset \TT, (g+\supp (\alpha \PP))\cap A\neq \emptyset\},
\end{split}
\end{equation*}
where once again $\#B$ denotes the number of elements in a finite set $B$.

\begin{thm}\label{thm: uniform patch frequencies}
	Let $\sigma$ be an irreducible incommensurable substitution scheme in $\R^d$, and let $\SS\in\X_\sigma^F$ be a stationary tiling. Let $\PP$ be a patch in $\SS$ and let $I$ be a bounded interval that contains $1$ and a left neighborhood of $1$. Then for every van Hove sequence $\left(A_q\right)_{q\ge1}$ in $\R^d$
	\begin{equation}\label{eq: frequencies}
	\emph{freq}(\PP,I,\SS):=\lim_{q\to\infty}\frac{L_{\PP,I}(A_q+h,\SS)}{\vol(A_q)}
	\end{equation}
	exists uniformly in $h\in\R^d$, and is positive.
\end{thm}

Let $s\in\R^+$ be so that $\SS=F_s(\SS)$. Since $\SS$ is fixed, we simplify notation and set $L_{\PP,I}(A):=L_{\PP,I}(A,\SS)$ and $N_{\PP,I}(A):=N_{\PP,I}(A,\SS)$.

\begin{lem}\label{lem:A.4 in LMS}
	Let $A\subset\R^d$ be a bounded set and let $\left(A_q\right)_{q\ge1}$ be a van Hove sequence in $\R^d$. Let $\PP$ be a patch in $\SS$ and let $I\subset\R$ be a bounded interval that contains $1$ and a left neighborhood of $1$. Then 
	\begin{enumerate}
		\item there exists $c_1>0$, which depends only on $\sigma$, so that 
		\begin{equation*}
		L_{\PP,I}(A)\le c_1 \vol(A).
		\end{equation*}
		\item there exist $c_2,q_0>0$, which depend on $\PP,I$ and $\left(A_q\right)_{q\ge1}$, so that for all $q\ge q_0$
		\begin{equation*}
				L_{\PP,I}(A_q+h)\ge c_2 \vol(A_q).
		\end{equation*}
		\item $\lim_{q\rightarrow \infty}N_{\PP,I}(\partial(A_q+h))/L_{\PP,I}(A_q+h)=0$ uniformly in $h\in\R^d$. 
	\end{enumerate}
\end{lem}

\begin{proof}
	We follow the proof of Lemma A.4 in \cite{Lee-Moody-Solomyak2}. Parts (1) and (3) are identical but are short and so we include them here. Lemma \ref{thm:The_scales_for_patches_are_dense} replaces primitivity in part (2).
	\begin{enumerate}
		\item Select a tile from the patch $\PP$. Distinct patches of the form $\alpha \PP$ in $\SS$ will have distinct selected tiles, therefore 
		\begin{equation*}
		L_{\PP,I}(A)\le V_{\min}^{-1}\vol(A),
		\end{equation*}
		where $V_{\min}$ is the infimum over tile volumes in $\SS$.

		\item By Lemma \ref{thm:The_scales_for_patches_are_dense}, and since $\SS$ is stationary, there exists an $\ell\in \N$ so that for every tile $T$ in $\SS$, the patch $F_{\ell s}(T)$ contains a copy of the patch $\alpha \PP$ as a sub-patch, for some $\alpha\in I$.
		It follows that for any set $A$, the number $L_{\PP,I}(A)$ is at least the number of $\ell$-supertiles with support contained in $A$. Since tiles in $\SS$ are of volume at most $1$, the volume of an $\ell$-supertile is at most $e^{\ell sd}$. Let $r$ be the supremum of the diameters of $\ell$-supertiles. We have obtained
		\begin{equation*}\label{eq: bound L from below for van hove}
		\begin{split}
		L_{\PP,I}(A_q+h)&\ge e^{-\ell sd}\vol(A_q^{-r}+h)\\
		&=e^{-\ell sd}\vol(A_q^{-r})\\
		&\ge e^{-\ell sd}\left(\vol(A_q)-\vol\left((\partial A_q)^{+r}\right)\right).
		\end{split}
		\end{equation*}
		Since $\left(A_q\right)_{q\ge1}$ is van Hove, the proof is complete.
		 
		\item Let $r$ denote the supremum of the diameters of the supports of $\alpha \PP$, for $\alpha\in I$. In view of parts (1) and (2) of this lemma
		\begin{equation*}
		\frac{N_{\PP,I}(\partial A_q +h)}{L_{\PP,I}(A_q +h)}\le \frac{L_{\PP,I}\left((\partial A_q)^{+r}+h\right)}{L_{\PP,I}(A_q +h)}\le \frac{c_1\vol((\partial A_q)^{+r} +h)}{c_2\vol(A_q +h)}=\frac{c_1\vol((\partial A_q)^{+r})}{c_2\vol(A_q)}\rightarrow 0,
		\end{equation*}
		uniformly in $h\in\R^d$, as required.\qedhere
	\end{enumerate}
\end{proof}

Consider  \eqref{eq: bound L from below for van hove} for $q$-supertiles, that is, $A_q=e^{qs}\supp T$, where $T$ is a tile in $\SS$. Since there are only finitely many types of tiles and the set of scales in which they appear is bounded, we obtain the following. 

\begin{cor}\label{cor: Lemma A.4 for supertiles is uniform}
Under the assumptions of Lemma \ref{lem:A.4 in LMS}, if $A_q=e^{qs}\supp T$, where $T$ is a tile in $\SS$, then there exists a constant for which \emph{(2)} holds for all tiles $T$ in $\SS$, and the convergence in \emph{(3)} is uniform in the choice of $T$.
\end{cor}
	
The main step of the proof of Theorem \ref{thm: uniform patch frequencies} is given by the following Lemma.

\begin{lem}\label{lem: uniform patch frequency for supertiles}
	Let $\SS, \PP$ and $I$ be as above. Then 
	\begin{equation}\label{eq:C_P,I}
	c_{\PP,I}:=\lim_{q\to\infty}\frac{L_{\PP,I}(e^{qs}\supp T)}{\vol(e^{qs}\supp T)}>0
	\end{equation}
	exists uniformly in tiles $T$ in $\SS$.
\end{lem}

\begin{proof}
	Fix $\varepsilon>0$. By Lemma \ref{lem:A.4 in LMS} and Corollary \ref{cor: Lemma A.4 for supertiles is uniform}, there exists $m_0\in\N$ so that for every $m\ge m_0$ and every tile $T$ in $\SS$
	\begin{equation}\label{eq: N<eL for supertiles}
	N_{\PP,I}(\partial e^{ms}\supp T)\le \varepsilon\, L_{\PP,I}(e^{ms}\supp T).
	\end{equation}

	Choose a tile $T$ in $\SS$, and assume it is of type $j$ and scale $\alpha$, that is, $T$ is a translated copy of $\alpha T_{j}$, and set $A=\supp T$. For $q>m>m_0$, consider the decomposition of the $q$-supertile $e^{qs}T$ into $m$-supertiles. The support of an $m$-supertile of type $i$ is of the form $e^{ms}\beta A_i$, where $A_i$ is a translation of the support of the prototile $T_i\in\tau_\sigma$, and $\beta\in(\beta_i^{\min}, 1]$ is the scale in which the corresponding tile appears in $F_{(q-m)s}(\alpha T_{j})$. 
	
	Observe that the set $\left\{F_{ms}(\beta T_i):\,\beta\in(\beta^{\min}_{i},1]\right\}$ consists of a finite number of patches up to rescaling, because $F_{ms}(\beta T_i)=F_{ms-\log(1/\beta)}(T_i)$ and \[\SSS_i\cap\left\{ms-\log(1/\beta):\, \beta\in(\beta^{\min}_{i},1]\right\}
	\]
	is a finite set. Therefore, there is a finite number of tile substitutions under the semi-flow $F_t(T_i)$ for $ms-\log(1/\beta^{\min}_{i})<t\le ms$. It follows that $L_{\PP,I}(e^{ms}\beta A_i)$ is a piecewise constant function of the variable $\beta\in(\beta^{\min}_{i},1]$. Namely, there is a finite partition of the interval $(\beta^{\min}_{i},1]$ into $M_i$ sub-intervals $\{I_{i,1},\ldots,I_{i,M_i}\}$, depending only on $\PP, I$ and $m$, so that for every choice of representatives $\left\{\beta_{i,\ell}\in I_{i,\ell}:\,\ell=1,\ldots,M_i\right\}$ the equality
	\begin{equation*}
	L_{\PP,I}(e^{ms}\beta A_i)=L_{\PP,I}(e^{ms}\beta_{i,\ell} A_i).
	\end{equation*}
	holds for every $\beta\in I_{i,\ell}$ .
	 
	 Denote by $n_{i,\ell,q}$ the number of $m$-supertiles with support $e^{ms}\beta A_i$ with $\beta\in I_{i,\ell}$ in the decomposition of $e^{qs}T$ into $m$-supertiles. Together with  \eqref{eq: N<eL for supertiles}, we have
	 \begin{equation}\label{eq:double bound of L(A) for supertiles}
	 \begin{split}
	 \sum_{i=1}^n\sum_{\ell=1}^{M_i}L_{\PP,I}(e^{ms}\beta_{i,\ell} A_i)n_{i,\ell,q}&\le L_{\PP,I}(e^{qs}A)\\
	 &\le
	 \sum_{i=1}^n\sum_{\ell=1}^{M_i}L_{\PP,I}(e^{ms}\beta_{i,\ell} A_i)n_{i,\ell,q}+N_{\PP,I}(\partial e^{ms}\beta_{i,\ell} A_i)n_{i,\ell,q}\\ &\le(1+\varepsilon)\sum_{i=1}^n\sum_{\ell=1}^{M_i}L_{\PP,I}(e^{ms}\beta_{i,\ell} A_i)n_{i,\ell,q}.
	 \end{split}
	 \end{equation}

	Note that $n_{i,\ell,q}$ is exactly the number of tiles of type $i$ and scale in the interval $I_{i,\ell}$ in the patch $F_{(q-m)s-\log(1/\alpha)}(T_{j})$. By Theorem \ref{thm: type and scale frequencies}, there exists a constant $\varphi_{i,I_{i,\ell}}>0$ independent of $j$ so that
	\begin{equation*}
	\lim_{q\to\infty}\frac{n_{i,\ell,q}}{e^{((q-m)s-\log(1/\alpha))d}}=\varphi_{i,I_{i,\ell}}.
	\end{equation*}
	Since $\vol(e^{qs}A)=e^{(qs-\log(1/\alpha))d}$, this yields
\begin{equation}\label{eq:gamma_i,l_over_vol(e^qsA)}	
	\lim_{q\to\infty}\frac{n_{i,\ell,q}}{\vol(e^{qs}A)}=\varphi_{i,I_{i,\ell}}e^{-msd}.
\end{equation}
	
	Dividing \eqref{eq:double bound of L(A) for supertiles} by $\vol(e^{qs}A)$ and letting $q\rightarrow\infty$ as in \eqref{eq:gamma_i,l_over_vol(e^qsA)}, we obtain
	\begin{equation}\label{eq:limsup-liminf}
	\limsup_{q\rightarrow\infty}\frac{L_{\PP,I}(e^{qs}A)}{\vol(e^{qs}A)}-\liminf_{q\rightarrow\infty}\frac{L_{\PP,I}(e^{qs}A)}{\vol(e^{qs}A)}\le\varepsilon \sum_{i=1}^n\sum_{\ell=1}^{M_i}L_{\PP,I}(e^{ms}\beta_{i,\ell} A_i)\varphi_{i,I_{i,\ell}}e^{-msd}.	
	\end{equation} 
	By $(1)$ of Lemma \ref{lem:A.4 in LMS}, and since tiles in $\SS$ are of volume at most $1$, we have 
	\begin{equation*}
	L_{\PP,I}(e^{ms}\beta_{i,\ell} A_i)\le c_1\vol(e^{ms}\beta_{i,\ell} A_i)\le c_1e^{msd}.
	\end{equation*}
	The right-hand side of \eqref{eq:limsup-liminf} is thus bounded by
	\begin{equation*}
	\varepsilon \sum_{i=1}^n\sum_{\ell=1}^{M_i}c_1\varphi_{i,I_{i,\ell}}\le \varepsilon \,\widetilde{c},
	\end{equation*}
	with $\widetilde{c}$ independent of $q$ and of $m$. Since $\varepsilon>0$ is arbitrarily, the limit in \eqref{eq:C_P,I} exists. This limit, denoted by $c_{\PP,I}$, satisfies 
	\begin{equation*}
	\sum_{i=1}^n\sum_{\ell=1}^{M_i}L_{\PP,I}(e^{ms}\beta_{i,\ell} A_i)\varphi_{i,I_{i,\ell}}e^{-msd}\le c_{\PP,I}\le (1+\varepsilon)\sum_{i=1}^n\sum_{\ell=1}^{M_i}L_{\PP,I}(e^{ms}\beta_{i,\ell} A_i)\varphi_{i,I_{i,\ell}}e^{-msd},
	\end{equation*}
	independent of the choice of tile $T\in\SS$. By (2) of Lemma \ref{lem:A.4 in LMS}, for large enough values of $m$ the left-hand side is positive, and so $c_{\PP,I}$ is positive.
\end{proof}

The proof of Theorem \ref{thm: uniform patch frequencies} now follows by the same arguments as those given in detail in the proof of Lemma A.6 in \cite{Lee-Moody-Solomyak2}.  

\begin{proof}[Sketch of proof of Theorem \ref{thm: uniform patch frequencies}]
Consider the decomposition of  $\R^d$ into the $m$-supertiles $e^{ms}T$ for $T$ in $\SS$. $L_{\PP,I}(A_q+h)$ can be approximated by sums of $L_{\PP,I}(e^{ms}\supp T)$, using Lemma \ref{lem: uniform patch frequency for supertiles}. The van Hove property is then used to show that boundary effects are negligible. 
\end{proof}

In fact, the arguments used above can be used to establish the existence of uniform patch frequencies for general tilings $\TT\in\X_\sigma^F$. 

\begin{cor}\label{cor: patch frequency for all tilings in the space}
	Let $\TT\in\X_\sigma^F$. For any legal patch $\PP$ in $\TT$ and a bounded interval $I\subset\R$ that contains a left neighborhood of $1$
	\begin{equation*}
	\emph{freq}(\PP,I,\TT)=\emph{freq}(\PP,I,\SS)>0,
	\end{equation*}
	and in particular, Corollary \ref{cor: every stationary tiline contains rescaled copies of legal patches} holds for $\TT$. Non-legal patches have zero frequency. 
\end{cor}

\begin{proof}
	We only sketch the proof, since it follows from standard arguments about supertiles as in the proof of \cite[Theorem 4.11 (ii)]{Lee-Solomyak}. If $\PP$ is a legal patch and $I$ is an interval as above, then for any tiling $\TT\in\X_\sigma^F$ the associated frequency in supertiles approaches $c_{\PP,I}$, and when considering general sets the difference amounts to boundary effects. In the non-legal case, since such patches must intersect the boundaries of supertiles of arbitrarily large order, the van Hove property  of sequences of supertiles deem the frequencies negligible. 
\end{proof}

We end this section by showing that in the classical sense of uniform patch frequency, all patches in tiling $\X_\sigma^F$ have zero frequency. 

\begin{thm}\label{thm: zero uniform patch frequency for patches}
	Let $\sigma$ be an irreducible incommensurable substitution scheme in $\R^d$, let $\SS\in\X_\sigma^F$ be a stationary tiling, and let $\PP$ be a patch in $\TT\in\X_\sigma^F$. For any van Hove sequence $\left(A_q\right)_{q\ge1}$ in $\R^d$
	\begin{equation*}
		\emph{freq}(\PP,\TT):=\lim_{q\to\infty}\frac{L_{\PP}(A_q+h)}{\vol(A_q)}=0
	\end{equation*} 
	uniformly in $h\in\R^d$, where 
	\begin{equation*}
		L_{\PP}(A):=L_{\PP,\{1\}}=\#\{g\in\R^d:\ g+ \PP\subset \TT, (g+\supp (\PP))\subset A\}.
	\end{equation*}
\end{thm}

\begin{proof}
	The proof follows from the fact that for every $\varepsilon>0$ there exists $m_0\in\N$ so that for every $m\ge m_0$ 
	\begin{equation*}\label{eq:L_P(e^msA)/vol(e^msA)_is_small}
	\frac{L_{\PP}(e^{ms}A)}{\vol(e^{ms}A)}<\varepsilon
	\end{equation*}
	holds, for every set $A$ that is the support of a tile $T$. This is true because the patch $\PP$ contains at least one tile, say a translated copy of $\beta T_i$ for some prototile $T_i\in\tau_\sigma$ and $\beta\in(\beta^{\min}_{i},1]$. Clearly, $L_{\PP}(e^{ms}A)$ is smaller than the number of copies of $\beta T_i$ contained in $F_{ms}(A)$, which is $o(\vol(e^{ms}A))$ by Corollary \ref{cor: number of tiles of fixed scale in o(e^{dt})}. As in Theorem \ref{thm: uniform patch frequencies}, the proof now follows from standard approximation arguments for supertiles and the van Hove property.
	\end{proof}

\section{Unique ergodicity}\label{sec: unique ergodicity}

This section is dedicated to the proof of unique ergodicity of incommensurable tiling dynamical systems.

\begin{thm}\label{thm: X is uniquely ergodic}
	Let $\sigma$ be an irreducible incommensurable substitution scheme in $\R^d$. Then the dynamical system $(\X_\sigma^F,\R^d)$ is uniquely ergodic.
\end{thm}

Our proof draws inspiration from that given in \cite{Lee-Solomyak} for the case of tiling spaces associated with fixed scale substitution tilings of infinite local complexity. Although the approach and framework follow that of Lee and Solomyak, since Theorem \ref{thm: zero uniform patch frequency for patches} implies that the standard patch frequencies are all zero, their arguments cannot be applied directly. The main innovation is given in Lemma \ref{lem: convergence to weighted frequencies}, in which we take advantage of the fact that patch frequencies are non-zero only when patches are counted with their rescalings in some non-trivial interval of scales. This allows us to use the theory of Riemann-Stieltjes integration in order to evaluate the countable sum to which the ergodic averages converge.

\subsection{Cylinder sets and partitions of the space of tilings}

Let $\sigma$ be an irreducible incommensurable multiscale substitution scheme in $\R^d$, and let 
\begin{equation*}
\X_\SS:=\overline{\OO(\SS)}=\overline{\{\SS-x\,:\,x\in\R^d\}},
\end{equation*}
where $\SS\in\X_\sigma^F$ is a stationary tiling. As shown in \S\ref{sec:Dynamics}, the dynamical system $(\X_\sigma^F,\R^d)$ is minimal, and so $\X_\SS=\X_\sigma^F$. Nevertheless, we use the notation $\X_\SS$ to emphasize that the space is the orbit closure of a specified stationary tiling $\SS$. We begin by defining a sequence of partitions of $\X_\SS$ into finitely many cylinder sets. 

The idea is to create ``pixelized'' images from patches and to define cylinders as the collection of tilings with the same ``pixelized'' image on a large centered cube
\begin{equation*}
C_m:=[-2^m,2^m)^d.
\end{equation*}
Subdivide $C_m$ into small cubes $c_\omega$ of side length $\frac{1}{2^m}$ taken to be products of half open intervals (this are the ``pixels''), and write 
\begin{equation*}
C_m=\bigsqcup_{\omega = 1}^{2^{2m+1}d} c_\omega.
\end{equation*}
Let $\TT\in\X_\SS$, and consider the patch $[C_m]^\TT$, which consists of all tiles in $\TT$ that intersect $C_m$. We use this patch to color the small cubes with colors $\{0,1,\ldots,n\}$ according to the following rule. Recall that by Proposition \ref{prop:limiting tiles are similar to prototiles}, every tile in every tiling in $\X_\sigma^F$ is similar to one of the prototiles in $\tau_\sigma$ and can be assigned a type. If $c_\omega$ is contained in the interior of the support of a tile of type $i\in\{1,\ldots,n\}$, we color the cube $c_\omega$ in color $i$. Otherwise, $c_\omega$ intersects the boundary of the support of a tile in $\TT$. In such a case we color $c_\omega$ in the color $0$. This coloring of the small cubes in $C_m$ according to $[C_m]^\TT$ is the \textit{$m$-pixelization of $\TT$}, and is denoted by $\Pi_m(\TT)$, see Figure \ref{fig:pixelization} for an illustration.

\begin{figure}[ht!]
	\includegraphics[scale=0.9]{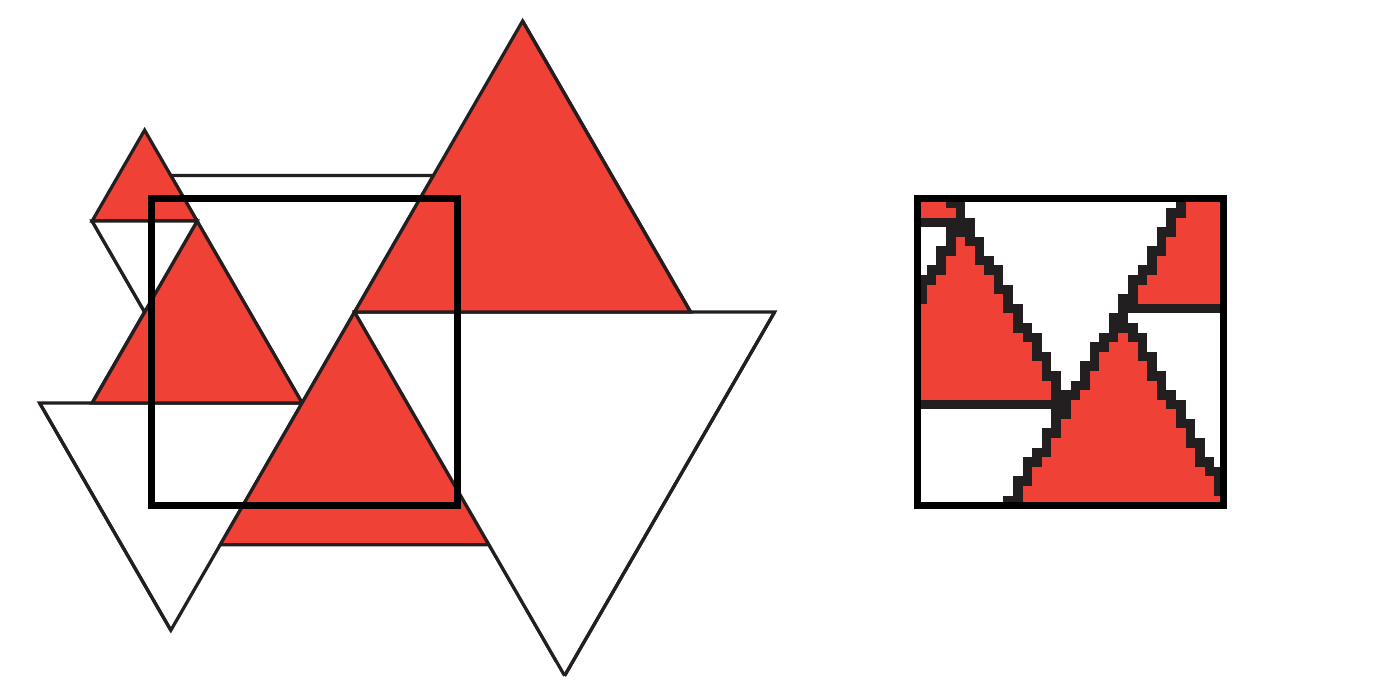}\caption{A patch $[C_m]^\TT$ and its $m$-pixelization.}	\label{fig:pixelization}	
\end{figure}

Next, denote by $\{U_{m,\ell}\}_{\ell=1}^{N_m}$ the finite set of all different colorings of the small cubes in $C_m$ using the colors $0,1,\ldots,n$. For $\ell\in\{1,\ldots,N_m\}$, let
\begin{equation*}
X(U_{m,\ell}):=\{\TT\in\X_\SS\,:\,\Pi_m(\TT)=U_{m,\ell}\}
\end{equation*}
be the set of tilings in $\X_\SS$ whose $m$-pixelization is $U_{m,\ell}$. This is a cylinder set in $\X_\SS$, and clearly for every $m\in\N$
\begin{equation}\label{eq:cylinder_decomposition}
\X_\SS=\bigsqcup_{\ell=1}^{N_m}X(U_{m,\ell}),
\end{equation} 
where $\bigsqcup$ denotes a disjoint union.

Fix $m\in\N$ and $\ell\in\{1,\ldots,N_m\}$, and consider the set $\mathcal{G}_{m,\ell}$ of all patches of the form $[C_m]^{\SS-x}$ with $\SS-x\in X(U_{m,\ell})$ and $x\in\R^d$. For every patch $\PP\in \mathcal{G}_{m,\ell}$, there is a Borel set $V_\PP\subset \R^d$ of \textit{maximal wiggle}, for which $\Pi_m(\PP-x)=U_{m,\ell}$ for all $x\in V_\PP$ and $V_\PP$ is maximal with this property with respect to inclusion. Since the small cubes $c_\omega$ are products of half open intervals, for every patch $\PP\in \mathcal{G}_{m,\ell}$ the set $V_\PP$ contains an open set. Note that although $V_\PP$ always contains the origin, it is not necessarily an interior point of $V_\PP$.

The tiling $\SS$ has countably many patches up to translation equivalence, and so there are countably many patches of the form $[C_m]^{\SS-x}$ in $\mathcal{G}_{m,\ell}$, modulo translation equivalence. We choose a set $(\mathcal{G}_{m,\ell})_\equiv$ of representatives in the following way. First, note that given a patch $\PP\in\mathcal{G}_{m,\ell}$, there exists an interval of scales in which dilations of $\PP$ have $U_{m,\ell}$ as their $m$-pixelization, where the $m$-pixelization of a patch with support that covers $C_m$ is defined similarly to the $m$-pixelization of a tiling. If this interval of scales is degenerate and contains only $1$, we can take advantage of the fact that the set $V_{\PP}$ of maximal wiggle contains an open set, and so for some small translation of $\PP$ the corresponding set of scales is non-degenerate. In addition, recall that by Lemma \ref{thm:The_scales_for_patches_are_dense}, patches appear in $\SS$ in a dense set of scales. Combining the above we conclude that it is always possible to choose a set $(\mathcal{G}_{m,\ell})_\equiv$ of representatives modulo translation equivalence of the form

\begin{equation}\label{eq:I_j} 
	(\mathcal{G}_{m,\ell})_\equiv=\bigcup_{j\ge 1}\{s_{ij}\PP_j\}_{i\ge 1},
\end{equation} 
with the property that for every $j\in\N$, the scaling constants $\{s_{ij}\}_{i\ge1}$ are dense in some interval $I_j$ that contains $1$ and a left neighborhood of $1$.

Given a patch $\PP$ and a set $V\subset\R^d$, define
\begin{equation*}
X(\PP,V):=\left\{\TT\in\X_\SS\,:\,\exists x\in V\text{ s.t. }\PP-x
\subset\TT\right\}.
\end{equation*}
It follows that 
\begin{equation*}
X(U_{m,\ell})=\bigsqcup_{i,j\ge1}X(s_{i,j}\PP_j,V_{s_{i,j}\PP_j})\bigsqcup\widetilde{X}_{m,\ell},
\end{equation*} 
where $\PP_j$, $\{s_{ij}\}_{i\ge1}$ and $I_j$ as in \eqref{eq:I_j}, $V_{s_{ij}\PP_j}$ are the sets of maximal wiggle for the patches $s_{ij}\PP_j$, and
\begin{equation*}
\widetilde{X}_{m,\ell}:=\left\{\TT\in\X_\SS\,:\,\Pi_m(\TT)=U_{m,\ell}, \not\exists x\in\R^d \text{ s.t. } [C_m]^\TT +x \subset\SS \right\}
\end{equation*}
are the tilings so that the patch $[C_m]^\TT$ is only admitted in the limit, compare  \cite{Lee-Solomyak}.

\subsection{Proof of unique ergodicity}

Denote by $\chi_{m,\ell}$ the characteristic function of $X(U_{m,\ell})$. Using the decomposition \eqref{eq:cylinder_decomposition}, in order to establish unique ergodicity for $(\X_\SS,\R^d)$ it is enough to show that for every van Hove sequence $\left(A_q\right)_{q\ge1}$ and every pair $(m,\ell)$
\begin{equation}\label{eq:unique_ergodicity_criterion}
\lim_{q\rightarrow \infty}\frac{1}{\vol(A_q)}\int_{A_q}\chi_{m,\ell}(\SS-x-h)dx= u_{m,\ell}
\end{equation}
uniformly in $h\in\R^d$, where $u_{m,\ell}$ is a constant depending only on $m$ and on $\ell$. We note that relying on the decomposition of the space \eqref{eq:cylinder_decomposition} into finitely many pairwise disjoint cylinders, which can be made arbitrarily small, the sufficiency of \eqref{eq:unique_ergodicity_criterion} is standard, see e.g. \cite[Theorem 2.6]{Lee-Moody-Solomyak1}. This is also the approach in  \cite[Theorem 3.2]{Lee-Solomyak}, where it is shown that
\begin{equation*}
\lim_{q\rightarrow \infty}\left(\frac{1}{\vol(A_q)}\int_{A_q}\chi_{m,\ell}(\SS-x-h)dx-\sum_{j=1}^\infty \sum_{i=1}^\infty\frac{\vol(V_{s_{i,j}\PP_j})L_{s_{i,j}\PP_j}(A_q+h)}{\vol(A_q)}\right)=0.
\end{equation*}
Here $L_\PP(A):=L_{\PP,\{1\}}(A)$, where as in \S\ref{sec: patches} we simplify notation and put $L_{\PP,I}(A)=L_{\PP,I}(A,\SS)$. Combined with \eqref{eq:unique_ergodicity_criterion}, unique ergodicity of $(\X_\SS,\R^d)$ will follow from the next result.

\begin{prop}\label{prop: equation condition for unique ergodicity}
	There exist constants $\widetilde{v}_j$ so that
	\begin{equation*}\label{eq: equation condition for unique ergodicity}
	\lim_{q\rightarrow \infty}\sum_{j=1}^\infty \sum_{i=1}^\infty\frac{\vol(V_{s_{i,j}\PP_j})L_{s_{i,j}\PP_j}(A_q+h)}{\vol(A_q)}=\sum_{j=1}^\infty\widetilde{v}_j\emph{freq}(\PP_j,I_j,\SS)<\infty,
	\end{equation*}
	uniformly in $h\in\R^d$, with $\PP_j$, $\{s_{ij}\}_{i\ge1}$ and $I_j$ as in \eqref{eq:I_j} and $\emph{freq}(\PP,I,\SS)$ as in \eqref{eq: frequencies}.
\end{prop}

For the proof of Proposition \ref{prop: equation condition for unique ergodicity} we establish the following two lemmas. Lemma \ref{lem: convergence to weighted frequencies} in particular is a key step in the proof, and the density of scales in which patches appear in incommensurable tilings   plays an important role, together with the existence of uniform patch frequencies as established in \S\ref{sec: patches}. 

\begin{lem}\label{lem: convergence to weighted frequencies}
	For every $j\ge1$ there exists a constant $\widetilde{v}_j$ so that
	\begin{equation*}
	\lim_{q\rightarrow \infty} \sum_{i=1}^\infty\frac{\vol(V_{s_{i,j}\PP_j})L_{s_{i,j}\PP_j}(A_q+h)}{\vol(A_q)}=\widetilde{v}_j\emph{freq}(\PP_j,I_j,\SS),
	\end{equation*}
	uniformly in $h\in\R^d$.
\end{lem}

\begin{proof}
	Fix $j$, and write $a:=\inf I_j$ and $b:=\sup I_j$. For every $x\in[a,b]$ denote $I(x):=[a,x]$. Define the functions  $f,g_{q,h}:[a,b]\rightarrow\R$ by
	\begin{equation*}
	\begin{split}
		f(x)&=\vol(V_{x\PP_j})\\
		g_{q,h}(x)&=\frac{L_{\PP_j,I(x)}(A_q+h)}{\vol(A_q)},	
	\end{split}
	\end{equation*}	
	where $f(a)=0$ if $\Pi_m(a\PP_j)\neq U_{m,\ell}$.
	By definition, $f$ is continuous and $g_{q,h}$ is monotone for every $q\in\N$ and $h\in\R^d$, and so the Riemann-Stieltjes integral 
	\begin{equation*}
	\int_{a}^bfdg_{q,h}
	\end{equation*}
	exists (see \cite[\S 7]{Apostol}). Fix $N\in\N$, and relabel the elements of  $\{s_{i,j}\}_{i=1}^N$ in an increasing order. Let $\{x_i\}_{i=1}^N$ be so that
	\begin{equation*}
		\begin{split}
			&a=x_0<x_1<\cdots<x_N=b\\
			&x_{i-1}\le s_{i,j}\le x_i \qquad \forall 1\le i\le N.	
		\end{split}
	\end{equation*}	
	By definition of the Riemann-Stieltjes integral and by our construction, we have 
	\begin{equation}\label{eq:connecting_RSintegral_with_the_required_sum}
	\begin{split}
	\int_{a}^bfdg_{q,h}&=\lim_{N\to\infty}\sum_{i=1}^N f(s_{i,j})\left(g_{q,h}(x_i)-g_{q,h}(x_{i-1})\right)\\
	&=\sum_{i=1}^\infty\frac{\vol(V_{s_{i,j}\PP_j})L_{s_{i,j}\PP_j}(A_q+h)}{\vol(A_q)},
	\end{split}
	\end{equation}
	where the second equality holds because for fixed $q$ and $h$ the sum on the right-hand side is finite.
		We remark that even though almost every term in the sum established in \eqref{eq:connecting_RSintegral_with_the_required_sum} is equal to $0$, as $q$ increases and $h$ varies the quantity $L_{s_{i,j}\PP_j}(A_q+h)$ gets positive values for every $i$. 
	
	By Theorems \ref{thm: uniform patch frequencies} and \ref{thm: zero uniform patch frequency for patches}, the function
	\begin{equation*}
	g_\infty(x):=\lim_{q\to\infty}g_{q,h}(x)=\text{freq} (\PP_j,I(x),\SS)
	\end{equation*}
	is well defined, and the convergence is uniform in $h$. Observe that $f$ is bounded and that $g_\infty$ is monotone. By the first mean value theorem for Riemann-Stieltjes integrals (see \cite[Theorem 7.30]{Apostol}) and Theorem \ref{thm: zero uniform patch frequency for patches}, there exists a positive constant $\inf\{f(x)\}\le \widetilde{v}_j\le \sup\{f(x)\}$ for which
	\begin{equation}\label{eq:connecting_the_RSintegral_and_freq}
	\int_{a}^bfdg_\infty=\widetilde{v}_j(g_\infty(b)-g_\infty(a))=\widetilde{v}_j\text{freq}(\PP_j,I_j,\SS).
	\end{equation}
	
	In view of \eqref{eq:connecting_RSintegral_with_the_required_sum} and \eqref{eq:connecting_the_RSintegral_and_freq}, to finish the proof we therefore must show that the order of the limit and the integration can be switched, that is 
	\begin{equation}\label{eq: change limit and integral}
	\lim_{q\rightarrow \infty}	\int_{a}^bfdg_{q,h}= \int_{a}^bfd\left(\lim_{q\rightarrow \infty}g_{q,h}\right)=	\int_{a}^bfdg_\infty.
	\end{equation}
	Indeed, let us fix $q\in\N$. Using integration by parts (see \cite[\S 7.5]{Apostol}) we obtain
	\begin{equation}\label{eq:integration_by_parts_1}
	\int_{a}^bfdg_{q,h}= f(b)g_{q,h}(b)-f(a)g_{q,h}(a)-\int_{a}^bg_{q,h}df.
	\end{equation}
	Note that $g_{q,h}$ is a monotone, piecewise-constant function, and that $f$ is continuous. Then the integral on the right-hand side of \eqref{eq:integration_by_parts_1} exists, and can be viewed as a Lebesgue integral. Taking limits as $q\to\infty$ in \eqref{eq:integration_by_parts_1} we obtain
	\begin{equation}\label{eq: limit of integration by parts}
	\lim_{q\rightarrow \infty}\int_{a}^bfdg_{q,h}= f(b)g_\infty(b)-f(a)g_\infty(a)-\lim_{q\rightarrow \infty}\int_{a}^bg_{q,h}df.
	\end{equation}

	Since the patches $s_{i,j}\PP_j$ are of bounded diameter $\rho$ (which is related to $m$), we have
	\begin{equation*}
	g_{q,h}\le C(\rho),
	\end{equation*}
	where $C(\rho)$ is a constant that depends only on the tiling $\SS$ and on $\rho$. Therefore, by the Lebesgue's dominated convergence theorem and another integration by parts 
	\begin{equation}\label{eq: dominated convergence}
	\lim_{q\rightarrow \infty}\int_{a}^bg_{q,h}df=\int_{a}^bg_\infty df=g_\infty(b)f(b)-g_\infty(a)f(a)-\int_{a}^bfdg_\infty.
	\end{equation}
	Combining \eqref{eq: limit of integration by parts} and \eqref{eq: dominated convergence} we arrive at \eqref{eq: change limit and integral}, thus finishing the proof.
\end{proof}

\begin{lem}\label{lem: sufficient condition for unique ergodicity}
	Assume that for every $\varepsilon>0$, there exists $q_0\in\N$ so that for every $q\ge q_0$ and every $h\in\R^d$
	\begin{equation*}
	\sum_{j=q_0}^\infty \sum_{i=1}^\infty\frac{\vol(V_{s_{i,j}\PP_j})L_{s_{i,j}\PP_j}(A_q+h)}{\vol(A_q)}<\varepsilon.
	\end{equation*}
	Then Proposition \ref{prop: equation condition for unique ergodicity} holds.
\end{lem}

\begin{proof}
	The proof is very similar to the proof of \cite[Corollary 3.3]{Lee-Solomyak}. The key step is to use Lemma \ref{lem: convergence to weighted frequencies} in order to apply Fatou's lemma and bound from below the limit inferior of the expression on the left-hand side of \eqref{eq: equation condition for unique ergodicity}. The limit superior is then bounded from above using the assumed inequality. 
\end{proof}

We now use Lemmas \ref{lem: convergence to weighted frequencies} and \ref{lem: sufficient condition for unique ergodicity} to establish Proposition \ref{prop: equation condition for unique ergodicity}, from which unique ergodicity of the dynamical system follows, thus proving Theorem \ref{thm: X is uniquely ergodic}.

\begin{proof}[Proof of Proposition \ref{prop: equation condition for unique ergodicity}]
	Fix $r>0$ and enumerate all patches in $\SS$ that have diameter less than $r$. A patch $\PP$ is called \textit{$k$-special} if it occurs as a sub-patch of a $k$-supertile, with $k$ minimal, and we denote $\text{sp}(\PP)=k$. As shown in the proof of Lemma \ref{lem: uniform patch frequency for supertiles}, up to dilation there are finitely many patches that can be supported on a $k$-supertile. Group all inflations of a given patch together. 
	By the above and by Lemma \ref{lem: sufficient condition for unique ergodicity}, it is enough to show that 
	\begin{equation}\label{eq:goal_of_8.2_in_terms_of_sp(P)}
	\sum_{\text{sp}(\PP)>k}\sum_{i=1}^\infty \frac{L_{s_i\PP}(A_q+h)}{\vol(A_q)}
	\end{equation}
	can be made arbitrarily small, where $(s_i)_{i\ge 1}$ is an enumeration of the set of possible scales in which $\PP$ can appear. Namely, we need to show that for every $\varepsilon>0$ there exist $k_1,q_1>0$ such that for every $k\ge k_1$ and $q\ge q_1$ the quantity in \eqref{eq:goal_of_8.2_in_terms_of_sp(P)} is less than $\varepsilon$. 
	
	Indeed, since the support of every $k$-special patch intersects the boundary of some $k$-supertile, we have
	\begin{equation}\label{eq: first inequality in proof of unique ergodicity}
	\sum_{\text{sp}(\PP)>k}\sum_{i=1}^\infty L_{s_i\PP}(A_q+h)\le\#\left\{\PP\,:\,\text{diam}(\PP)\le r, \supp (\PP)\subset \bigcup_{T\in\mathcal{A}}\left(\partial(e^{k s} \supp T)\right)^{+r}\right\}
	\end{equation}
	where the union is taken over the set $\mathcal{A}$ of all tiles $T\in\SS$ such that $e^{k s}\supp(T)\subset (A_q+h)^{+r}$. The rest of the proof is now identical to that of \cite[Theorem 4.14]{Lee-Solomyak}. There exists a constant $C_r$ such that the right-hand side of \eqref{eq: first inequality in proof of unique ergodicity} is bounded by 
	\begin{equation}\label{eq: second inequality in proof of unique ergodicity}
	C_r \sum_{T\in\mathcal{A}} \vol \left(\partial(e^{k s} \supp T)\right)^{+r}.
	\end{equation}
	Given $\delta>0$, since $\left(e^{k s}\supp T\right)_{k\ge1}$ is van Hove, \eqref{eq: second inequality in proof of unique ergodicity} can be  bounded by 
	\begin{equation*}
	 C_r\delta\vol\left((A_q+h)^{+r}\right)=C_r\delta\vol\left((A_q)^{+r}\right)
	\end{equation*}
	for sufficiently large $k$. The sequence $\left(A_q\right)_{q\ge1}$ is also van Hove, and so
	\begin{equation*}
	\vol\left((A_q)^{+r}\right)\le(1+\delta)\vol A_q
	\end{equation*}
	for sufficiently large $q$. Therefore, for any sufficiently large $k$, for sufficiently large $q$
	\begin{equation*}
	\sum_{\text{sp}(\PP)>k}\sum_{i=1}^\infty \frac{L_{s_i\PP}(A_q+h)}{\vol(A_q)}\le C_r\delta(1+\delta).
	\end{equation*}
	Since the right-hand side is arbitrarily small, the proposition follows.
\end{proof}

\begin{remark}
		Unique ergodicity for certain fusion tilings was established in \cite{Frank-Sadun (ILC fusion)}, and in particular for their one dimensional construction, which as mentioned in Example \ref{ex: Kakutani 1/3} corresponds to the $\frac{1}{3}$-Kakutani tiling. We note that in contrast to their approach, the proof described here does not require reconizability. 
		
		We believe that the arguments in \S\ref{sec: patches} and \S\ref{sec: unique ergodicity} can be extended to the case of schemes with ``incommensurability of orientations'', in the sense of the pinwheel tilings. Presumably, an additional layer of isometries would be added to the information carried by the associated graph, and in the definition of patch frequencies, patches would be counted together with isometric copies within a neighborhood of the identity in the associated isometry group. More on this will appear in future work.
\end{remark}

\newpage

\section* {Appendix 1. Patches of multiscale substitution tilings}

\begin{figure}[ht!]
	\includegraphics[scale=0.71]{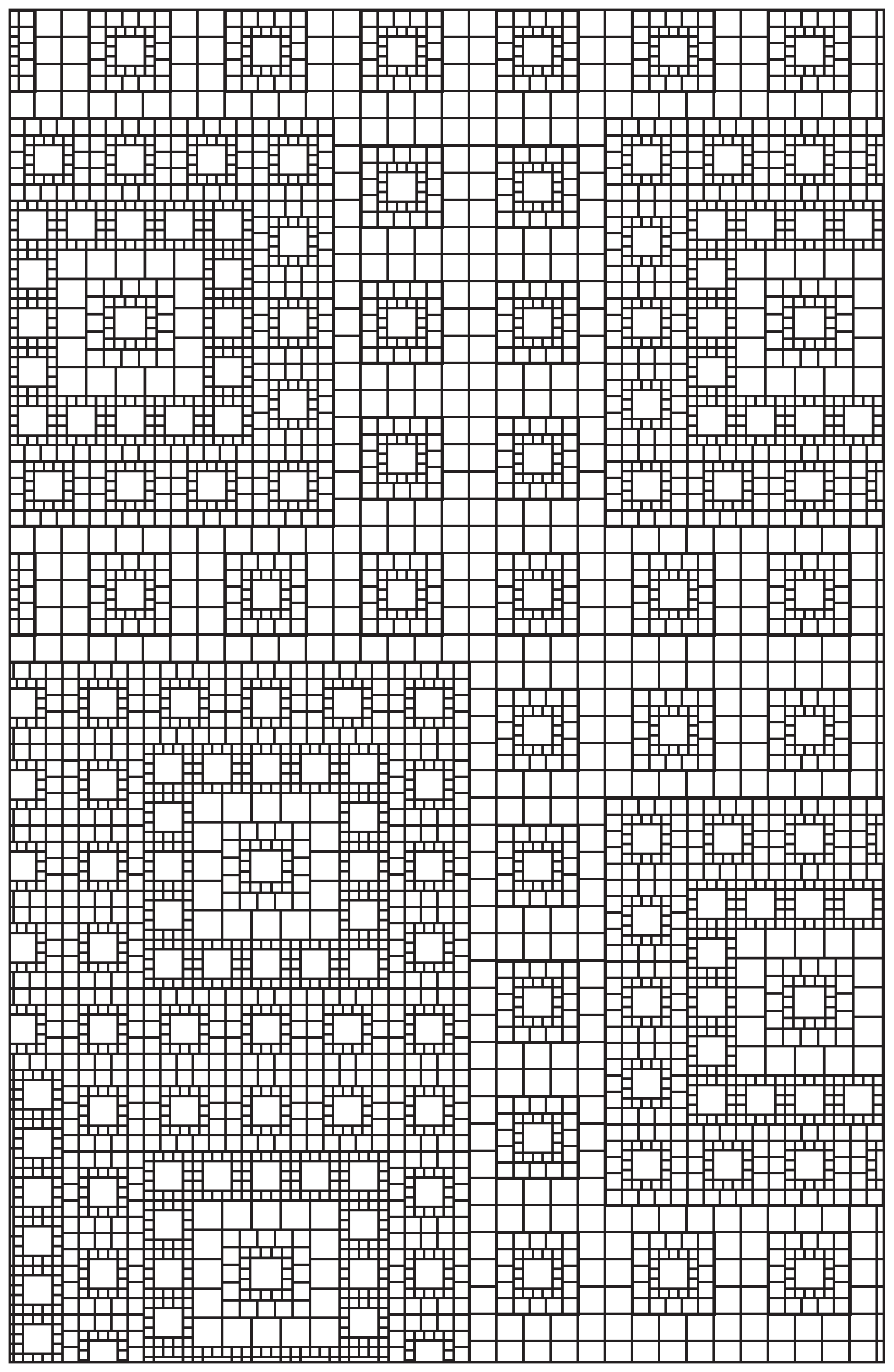}\caption{A fragment of an incommensurable multiscale substitution tiling of $\R^2$, generated by the square scheme illustrated in Figure \ref{fig: square scheme}.}	\label{fig:square full patch}
\end{figure}  

\newpage

\begin{figure}[ht!]
	\includegraphics[scale=0.71]{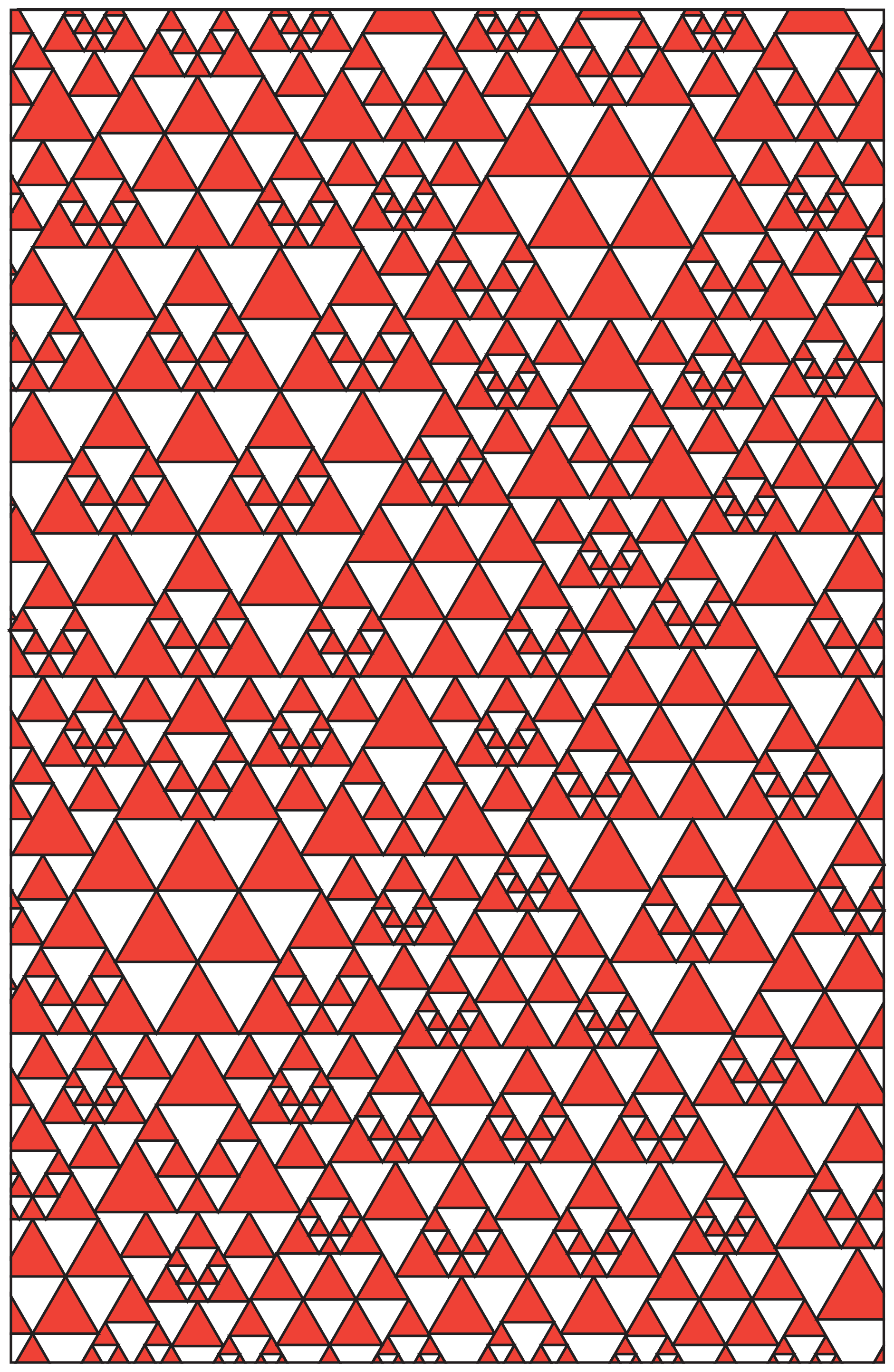}\caption{A fragment of an incommensurable multiscale substitution tiling of $\R^2$, generated by the triangles scheme illustrated in Figure \ref{fig:triangle substitution rule}.}	\label{fig:triangle patch}
\end{figure}

\end{document}